\documentclass[a4paper,10pt]{article}

\usepackage{dsfont}
\usepackage{amssymb}
\usepackage{latexsym}
\usepackage{amsmath}
\usepackage{xcolor}
\usepackage{comment}
\usepackage{amsthm}
\usepackage[dvips]{graphicx}
\usepackage{layout} 
\usepackage{ulem}
\usepackage{enumerate}
\usepackage{bm}
\usepackage{bbm} 
\usepackage[bbgreekl]{mathbbol} 
\usepackage{authblk}
\usepackage{setspace} 
\newif\ifrs
\rstrue
\ifrs \usepackage{mathrsfs} \fi  
\newif\ifcol
\coltrue 

\newtheorem{theorem}{Theorem}[section]

\newtheorem{lemma}[theorem]{Lemma}
\newtheorem{definition}[theorem]{Definition}

\newtheorem{remark}[theorem]{Remark}
\newtheorem{example}[theorem]{Example}
\numberwithin{equation}{section}
\newtheorem{rem}[theorem]{Remark}
\newtheorem{theorem*}{Theorem}
\newtheorem{ass*}[theorem*]{Assumption}
\newtheorem{note*}[theorem*]{Note}
\newtheorem{lemma*}[theorem*]{Lemma}
\newtheorem{definition*}[theorem*]{Definition}
\newtheorem{proposition*}[theorem*]{Proposition}
\newtheorem{corollary*}[theorem*]{Corollary}
\newtheorem{remark*}[theorem*]{Remark}
\newtheorem{example*}[theorem*]{Example}
\numberwithin{equation}{section}
\newif\ifcol
\colfalse
\ifcol
\newcommand{\colorr}{\color[rgb]{0.8,0,0}}
\newcommand{\colorg}{\color[rgb]{0,0.5,0}}

\newcommand{\colorn}{\color[rgb]{1,1,1}}

\else

\newcommand{\colorr}{\color{black}}
\newcommand{\colorg}{\color{black}}

\newcommand{\colorn}{\color{black}}

\fi
\newif\ifcol
\colfalse
\ifcol

\else
\fi
\newif\ifcol
\colfalse
\ifcol

\else
\fi
\newif\ifcol
\colfalse
\ifcol

\else
\fi
\newif\ifcol
\colfalse
\ifcol

\else
\fi
\newif\ifcol
\colfalse
\ifcol

\else
\fi
\excludecomment{en-text}
\includecomment{jp-text}
\includecomment{comment}

\def\ol{\overline}
\def\wt{\widetilde}
\def\wh{\widehat}

\def\bd{\begin{description}}
\def\ed{\end{description}}

\def\D2{\bbD_{2,\infty-}}

\def\tj{{t_j}}

\def\D{{\bf D}}

\def\F{{\bf F}}

\def\cala{{\cal A}}
\def\calb{{\cal B}}

\def\cald{{\cal D}}
\def\cale{{\cal E}}
\def\calf{{\cal F}}
\def\calg{{\cal G}}
\def\calh{{\cal H}}

\def\calj{{\cal J}}
\def\calk{{\cal K}}

\def\calm{{\cal M}}
\def\caln{{\cal N}}

\def\cals{{\cal S}}
\def\calt{{\cal T}}

\def\calx{{\cal X}}
\def\caly{{\cal Y}}

\def\ds{\displaystyle}
\def\yeq{\>=\>}
\def\yleq{\>\leq\>}

\def\sfm{{\sf m}}

\def\sfd{{\sf d}}
\def\sfp{{\sf p}}
\def\sfq{{\sf q}}
\def\sfr{{\sf r}}

\def\half{\frac{1}{2}}

\def\down{\downarrow}

\def\y{\vspace*{3mm}\\}
\def\halflineskip{\vspace*{3mm}}
\def\nn{\nonumber}
\def\be{\begin{equation}}
\def\ee{\end{equation}}
\def\bea{\begin{eqnarray}}
\def\eea{\end{eqnarray}}
\def\beas{\begin{eqnarray*}}
\def\eeas{\end{eqnarray*}}
\def\bi{\begin{itemize}}
\def\ei{\end{itemize}}
\def\im{\item}
\def\bd{\begin{description}}
\def\ed{\end{description}}
\newcommand{\bbC}{{\mathbb C}}
\newcommand{\bbD}{{\mathbb D}}

\newcommand{\bbH}{{\mathbb H}}

\newcommand{\bbN}{{\mathbb N}}

\newcommand{\bbR}{{\mathbb R}}

\newcommand{\bbT}{{\mathbb T}}

\newcommand{\bbV}{{\mathbb V}}

\newcommand{\bbY}{{\mathbb Y}}
\newcommand{\bbZ}{{\mathbb Z}}

\newcommand{\sfa}{{\sf a}}
\newcommand{\sfb}{{\sf b}}

\def\sfr{{\sf r}}
\def\sfp{{\sf p}}
\def\sfd{{\sf d}}
\def\sfm{{\sf m}}
\begin{document}

\title{
	Quasi-maximum likelihood estimation and penalized estimation under non-standard conditions
\footnote{
		$2020 ~Mathmatics ~Subject~Classification $: Primary 62F12; Secondary 62E20. 
		
This work was in part supported by 
Japan Science and Technology Agency CREST  JPMJCR2115; 
Japan Society for the Promotion of Science Grants-in-Aid for Scientific Research 
No. 17H01702 (Scientific Research);  
and by a Cooperative Research Program of the Institute of Statistical Mathematics. 
}
}
\author[1,2]{Junichiro Yoshida}
\author[1,2]{Nakahiro Yoshida}
\affil[1]{Graduate School of Mathematical Sciences, University of Tokyo
\footnote{Graduate School of Mathematical Sciences, University of Tokyo: 3-8-1 Komaba, Meguro-ku, Tokyo 153-8914, Japan. e-mail: nakahiro@ms.u-tokyo.ac.jp}
}
\affil[2]{CREST, Japan Science and Technology Agency
}
\maketitle
\ \\
{\it Summary} 
{
The purpose of this article is to develop a general parametric estimation theory  that allows the  derivation of  the limit distribution of estimators in non-regular models where the true parameter value may lie on the boundary of the parameter space or where even identifiability  fails.
For that, we propose a more general local approximation of the parameter space (at the true value) than previous studies. This estimation theory is comprehensive in that it can handle  penalized estimation  as well as quasi-maximum likelihood estimation under such non-regular models.  Besides, our results can apply to the so-called non-ergodic statistics, where the Fisher information is random in the limit,  including the regular experiment that is locally asymptotically mixed normal. In penalized estimation, depending on the boundary constraint, even the Bridge  estimator  with $q<1$ does not necessarily give selection consistency. Therefore, some sufficient condition for selection consistency is described, precisely evaluating the balance between the boundary constraint and the form of the penalty.

Examples handled in the paper are: (i) ML estimation of the generalized inverse Gaussian distribution, (ii) quasi-ML estimation of  the diffusion parameter in a non-ergodic It\^o process whose parameter space consists of  positive semi-definite symmetric matrices, while  the drift parameter is treated as nuisance and (iii) penalized ML estimation of variance components of random effects in  linear mixed models. 

}
\section{Introduction}
The purpose of this article is to develop a general parametric estimation theory  that allows the  derivation of  the limit distribution of estimators in non-regular models where the true parameter value may lie on the boundary of the parameter space or where even identifiability  fails.  (Non-identifiable cases are dealt with specifically  in the subsequent paper  \cite{yoshida2022penalized}.) We generalize the local asymptotic theory, established by  Ibragimov and Khas'minskii  \cite{ibragimov1972asymptotic, IbragimovHascprimeminskiui1981},  which uses the convergence of the random field formed by the likelihood ratios.

Let us recall their result briefly. We denote by $\cale^\epsilon =\{\calx, \cala, P_\theta^\epsilon, \theta \in\Theta\}$  a sequence of statistical experiments with $\epsilon \in (0, 1]$ for a parameter space $\Theta \subset \bbR^m$. Let $\varphi(\epsilon)$ be a positive normalizing factor tending to zero as $\epsilon \downarrow 0$. For a $\theta^* \in {\rm Int}(\Theta)$, define a random field $Z_\epsilon$ by 
\bea\label{Z_epsilon}
Z_\epsilon (u) = \frac{dP^\epsilon_{\theta^*+\varphi(\epsilon)u}}{dP^\epsilon_{\theta^*}}(X^\epsilon)\qquad(u \in \bbR^m).
\eea
Suppose that $Z_\epsilon$ can be extended on a $\hat C(\bbR^m)$-valued random variable\footnote{$\hat C(\bbR^m)$ is the space of  continuous functions on $\bbR^m$ that tends to zero at the infinity} and that the extended $Z_\epsilon$ converges to $Z$ in $\hat{C}(\bbR^m)$, where $Z$ is some $\hat C (\bbR^m)$-valued random variable. Then we have
\bea\label{Ibragimov}
\hat u_\epsilon =\underset{ u}{\rm argmax\,} Z_\epsilon(u) \overset{d}{\to}   \hat u = \underset{ u\in \bbR^m}{\rm argmax\,} Z(u).
\eea
 When the experiment $\cale^\epsilon$ is locally asymptotically normal (LAN), $Z(u)$ takes the form of $Z(u)=\exp\big(\Delta(\theta^*)-2^{-1}I(\theta^*)[u^{\otimes2}]\big)$, where $\Delta(\theta^*) \sim N_m\big(0, I(\theta^*)\big)$ and $I(\theta^*)$ is the Fisher information matrix, and we obviously obtain  $\hat u= I(\theta^*)^{-1}\Delta(\theta^*) \sim N_m\big(0, I(\theta^*)^{-1}\big)$. In this paper, we extend this theory to a more general one  
 that  can be  applied even when $\theta^*$ may lie on the boundary of $\Theta$. Moreover, in our extended theory,   $Z_\epsilon$  is not necessarily defined as (\ref{Z_epsilon}) and not necessarily asymptotically quadratic. Therefore, as explained below,  penalized estimation can be handled as well as quasi-maximum likelihood estimation. 
 Besides, our results can apply to the so-called non-ergodic statistics, where the Fisher information is random in the limit,  including the regular experiment that is locally asymptotically mixed normal (LAMN).

When the true parameter value $\theta^*$  lies on the boundary of $\Theta$,  a cone set locally approximating  $\Theta$ at $\theta^*$ is usually used. This set was 
introduced by Chernoff  \cite{chernoff1954distribution}, and with it, Self and Liang \cite{self1987asymptotic} derived the limit distribution of  the maximum likelihood estimation (MLE)   when $\theta^*$ is on the boundary.  Also, using a cone set that is a generalization of Chernoff's cone set, Andrews \cite{andrews1999estimation} derived the asymptotic distribution of the quasi-maximum likelihood estimator (QMLE)  when $\theta^* $ is on the boundary. These cone sets are denoted by $\Lambda$.
However, if  the rate of  convergence of a given estimator  to $\theta^*$ is different for each component, then  $\Theta$  may not be locally approximated by any cone set. In this paper, the asymptotic behavior of the estimator  is derived even in such complex cases  by generalizing the local approximation method for the parameter space.  Instead of $\Lambda$, we denote by $U$ the general set locally approximating $\Theta$ at $\theta^*$, which is not necessarily a cone. (Examples of $U$ are listed in Section 2.3.)  Then our theorem (Theorem \ref{thm1}) shows that (\ref{Ibragimov}) is generalized as
\bea\label{Ibragimovgeneral}
\hat u_\epsilon =\underset{ u}{\rm argmax\,} Z_\epsilon(u) \overset{d}{\to}   \hat u = \underset{ u\in  U}{\rm argmax\,} Z(u).
\eea
Besides, it is  shown that  our approximation method with $U$ is a generalization of the previous ones with  $\Lambda$ such as  Chernoff  \cite{chernoff1954distribution} and  Andrews \cite{andrews1999estimation} (in Section 2.4). As an example of a direct application of Theorem \ref{thm1}, the maximum likelihood estimation of the generalized inverse Gaussian distribution is treated in Section 2.5.

As the first main application of Theorem \ref{thm1}, we  derive the limit distribution of  the QMLE when identifiability is valid but the true parameter value possibly lies on the boundary  (Theorem \ref{thmB}).  An example   is quasi-maximum likelihood estimation of  the diffusion parameter in a non-ergodic It\^o process whose parameter space consists of  semi-positive definite symmetric matrices, while  the drift parameter is treated as nuisance.  
For this example,  a nuisance  parameter $\tau \in \calt$ is considered along with $\theta$ although the effect of $\tau$ is explicitly assumed to asymptotically disappear. 

As the second application,  we also derive the limit distribution of  the penalized quasi-maximum likelihood estimator (PQMLE) when identifiablity is valid but the true parameter value may be on the boundary (Theorem \ref{thmpenalized}). 
 One of the most simple penalty terms  is the Bridge (Frank and Friedman 1993)  expressed as 
\beas
p_\lambda (\theta) \yeq 
\lambda \sum_{i=1}^{\sfp} |\theta_i|^q
 \qquad\big(\theta = (\theta_1, ..., \theta_\sfp) \in \Theta\big),
\eeas
where $q > 0$ is a constant, and $\lambda > 0$ is a tuning parameter. For $q \leq 1$, the estimator performs
variable selection. Especially, when $q=1$, the estimator is called the Lasso (Tibshirani
\cite{tibshirani1996regression}). 
Under regular conditions, the penalized maximum likelihood estimator (PMLE) has been studied before. In the Bridge case,  Knight and Fu \cite{fu2000asymptotics} derived  the limit distribution of the PMLE. Zou \cite{zou2006adaptive} proposed the adaptive Lasso and showed its oracle property. 

For the PQMLE under regular conditions, various results have been recently shown: see e.g.  De Gregorio and Iacus \cite{de2012adaptive} and Ga\"{i}ffas and Matulewicz \cite{gaiffas2019sparse}. Masuda and Shimizu \cite{masuda2017moment} and Kinoshita and Yoshida \cite{kinoshita2019penalized} derived the moment convergence of  PQMLE. 
  Both of those studies derived the polynomial type large deviation inequality (PLDI), an inequality given by Yoshida \cite{yoshida2011polynomial}  for the moment convergence of the estimator.  Also, using the PLDI,  Umezu et al.$\,$\cite{umezu2019aic}  derived an information criterion based on the definition of the AIC.

Under a non-regular condition that the true parameter may be on the boundary,  Wong et al.\,\cite{wong2016oracle} studied the  PMLE with the adaptive Lasso  and  derive its oracle property, assuming that the parameter space $\Theta$ can be simply expressed in the form of a direct product. 
 In this paper, under a boundary constraint, the PQMLE with many kinds of penalties including  the Bridge and the adaptive Lasso is studied, not assuming  that $\Theta$ can be decomposed as a direct sum. Then depending on the boundary constraint, even the Bridge  estimator  with $q<1$ does not necessarily show selection consistency. Therefore, some sufficient condition for selection consistency is described, precisely evaluating the balance between the boundary constraint and the form of the penalty in Section 4.3. As an example,  we discuss  the condition for the oracle property in penalized estimation of variance components of random effects in  linear mixed models, penalizing only the diagonal elements. Random effect selection in mixed models is an important application of penalized estimation with  a boundary constraint, and is   an area of extensive research (see M\"uller et al.\,\cite{muller2013model} for a review).
 {Bondell, Krishna and Ghosh \cite{bondell2010joint} and Ibrahim et al.\,\cite{ibrahim2011fixed} consider penalized estimation using  Cholesky parametrizations, while we treat variance components without re-parametrization.}

 { As the third application, in the subsequent paper \cite{yoshida2022penalized}, we deal with  non-identifiable models where   the true value of $\theta \in \Theta$ is not necessarily uniquely determined.  Then the usual estimators such as the QMLE are difficult to use for the purpose of parameter estimation since their limit cannot  be determined uniquely. 
 Therefore, we  stabilize the asymptotic behavior  by considering a suitable penalty term, and  handle the PQMLE.  Let us denote by $\Theta^*$ the set consisting of the true values of $\theta \in \Theta$.
 We add to quasi-log likelihood function a penalty term  whose minimizer  on $\Theta^*$ is  uniquely determined.  Denote the minimizer by $\theta^* \in \Theta^*$. Then the corresponding PQMLE converges to $\theta^*$ in probability, and we can derive its  limit distribution and selection consistency.  This holds even when another unknown  parameter $\tau \in \calt$ which can be a nuisanse  is considered along with $\theta$. 

 The article is organized as follows. In Section 2, the  general theory for (\ref{Ibragimovgeneral})  is described. The key theorem is stated in Section 2.2.  (Its proof is given in Section 5.)   An application to the quasi-maximum likelihood estimation is depicted in Section 3.  (Some proofs are given in Section \ref{section6}.)  Also, in Section  4,  the penalized quasi-maximum likelihood estimation is described, and sufficient conditions for selection consistency 
 are given.  }

\section{General theory}
\subsection{Settings}
Denote by $\Xi = \Theta \times \mathcal{T}$ the unknown parameter space, 
where $\Theta$ is a measurable subset of  ${ \mathbb{R}^\sfp}$, and $ \mathcal{T}$ is a measurable subset of $\mathbb{R}^{\sfq}$. 
In this section, we estimate the true value {(or one of the true values)} $\theta^*$ of the unknown parameter $\theta\in\Theta$, while $\tau\in \calt$ is treated as a nuisance parameter.  
Given a probability space $(\Omega, \mathcal{F}, P)$ specifying the distribution of the data, the statistical inference will be carried out based on a continuous random field $\mathbb{H}_T : \Omega \times \Xi \rightarrow \mathbb{R}$ for 
$T\in\bbT\subset\bbR$ satisfying $\sup\bbT=\infty$,  where a continuous random field means that for each $\omega\in \Omega$, $\bbH_T(\omega)$ is continuous on $\Xi$.
Examples of $\bbH_T$ are (quasi-)log likelihood functions and penalized (quasi-)log likelihood functions.

For each $T\in\bbT$, take  an arbitrary $\calt$-valued random variable $\hat\tau_T$, and suppose that   we can take  a $\Theta$-valued random variable $\hat\theta_T$  that  { asymptotically maximizes}  $\bbH_T(\theta, \hat\tau_T)$ on $\Theta$. A common example of $(\hat\theta_T, \hat\tau_T)$ is the joint maximizer. Under some conditions in the next subsection, we derive the asymptotic behavior of the estimator $\hat\theta_T$. Note that those conditions depend neither on   $\hat\tau_T$ nor  on the way $\hat\theta_T$ is taken, and that the result shows that the asymptotic behavior of $\hat{\theta}_T$ is the same for different sequences $\hat{\tau}_T$. 


\subsection{A limit theorem}


Define $U_T$ by 
\begin{eqnarray}\label{U_T}
U_T = \{u\in \mathbb{R}^\sfp ; \theta^* + a_T u \in \Theta\} 
\end{eqnarray}
for a deterministic sequence $a_T$ in $GL(\sfp)$ with $\lim_{T\to\infty}\|a_T\| \rightarrow 0$, where  for any real matrix $A$,   denote by  $A^\prime$ the transpose of $A$,  and  $\|A\|$ denotes $\big\{{\rm Tr}(AA^\prime)\}^{\frac{1}{2}}$.
We mimic the local asymptotic theory to define the random field $\mathbb{Z}_T : \Omega \times U_T \times\calt  \rightarrow \mathbb{R} $ by 
\begin{eqnarray}\label{Z_T}
	\mathbb{Z}_T(u, \tau) &=& 
	\exp\bigg(\mathbb{H}_T(\theta^* + a_Tu, \tau) - \mathbb{H}_T(\theta^*, \tau)\bigg)
	\qquad
	(u\in U_T).
\end{eqnarray} 
 Let $C(\bbR^\sfp)$ denote the set  of all  continuous functions  defined on $\bbR^\sfp$. We give it the metric topology induced by a metric $d_\infty$ defined as
\beas
d_\infty(f, g)=\sum_{n=1}^\infty2^{-n}\bigg(1\wedge\max_{|x|\leq n}|f(x)-g(x)|\bigg)\qquad\big(f,g \in C(\bbR^\sfp)\big).
\eeas
For each $T\in\calt$, let $\mathbb{V}_T$ be a $C(\mathbb{R}^\sfp)$-valued random variable.
Also, let $\mathbb{Z}$ be a $C(\mathbb{R}^\sfp)$-valued random variable which will be considered as the limit of $\mathbb{Z}_T$. Let $\calg$ be a sub-$\sigma$-algebra of $\calf$.

 For a topological space $S$,  a sequence of  $S$-valued random variables $Y_T$ $(T \in \bbT)$ and  a 
 $S$-valued random variable $Y$, we say that $Y_T$ converges stably with limit $Y$ and write as $Y_T {\to} ^{d_s(\calg)}Y$ if and only if for any  bounded continuous function  $f$ defined on $S$ and for any bounded $\calg$-measurable random variable $Z$,
 \beas
 E\big[f(Y_T)Z\big] \to  E\big[f(Y)Z\big] \qquad(T \to \infty).
 \eeas 
Also, for $A\subset \mathbb{R}^\sfp$, $\delta >0$, and $R>0$, we define  $A^{\delta}$ and $A(R)$ as
\begin{eqnarray*}  
	A^{\delta} = \big\{ x\in\mathbb{R}^\sfp ; \inf_{a \in A}|x-a| < \delta \big\} 
	&\text{and}&
	\ A(R)=A\cap{\overline{B_R}}, 
\end{eqnarray*}
respectively, where $B_R = \{x \in \mathbb{R}^\sfp ; |x|< R\}$ and  $\overline {B}$ denotes the closure of $B$ for a subset $B$.
Then we  define $U \subset \mathbb{R}^\sfp$ by \begin{eqnarray}\label{U}
 U = \bigcap_{\delta >0} \bigcup_{N=1}^{\infty} \bigcap_{T\geq N} 	{U_T}^{\delta} .\end{eqnarray}
It will be shown in Lemma \ref{lem1} that $U$ is closed. We consider the following conditions.
\bd
\im[{\bf[A1]}]
$\ds \varlimsup_{R \rightarrow \infty} \varlimsup_{T\rightarrow \infty}
P\bigg[\sup_{U_T \times \mathcal{T},  |u| \geq R} \mathbb{Z}_T(u, \tau)  \geq 1\bigg] \yeq0$.
\im[{\bf[A2]}] For every $R>0$, as $T\to\infty$,
\begin{eqnarray}
	\sup_{U_T(R) \times \mathcal{T}}|\mathbb{Z}_T(u, \tau) - \mathbb{V}_T(u)| \overset{P}{\rightarrow} 0, \nonumber\\ 
\mathbb{V}_{T} \overset{d_s(\calg)}{\rightarrow} \mathbb{Z}
\qquad \text{in } C(\overline{B_R}) .\label{2110240341}
\end{eqnarray}
\ed
More precisely, the convergence (\ref{2110240341}) means 
$\bbV_T|_{\sf C}\overset{d_s(\calg)}{\to}\bbZ|_{\sf C}$ for ${\sf C}=C(\overline{B_R} )$. \halflineskip

\bd
\im[{\bf[A3]}] $\ds U \supset \bigcap_{N=1}^{\infty} \overline{\bigcup_{T\geq N} U_T} $.

\im[{\bf[A4]}] 
There exists a $U$-valued random variable $\hat{u}$ such that with probability $1$,  
\begin{eqnarray*}{\mathbb{Z}(\hat{u})} = \sup_{U }{\mathbb{Z}(u)}
\end{eqnarray*}
and such that with probability 1, for all $u \in U$ with $u \neq \hat{u}$,
\begin{eqnarray*}  {\mathbb{Z}(u)} < {\mathbb{Z}(\hat{u})}.
\end{eqnarray*}

\im[{\bf[A5]}] There exist some $ T_0 \in \bbT$ and a sequence of  $U$-valued random variables $\{\hat{v}_T\}_{T\geq T_0, T\in \bbT}$ such that with probability $1$,  
\begin{eqnarray*}{\mathbb{V}_T(\hat{v}_T)} = \sup_{U}{\mathbb{V}_T(u)}
\end{eqnarray*}
and such that $\{\hat v_T\}_{T\geq T_0, T\in\bbT}$ is tight.

\ed
Define $\hat{u}_T$ by $\hat{u}_T= a_T^{-1}(\hat{\theta}_T - \theta^*)$. The following theorem constitutes the general result underlying Sections 3 and 4 and the subsequent paper \cite{yoshida2022penalized}.

\begin{theorem}\label{thm1} 
	 Under {\rm [{\bf A1}]-[{\bf A4}]},
		\begin{eqnarray}\label{hatu}
		\hat{u}_T \overset{d_s(\calg)}{\rightarrow} \hat{u}.
	\end{eqnarray} 
Moreover, if {\rm [{\bf A5}]} also holds, then
	\begin{eqnarray}
		\hat{u}_T -\hat{v}_T=o_P(1).\label{hatu-hatv}
	\end{eqnarray} 
\end{theorem} The proof of Theorem \ref{thm1} is given in Section 5.



\begin{rem}\label{remark1}
	{\rm
		\bd
		\im(i) If $\calt$ is closed, [{\bf A1}] implies  the following condition:
	\bd
	\im[${\bf[A1]}^\flat$]
	$\ds \varlimsup_{R \rightarrow \infty} \varlimsup_{T\rightarrow \infty}
	P\bigg[\exists\tau\in\calt~~s.t.~\sup_{U_T,  |u| \geq R} \mathbb{Z}_T(u, \tau)  = \sup_{U_T} \mathbb{Z}_T(u, \tau)\bigg] \yeq0$.
	\ed
Moreover, ${\bf[A1]}^\flat$ implies  the following condition:
	\bd
	\im[${\bf[A1]}^{\flat\flat}$]
	$\ds \varlimsup_{R \rightarrow \infty} \varlimsup_{T\rightarrow \infty}
	P\bigg[a_T^{-1}(\hat\theta_T-\theta^*)\geq R \bigg]=0$.
	\ed
	Theorem \ref{thm1} holds even if we substitute [{\bf A1}]  with  ${\bf[A1]}^{\flat\flat}$. (See the proof.)  However, the condition ${\bf[A1]}^{\flat\flat}$ depends on $\hat\tau_T$ and how $\hat\theta_T$ is taken.
	
	\im(ii) {From Measurable Selection Theorem,} if $[{\bf A1}]$ holds and there exists a neighborhood $\mathcal{N}$ of $\theta^*$ 
	in $\mathbb{R}^\sfp$  such that $\mathcal{N} \cap \Theta$ is closed in $\mathbb{R}^\sfp$, then   we can take $\hat\theta_T$  which satisfies the assumption that  $\hat\theta_T$ asymptotically maximizes $\bbH_T(\cdot, \hat\tau_T)$ on $\Theta$.
	
	\im(iii)   In [{\bf A3}], the reverse inclusion always holds. In fact,
	\beas
U =\bigcap_{\delta >0} \bigcup_{N=1}^{\infty} \bigcap_{T\geq N} 	{U_T}^{\delta} \subset
\bigcap_{\delta>0}  \bigcap_{N=1}^\infty \bigcup_{T\geq N}	{U_T}^{\delta} =
 \bigcap_{\delta>0}  \bigcap_{N=1}^\infty \bigg(\bigcup_{T\geq N}	{U_T}\bigg)^\delta = \bigcap_{N=1}^{\infty} \overline{\bigcup_{T\geq N} U_T}. 
\eeas

\ed}
\end{rem}

\subsection{Sufficient conditions for [{\bf A3}] and explicit expression for $U$}
Let $F \subset \bbR^\sfp$ be a non-empty closed subset. Let $f : F \rightarrow \bbR^{\sf{n}}$  be a continuous map, and let $g_T : F \rightarrow \bbR^{\sf{n}}$ be a continuous map depending on $T \in \bbT$.
We show that if the sets $U_T$ defined  by the  equation (\ref {U_T}) can be explicitly expressed using $f$ and $g_{T}$ satisfying suitable  assumptions, then the set $U$ defined by the equation (\ref{U}) can also be explicitly expressed, and furthermore the condition [{\bf A3}] holds. More precisely, we consider the following assumptions. Note that in the following, for $v=(v_1, ..., v_{\sf{n}})\in \bbR^{\sf{n}}$, $v \geq 0$ and $v>0$ mean that  $v_i \geq 0$ $(i=1, .., \sf{n})$ and $v_i>0$ $(i=1, ..., \sf{n})$, respectively. 
\bd
\im[{$\bf [U1]$}] For any $R>0$, there exists some $T_0=T_0(R) \in \bbT$ such that for any $T \in \bbT$ with $T\geq T_0$, 
\beas
U_T\cap B_R\yeq \{u \in F ; f(u)+g_T(u) \geq 0\}\cap B_R.
\eeas
\ed
Define $\calk_+$ and $\calk_c$ as two partitions of $\{1, ..., \sf{n}\}$. Let $f^+$, $f^c$, $g^+_{T}$ and $g^c_{T}$  denote $(f_k)_{k \in \calk_+}$, $(f_c)_{k \in \calk_c}$, $(g_k)_{k \in \calk_+}$ and $(g_k)_{k \in \calk_c}$, respectively, where  for each $k=1, ..., \sf{n}$, $f_k$ and $g_{k, T}$ denote the $k$-th component of  $f$ and $g_T$, respectively. 
\bd
\im[{$\bf [U2]$}] For any $R>0$, 
\beas \sup_{u \in F \cap B_R}|g_T(u)|\rightarrow 0\qquad (T\rightarrow \infty).
\eeas
Also,  for any $R>0$, there exists some $T_1=T_1(R) \in \bbT$ such that for any $T \in \bbT$ with $T\geq T_1$, \beas
g^+_{T}(u) \geq 0 \qquad(u \in F \cap B_R).
\eeas
\im[{$\bf [U3]$}]
\beas \overline{\big\{u \in F ; f^{+}(u)\geq 0,~ f^{c}(u)>0\big\}}\supset \big\{u \in F ; f(u) \geq 0 \big\}.
\eeas
\ed
\begin{theorem}\label{thmU}
 Assume $[{\bf U1}]$-$[{\bf U3}]$. Then  the condition $[{\bf A3}]$ holds.  Moreover, 
	\beas U\yeq \{u \in F ; f(u)\geq 0\}.
	\eeas
	\end{theorem}
\begin{proof}
	Assume $[{\bf U1}]$-$[{\bf U3}]$. 
	For any $R>0$, we have
	\beas U &=&\bigcap_{\delta >0} \bigcup_{N=1}^{\infty} \bigcap_{T\geq N}{U_T}^{\delta}\\
	&\supset& \bigcup_{N=1}^{\infty} \bigcap_{T\geq N}{U_T}\cap B_R\\
	&=& \bigcup_{N=1}^{\infty} \bigcap_{T\geq N}\big\{f(u)+g_T(u) \geq 0\big\}\cap B_R \qquad\big(\because [{\bf U1}]\big)\\
	&\supset& \bigcup_{\substack{\epsilon=(\epsilon_0 , ..., \epsilon_0)\in \bbR^{\sf{n}}\\ \epsilon_0 >0}}\bigcup_{N=1}^{\infty} \bigcap_{T\geq N}\big\{f^+(u)\geq 0, f^c(u) -\epsilon\geq 0\big\}\cap B_R \qquad\big(\because [{\bf U2}]\big)\\
	&=&\big\{f^+(u)\geq 0, f^c(u)> 0\big\}\cap B_R.
	\eeas 
	Thus, $U \supset \{f^+(u)\geq 0, f^c(u)> 0\}$. Since $U$ is closed from Lemma \ref{lem1} and since  [{\bf U3}] holds, we have
	\bea\label{202206011818}  U \supset \overline{\big\{f^+(u)\geq 0, f^c(u)> 0\big\}}\supset \{f(u) \geq 0\}.	\eea
	Similarly, for any $R>0$, \beas U\cap B_R&\subset& \bigcap_{N=1}^{\infty} \overline{\bigcup_{T\geq N} U_T}\cap B_R \\
	&\subset& \bigcap_{N=1}^{\infty} \overline{\bigcup_{T\geq N} U_T\cap B_{2R}} \\
	&=& \bigcap_{N=1}^{\infty} \overline{\bigcup_{T\geq N} \{f(u)+g_T(u)\geq 0\}\cap B_{2R}}  \qquad \big(\because [{\bf U1}]\big)\\
	&\subset& \bigcap_{\substack{\epsilon=(\epsilon_0 , ..., \epsilon_0)\in \bbR^{\sf{n}}\\ \epsilon_0 >0}}\bigcap_{N=1}^{\infty} \overline{\bigcup_{T\geq N} \{f(u)+\epsilon\geq 0\}\cap B_{2R}}  \qquad \big(\because [{\bf U2}]\big)\\
	&\subset& \{f(u)\geq 0\}.
	\eeas
	Thus, \bea\label{202206011819} U\subset \bigcap_{N=1}^{\infty} \overline{\bigcup_{T\geq N} U_T} \subset \{f(u)\geq 0\}.
 	\eea
 	From (\ref{202206011818}) and (\ref{202206011819}), [{\bf A3}] holds, and $U=\{f(u)\geq 0\}$.
	\end{proof}

 Several examples of explicit expressions of $U$ are given below.
\begin{example}\label{Exexplicit1}{\rm Take $\Theta$ as 
	\beas
	\Theta=\prod_{i=1}^\sfa [0, a_i]  \times \prod_{j=1}^{\sfb} [-b_j, 0]  \times \prod_{k=1}^{\sf{c}} [-c_k, d_k], \eeas
	where $a_i, b_j, c_k, d_k$ $(i=1, ..., \sfa, j=1, ..., \sfb, k=1, ..., \sf{c})$ are all positive numbers.
	Let $\sfp=\sfa+\sfb+\sf{c}$, and take $a_T \in GL(\sfp)$ as  diagonal matrices with $\|a_T\| \to 0$. Decompose  the true value $\theta^*$ of $\theta \in \Theta$ as 
	\beas
	\theta^* \yeq (\alpha_1^*, ..., \alpha_\sfa^*, \beta_1^*, ..., \beta_\sfb^*, \gamma^*_1, ..., \gamma^*_{\sf{c}}),
	\eeas 
	where $0 \leq \alpha^*_i < a_i, -b_j< \beta_j^* \leq 0, -c_k < \gamma_k^*<d_k$ $(i=1, ..., \sfa, j=1, ..., \sfb, k=1, ..., \sf{c})$. Then [{\bf A3}] holds, and
	\bea\label{Exexplicit1U}
	U &=&  \prod_{i=1}^\sfa A_i  \times \prod_{j=1}^{\sfb} B_j  \times \prod_{k=1}^{\sf{c}} \bbR,
	\eea
	where for each $i=1, ..., \sfa$ and each $j=1, ..., \sfb$,
	\beas
	A_i =\begin{cases} [0, \infty) & (\alpha_i^*=0) \\ \bbR &(\alpha_i >0)\end{cases}, \qquad B_j =\begin{cases} (-\infty, 0]& (\beta_j^*=0) \\ \bbR &(\beta_j <0)\end{cases}.
	\eeas
	In fact, from the  definition (\ref{U_T}) of $U_T$, for any $R>0$, there exists some $T_0$ such that 
	\beas
	U_T \cap B_R \yeq  \bigg( \prod_{i=1}^\sfa A_i  \times \prod_{j=1}^{\sfb} B_j  \times \prod_{k=1}^{\sf{c}} \bbR \bigg)\cap B_R\qquad(T \geq T_0).
	\eeas 
 Taking the set on the right-hand side of  (\ref{Exexplicit1U}) as $F$ and defining the  functions $f$ and $g_T$ on $F$ as $f=g_T=0$, [{\bf U1}]-[{\bf U3}] obviously hold. (In this case, $\calk_+ = \{1\}$ and $\calk_c= \phi$.)
		Therefore, from Theorem \ref{thmU}, (\ref{Exexplicit1U}) holds. 
}
	\end{example}

\begin{example}[{The space of positive semi-definite matrices}]\label{Expositivesemidefinite}{\rm {The parameter space treated in this example appears in Section \ref{Ito}.}
		Let $\sfm$ be a positive integer, and take $\sfp$ as $\sfp=\frac{\sfm(\sfm+1)}{2}$. Denote by $\cals^\sfm$ and $\cals_+^\sfm$ the set of  all $\sfm$-dimensional real symmetric matrices and the set of all  positive-semi-definite matrices included in $\cals^\sfm$, respectively.  Define a bijection $\psi :  \cals^\sfm  \to \bbR^{\sfp}$  as for any $A=(A_{ij})_{1\leq i,j \leq \sfm} \in \cals^\sfm$,
		\bea
		\psi(A)=\big(A_{11}, ..., A_{1\sfm}, A_{22}, ..., A_{2\sfm}, ..., A_{\sfm\sfm}\big). \label{psi}
		\eea
		Take $\Theta\subset{\bbR^\sfp}$ as 
		\beas
		\Theta \yeq \psi(\cala),
		\eeas  
		where $\cala$ is a compact subset of $\cals_+^{\sfm}$.  Take  the true value $\theta^*$  of $\theta \in \Theta$ as $\theta^*=\psi(A^*)$, where $A^* \in \cala$.
		Assume that  for some $\delta>0$, 
		\bea\label{A^*}
		\big\{A \in \cala ; \|A-A^*\|<\delta \big\}\yeq \big\{A \in \cals_+^{\sf{m}} ;  \|A-A^*\|<\delta \big\}.
		\eea
		  Define $a_T \in GL(\sfp)$ as $a_T=T^{-\frac{1}{2}} I_\sfp$.  We   denote   ${\rm rank}(A^*)$ by $\sfr^*$.
		Then [{\bf A3}] holds, and
		\bea\label{ExpositivesemidefiU}
		U \yeq \psi\bigg(\big\{ w \in \cals^{\sfm} ; K^\prime w K \in \cals_+^{\sfm-\sfr^*}\big\}\bigg), 
		\eea
		where $K $ is a $\sfm \times (\sfm-\sfr^*)$  matrix whose column vectors form a basis for ${\rm Ker}(A^*)$. \big(If $\sfr^*=\sfm$, then  consider $U$ as $\psi(S^\sfm)$.\big)
		
		We show (\ref{ExpositivesemidefiU}) when $\sfr^* < \sfm$. (If $\sfr^*=\sfm$, then $A^* \in {\rm Int}(\cala)$, and (\ref{ExpositivesemidefiU}) obviously holds.) Take an arbitrary $R>0$. From (\ref{A^*}) and  the  definition (\ref{U_T}) of $U_T$,  for sufficiently large $T \in \bbT$,
		\beas
		U_T\cap B_{2R} &=& \psi\bigg(\big\{ w \in \cals^{\sfm} ; A^*+T^{-\frac{1}{2}} w \in \cals^\sfm_+\big\}\bigg)\cap B_{2R}.
		\eeas
		Define $W$ as 
		\beas
		W \yeq \big\{ w \in \cals^{\sfm} ; K^\prime w K \in \cals^{\sfm-\sfr^*}_+\big\}.
		\eeas
		Then we obviously have $ U_T \cap B_{2R} \subset \psi(W)$ for such large $T \in \bbT$. Therefore, $U \cap B_R \subset \bigcap_{N=1}^{\infty} \overline{\bigcup_{T\geq N} U_T}\cap B_R \subset \bigcap_{N=1}^{\infty} \overline{\bigcup_{T\geq N} U_T \cap B_{2R}}\subset \psi(W)$. This implies $U \subset \psi(W)$. Thus, for [{\bf A3}] and (\ref{ExpositivesemidefiU}), it suffices to show that
		\bea\label{Expositivesemithedesiredresult}
		\psi(W) \subset  U. 
		\eea
		Take any $w \in W$. In order to derive (\ref{Expositivesemithedesiredresult}),  we show that for sufficiently large $T \in \bbT$,
		\bea\label{Expsitivesemilem}
		\psi(w+T^{-\frac{1}{4}}I_{\sfm}) \in U_T,
		\eea
		where $I_\sfm$ denotes the $\sfm$-dimensional identity matrix. Take any $x \in \bbR^\sfm$. Decompose $\bbR^\sfm$ as $\bbR^\sfm={\rm Ker}(A^*)^{\perp} \oplus{\rm Ker}(A^*)$. We decompose $x$ as $x=y+Kz$ for some $y \in {\rm Ker}(A^*)^{\perp}$ and some $z \in \bbR^{\sfm-\sfr^*}$. Then since $w \in W$, for sufficiently large $T \in \bbT$, we have
		\beas
		x^{\prime}\big(A^*+T^{-\frac{1}{2}}(w+T^{-\frac{1}{4}}I_\sfm)\big)x &\geq& y^\prime \big(A^*+T^{-\frac{1}{2}}(w+T^{-\frac{1}{4}}I_\sfm)\big)y\\
		&& -2\|w+T^{-\frac{1}{4}}I_\sfm\|T^{-\frac{1}{2}}|y||Kz|+T^{-\frac{3}{4}}|Kz|^2 \\
		&\geq& \frac{\epsilon}{2}|y|^2-2\|w+T^{-\frac{1}{4}}I_\sfm\|T^{-\frac{1}{2}}|Kz||y|+T^{-\frac{3}{4}}|Kz|^2 \\
		&\geq& -\frac{\|w+T^{-\frac{1}{4}}I_\sfm\|^2}{\epsilon/2}T^{-1}|Kz|^2+T^{-\frac{3}{4}}|Kz|^2,
		\eeas 
		where $\epsilon$ is a positive constant  depending only on $A^*$. Thus, for sufficiently large $T \in \bbT$, 
		\beas
		A^*+T^{-\frac{1}{2}}(w+T^{-\frac{1}{4}}I_\sfm) \in \cals_+^\sfm,\eeas
		 which implies (\ref{Expsitivesemilem}).  Since $\psi(w+T^{-\frac{1}{4}}I_\sfm)$ converges to $\psi(w)$ as $T \to \infty$, we have $\psi(w) \in U$. Therefore, (\ref{Expositivesemithedesiredresult})  holds.

	}
\end{example}

\begin{example}[Non-conical $U$]\label{exU3}
	{\rm  This example derives from Section \ref{linearmixedmodel}.
		Take $\Theta$ as $\Theta = \psi({\cald})$, where $\psi$ is defined as (\ref{psi})  when $\sfm =2$, and ${\cald} \subset \cals_+^{2}$ is a compact subset.   Let $D^*=(D^*_{ij})_{1\leq i, j \leq 2}$ be the true value of $D=(D_{ij})_{1\leq i, j \leq 2} \in {\cald}$.  Suppose that for some $\delta>0$, \beas
		\big\{D \in  {\cald}  ; \|D-D^*\|< \delta \big \} \yeq \big\{D \in  {\cals_+^{2}}  ; \|D-D^*\|< \delta \big \}.
		\eeas
		Consider a case where
		\beas
		D_{11}^*>0, ~~D_{12}^*=D_{22}^*=0.
		\eeas
	Unlike Example \ref{Expositivesemidefinite}, take $a_T$ as
	$a_T = {\rm diag}(T^{-\frac{1}{2}}, T^{-\frac{1}{2}}, T^{-\frac{\rho}{2}})$,\footnote{For any $a_1, ..., a_n \in \bbR$, ${\rm diag}(a_1, ..., a_n)$  denotes an $n \times n$  diagonal matrix whose $(i, i)$ entry is $a_i$ for every $i=1, ..., n$.} where  $\rho$ is a positive number defined as  $\rho = \frac{r}{q}\vee 1$, and $0<q \leq 1$ and $0\leq r \leq 1$ are  tuning parameters.
Then [{\bf A3}] holds  and $U=W$, where $W$ is defined as
\bea\label{ex3U}
W \yeq\begin{cases}  \big\{(w_1, w_2, w_3) \in \bbR^3 ; w_3 \geq 0 \big\} & (\rho <2 \text{ \,i.e. } q>\frac{r}{2})\\
	\big\{(w_1, w_2, w_3) \in \bbR^3 ; D_{11}^*w_3 - w_2^2 \geq 0\big\} & (\rho=2 \text{ \,i.e. }q=\frac{r}{2})\\
	\big\{(w_1, w_2, w_3) \in \bbR^3 ; w_3 \geq 0, w_2= 0\big\} & (\rho >2 \text{ \,i.e. }q<\frac{r}{2})
	\end{cases}.
\eea 

In fact,  from the  definition (\ref{U_T}) of $U_T$, for any $R>0$ and for sufficiently large $T \in \bbT$,
\beas
U_T \cap B_R &=& \big\{(w_1, w_2, w_3) \in \bbR^3 ; \\
&&D_{11}^* +T^{-\frac{1}{2}}w_1 \geq 0, ~w_3 \geq 0,~ (D_{11}^*+T^{-\frac{1}{2}}w_1)T^{-\frac{\rho}{2}}w_3 - T^{-1}w_2^2 \geq 0\big\} \cap B_R\\
&=& \big\{(w_1, w_2, w_3) \in \bbR^2 \times [0, \infty) \,;\, (D_{11}^*+T^{-\frac{1}{2}}w_1)T^{-\frac{\rho}{2}}w_3 - T^{-1}w_2^2 \geq 0\big\} \cap B_R.
\eeas
Define $f, g_T : \bbR^2 \times [0, \infty) \to \bbR$  as
 for any $w=(w_1, w_2, w_3) \in \bbR^2 \times [0, \infty)$,
\beas
f(w) \yeq  \begin{cases}  D_{11}^*w_3  & (q>\frac{r}{2})\\
	 D_{11}^*w_3 - w_2^2& (q=\frac{r}{2})\\
 -w_2^2& (q<\frac{r}{2})
\end{cases},  ~~g_T(w) \yeq  \begin{cases}   T^{-\frac{1}{2}}w_1w_3 - T^{-1+\frac{\rho}{2}}w_2^2& (q>\frac{r}{2})\\
T^{-\frac{1}{2}}w_1w_3& (q=\frac{r}{2})\\
(D_{11}^*+T^{-\frac{1}{2}}w_1)T^{-\frac{\rho}{2}+1}w_3 & (q<\frac{r}{2})
\end{cases}, 
\eeas respectively. Then [{\bf U1}]-[{\bf U3}]  holds for 
\beas
\calk_+ \yeq  \begin{cases}  \phi  & (q \geq \frac{r}{2})\\
	\{1\}& (q<\frac{r}{2})
	\end{cases},
\eeas and we obtain $U=\{w \in \bbR^2 \times [0, \infty) ; f(w) \geq 0\}$. Thus, (\ref{ex3U}) holds.
}

\end{example}

\subsection{Relationship to previous studies}
Recall that $U_T$ is defined by  (\ref{U_T}).
We show that the approach of using the limit set $U$ defined by   (\ref{U}) is a generalization of the prior work on the case where the true value is on the boundary. 
Chernoff \cite{chernoff1954distribution} considers a local approximation of the set $\Theta-\theta^*$ by a cone\footnote{We call a non-empty set $\Lambda$ a cone (with its vertex at the origin) if for any $\lambda\in \Lambda$ and any $t \geq 0$, $t\lambda \in \Lambda$.}, and  Andrews \cite{andrews1999estimation} extends this method. The definition of the local approximation by a cone  is as follows.
\begin{definition}
	{\rm We say that  a sequence of sets $\{\Phi_T \subset \bbR^{\sf{p}} ; T \in \bbT\}$ is locally  approximated (at the origin) by a cone $\Lambda \subset \bbR^{\sf{p}}$ if
		\beas
		&&d(\phi_T, \Lambda)=o(|\phi_T|) \qquad\forall\{\phi_T\in\Phi_T ; T\in \bbT\} {\rm ~such~ that~}  |\phi_T|\rightarrow 0,\\
		&&d(\lambda_T, \Phi_T)=o(|\lambda_T|) \qquad\forall\{\lambda_T\in\Lambda ; T\in \bbT\} {\rm ~such~that~} |\lambda_T|\rightarrow 0,\\
		\eeas where  for $A\subset \bbR^\sfp$ and for $x \in \bbR^\sfp$, $d(x, A)$ denotes $\inf_{a \in A}d(x, a)=\inf_{a \in A}|x-a|$.} \end{definition}
	This definition is proposed by Andrews \cite{andrews1999estimation}, and if $\Phi_T$ is simply equal to $\Theta-\theta^*$, then this definition is consistent with that  put forward by Chernoff \cite{chernoff1954distribution}.
Using this concept,  we consider the following   conditions  proposed by them.
\bd
\im[${\bf [An]}$] For some sequence of positive scalar constants $\{b_T \}_{T\in\bbT}$ with $b_T \rightarrow \infty$, $\{U_T/b_T ; T\in \bbT\}$ is locally approximated by some cone $\Lambda$.

\im[${\bf [Ch1]}$] $\bbT=\bbZ_{\geq1}$, where $\bbZ_{\geq 1}$ denotes the set of all positive integers. Moreover, $a_T$ is equal to the diagonal matrix $T^{-\frac{1}{2}}I_\sfp$.  

\im[${\bf [Ch2]}$]  $\{\Theta-\theta^* ; T\in\bbT\}$ is locally approximated by some cone $\Lambda$.
\ed
Note that if [{\bf Ch1}] and [{\bf Ch2}] holds, then [{\bf An}] also holds with $b_T =T^{\frac{1}{2}}$. The following theorem  says that [{\bf A3}] is a generalization of these conditions. 
\begin{theorem}\label{thmrelation}
	\bd \im[(i)] If $[{\bf An}]$ holds, then $[{\bf A3}]$ holds, and 
	\beas U=\overline{\Lambda}.
	\eeas
	\im[(ii)] Assume $[{\bf Ch1}]$.  Then $[{\bf A3}]$ holds if and only if $[{\bf Ch2}]$ holds. Furthermore, in this case, 
	\beas U=\overline{\Lambda}.
	\eeas
	\ed
	\end{theorem}

\begin{remark}
	{\rm \bd \im[(i)] 
		Strictly speaking, Andrews \cite{andrews1999estimation} considers [{\bf An}] only when  $\bbT \subset(0, \infty)$ and $b_T \leq c \lambda_{\rm min}[a_T^{-1}]$\footnote{ $\lambda_{\rm min}[A]$ denotes the minimum eigenvalue of a  matrix $A$.}  for some $0<c<\infty$.
	\im[(ii)] Since $U$ can be a set other than a cone such as Example \ref{exU3}, [{\bf A3}] is strictly weaker than [{\bf An}] and $[{\bf Ch1}]+[{\bf Ch2}]$.
\ed }
\end{remark}

\begin{proof}[\underline{Proof of Theorem \ref{thmrelation}}]
	~
	
	\vspace{2mm}
\noindent {\bf (i)} Assume [{\bf An}]. First, we show $\overline{\Lambda}\subset U$. Take an arbitrary $\lambda\in\Lambda$. Then $\lambda/b_T \in\Lambda$ for any $T \in \bbT$. Since $\{U_T/b_T ; T\in \bbT\}$ is locally approximated by a cone $\Lambda$, we have
\beas
d(U_T/b_T, \lambda/b_T)=o(1/b_T).
\eeas
Therefore, there exists a sequence $\{u_T \in U_T ; T\in\bbT \}$ such that
\beas
d(u_T, \lambda)=o(1).
\eeas
By the definition of $U$, this implies $\lambda \in U$. Thus  $\Lambda \subset U$, and since  $U$ is closed from Lemma \ref{lem1}, $\overline{\Lambda}\subset U$.

Therefore, it is sufficient to show that
\bea\label{202206021300} \big(U\subset \big)\bigcap_{N=1}^{\infty} \overline{\bigcup_{T\geq N} U_T}\subset \overline{\Lambda}
\eea
Take an arbitrary $u \in \bigcap_{N=1}^{\infty} \overline{\bigcup_{T\geq N} U_T}$. Then there exist a subsequence $\{T_n\}_{n=1}^\infty$ of $\bbT$ and a sequence $\{u_{T_n} \in U_{T_n} ; n=1, 2, ...\}$ such that 
\bea\label{202206021306} u_{T_n} \rightarrow u.
\eea
Since $\{U_T/b_T ; T\in \bbT\}$ is locally approximated by a cone $\Lambda$, we have
\beas
 d(u_{T_n}/b_{T_n}, \Lambda)=o(u_{T_n}/b_{T_n})=o(1/b_{T_n}).
\eeas
Therefore, there exists a sequence $\{\lambda_{T_n} \in \Lambda ; n=1, 2, ... \}$ such that
\beas
d(u_{T_n}, b_{T_n}\lambda_{T_n})=o(1).
\eeas
From (\ref{202206021306}), $b_{T_n}\lambda_{T_n}\rightarrow u$, which implies $u \in \overline{\Lambda}$. Thus, (\ref{202206021300}) holds.

\vspace{2mm}
\noindent {\bf (ii)}  Assume [{\bf Ch1}]. Since [{\bf An}] implies [{\bf A3}], the assumption [{\bf Ch2}] also implies [{\bf A3}], and $U=\overline\Lambda$.  

We show the converse. Assume $[{\bf Ch1}]$ and $[{\bf A3}]$. We first show $U$ becomes a cone. Take an arbitrary $u\in U$ and an arbitrary $k>0$. From the definition of $U$, there exists some sequence $\{\theta_n \in \Theta ; n=1, 2, ...\} $ such that
\beas n^{\frac{1}{2}}(\theta_n-\theta^*)\rightarrow u.
\eeas
Take a subsequence $\{T_n\}_{n=1}^{\infty}$ of $\bbZ_{\geq 1}$ satisfying that $T_n \uparrow \infty$ and 
\beas
(T_n-1)^{\frac{1}{2}} \leq kn^{\frac{1}{2}} \leq T_n^{\frac{1}{2}} \qquad(n=1, 2, ...).
\eeas
Then, obviously, $kn^{\frac{1}{2}}/T_n^{\frac{1}{2}}\rightarrow 1$. Therefore,
\beas
T_n^{\frac{1}{2}}(\theta_n -\theta^*)=\big(T_n^{\frac{1}{2}}k^{-1}n^{-\frac{1}{2}}\big)kn^{\frac{1}{2}}(\theta_n -\theta^*)\rightarrow ku.
\eeas
Since $T_n^{\frac{1}{2}}(\theta_n -\theta^*) \in U_{T_n}$ for any $n\in \bbZ_{\geq1}$, $ku \in \bigcap_{N=1}^{\infty} \overline{\bigcup_{T\geq N} U_T}$. From [{\bf A3}], $ku \in U$. Since $U$ obviously contains zero, $U$ is a cone.

 Next, we show that $\{\Theta-\theta^* ; n \in \bbZ_{\geq 1}\}$ is locally approximated by $U$.
That is, we show \bea
&&d(\phi_n, U)=o(|\phi_n|) \qquad\forall\{\phi_n\in\Theta-\theta^* ; n\in \bbZ_{\geq 1}\} {\rm ~such~ that~}  |\phi_n|\rightarrow 0\qquad and\label{6160155}\\
&&d(u_n, \Theta-\theta^*)=o(|u_n|) \qquad\forall\{u_n\in U ; n\in \bbZ_{\geq 1}\} {\rm ~such~that~} |u_n|\rightarrow 0 \label{6160156}.
\eea
We first show (\ref{6160155}).
 Take an arbitrary sequence $\{\phi_n\in\Theta-\theta^* ; n\in \bbZ_{\geq 1}\}$ with $|\phi_n| \rightarrow 0$. 
We may assume that $|\phi_n| \neq 0$ for any $n\in\bbZ_{\geq 1}$.
Take a subsequence $\{T_n\}_{n=1}^{\infty}$ of $\bbZ_{\geq 1}$ satisfying that $T_n \rightarrow \infty$ and 
\beas
(T_n-1)^{\frac{1}{2}} \leq |\phi_n|^{-1} \leq T_n^{\frac{1}{2}} \qquad(n=1, 2, ...).
\eeas
Then, obviously, $T_n^{\frac{1}{2}}|\phi_n|\rightarrow 1$. Therefore, we may assume that $T_n^{\frac{1}{2}}|\phi_n|\leq 2$. Then
\beas
|\phi_n|^{-1}d(\phi_n, U)
\yleq T_n^{\frac{1}{2}}d(\phi_n, U)&=&d(T_n^{\frac{1}{2}}\phi_n, T_n^{\frac{1}{2}}U)\\
&=&d(T_n^{\frac{1}{2}}\phi_n, U) \qquad\big(\because U \text{ is a cone.}\big)\\
&\leq& \sup_{u\in U_{T_n}\cap \overline{B_2}}d(u, U).
\eeas
Take any  $\delta>0$. From Lemma \ref{lem1}  (iii), for sufficiently large number $T_n$,
\beas U_{T_n} \cap \overline {B_2} \subset (U\cap \overline{B_2})^\delta.
\eeas
Therefore, for sufficiently large number $n$,
\beas
|\phi_n|^{-1}d(\phi_n, U)
\leq \sup_{u\in (U\cap \overline{B_2})^\delta}d(u, U)\leq \delta.
\eeas
This implies (\ref{6160155}). 

 We also show (\ref{6160156}). Take an arbitrary $\{u_n\in U ; n\in \bbZ_{\geq 1}\} {\rm ~such~that~} |u_n|\rightarrow 0 $.
We may assume that $|u_n| \neq 0$ for any $n\in\bbZ_{\geq 1}$.
Take a subsequence $\{T_n\}_{n=1}^{\infty}$ of $\bbZ_{\geq 1}$ satisfying that $T_n \rightarrow \infty$ and 
\beas
(T_n-1)^{\frac{1}{2}} \leq |u_n|^{-1} \leq T_n^{\frac{1}{2}} \qquad(n=1, 2, ...).
\eeas
Then, obviously, $T_n^{\frac{1}{2}}|u_n|\rightarrow 1$. Therefore, we may assume that $T_n^{\frac{1}{2}}|u_n|\leq 2$. Then
\beas
|u_n|^{-1}d(u_n, \Theta-\theta^*)
\yleq T_n^{\frac{1}{2}}d(u_n, \Theta-\theta^*)
&=&d(T_n^{\frac{1}{2}}u_n, U_{T_n})\\
&\leq& \sup_{u\in U \cap \overline{B_2}}d(u, U_{T_n}) \qquad\big(\because U \text{ is a cone.}\big).
\eeas
Take any  $\delta>0$. From Lemma \ref{lem1}  (ii), for sufficiently large number $T_n$,
\beas U  \cap \overline {B_2} \subset (U_{T_n}\cap \overline{B_4})^\delta.
\eeas
Therefore, for sufficiently large number $n$,
\beas
|u_n|^{-1}d(u_n, \Theta-\theta^*)
\leq \sup_{u\in (U_{T_n}\cap \overline{B_4})^\delta}d(u, U_{T_n})\leq  \delta.
\eeas
This implies (\ref{6160156}).  Thus, [{\bf Ch2}] holds.
	\end{proof}

\subsection{Example: generalized inverse Gaussian distribution}\label{2204270415}
{\rm As an example of a direct use of Theorem \ref{thm1}, we estimate the parameters of  the generalized inverse Gaussian distribution $\text{GIG}(\lambda,\delta,\gamma)$  that is a probability measure on $(0,\infty)$ with the density function 
\bea\label{220404270423}
p_{\text{GIG}}(x;\lambda,\delta,\gamma)
&=&
\frac{(\gamma/\delta)^\lambda}{2K_\lambda(\gamma\delta)}x^{\lambda-1}
\exp\bigg[-\half\bigg(\frac{\delta^2}{x}+\gamma^2x\bigg)\bigg]\qquad(x>0),
\eea
where $K_\nu$ is the modified Bessel function of the second kind with index $\nu$ defined by 
\beas 
K_\nu(x)
&=& 
\half\int_0^\infty y^{\nu-1}\exp\bigg[-\half x\bigg(y+\frac{1}{y}\bigg)\bigg]dy. 
\eeas
The function $K_\nu$ can be extended to complex $\nu$ when $\text{Re}(\nu)>-1/2$. 
Due to the integrability of the function on the right-hand side of (\ref{220404270423}), 
the parameters are restricted as follows. 
(i) $\lambda>0$, $\delta\geq0$, $\gamma>0$, 
(ii) $\lambda=0$, $\delta>0$, $\gamma>0$, and 
(iii) $\lambda<0$, $\delta>0$, $\gamma\geq0$. 
Here we will treat the case (i). 
Case (ii) can also be approached directly or through the duality 
\beas 
X\sim\text{GIG}(\lambda,\delta,\gamma)&\Leftrightarrow&X^{-1}\sim\text{GIG}(-\lambda,\gamma,\delta), 
\eeas
and in particular, 
$
X\sim\Gamma(\lambda,c)\Leftrightarrow X^{-1}\sim{\text{I}\Gamma }(\lambda,c)
$, 
the inverse gamma distribution.

A special case of (i) is the Gamma distribution 
$\Gamma(\lambda,\gamma^2/2)=\text{GIG}(\lambda,0,\gamma)$ having the density function 
\beas 
p_{\Gamma}(x;\lambda,\gamma^2/2)
&=& 
p_{\text{GIG}}(x;\lambda,0,\gamma)
\nn\\&=&
\frac{1}{\Gamma(\lambda)}\begin{pmatrix}\frac{\gamma^2}{2}\end{pmatrix}^\lambda
x^{\lambda-1}\exp\bigg(-\frac{\gamma^2}{2}x\bigg)\qquad(x>0)
\eeas
as the limit of $\text{GIG}(\lambda,\delta,\gamma)$ when $\delta\down0$ since 
\bea\label{2204270504}
K_\nu(z) 
&\sim& 
\half\Gamma(\nu)\bigg(\frac{z}{2}\bigg)^{-\nu}
\eea
as $z\to0$, when $\text{Re}(\nu)>0$ for $\nu\in\bbC$. 
The property (\ref{2204270504}) is verified with the representation 
\bea\label{0404291135}
z^\nu K_\nu(z) 
\yeq
\half\int_0^\infty t^{\nu-1}\exp\bigg[-\half\bigg(t+\frac{z^2}{t}\bigg)\bigg]dt
\eea
when $\text{Re}(\nu)>0$.

\begin{en-text}
	A special case of (\ref{220404270431}) is the inverse Gaussian distribution
	$IG(\delta,\gamma)=\text{GIG}(-1/2,\delta,\gamma)$ with the density 
	\beas 
	p_{\text{IG}}(x;\delta,\gamma)
	&=& 
	\frac{\delta}{\sqrt{2\pi}}e^{\delta\gamma}x^{-\frac{3}{2}}\exp\bigg[-\half\bigg(\frac{\delta^2}{x}+\gamma^2x\bigg)\bigg]
	\eeas
	for $x>0$. 
	A further special case is the L\'evy distribution $\text{L\'evy}(\delta)$ that has the density 
	\beas 
	p_{\text{L\'evy}}(x;\delta)
	&=& 
	\frac{\delta}{\sqrt{2\pi}}x^{-\frac{3}{2}}\exp\bigg(-\frac{\delta^2}{x}\bigg)
	\eeas
	for $x>0$. In fact, 
	\beas 
	\text{L\'evy}(\delta)  \yeq \text{IG}(\delta,0) \yeq \text{GIG}(-1/2,\delta,0).
	\eeas
\end{en-text}

\def\ul{\underline}
\def\ol{\overline}
Let us consider estimation of the parameters based on the independent observations $(X_j)_{j=1}^n$ 
from the experiment 
\beas
\big\{\text{GIG}(\lambda,\delta,\gamma);\>(\lambda,\delta,\gamma)\in[\ul{\lambda},\ol{\lambda}]\times[0,\ol{\delta}]\times[\ul{\gamma},\ol{\gamma}]\big\},
\eeas
where the end points of the intervals satisfy
\beas 
2<\ul{\lambda}<\ol{\lambda}<\infty,\quad
0<\ol{\delta}<\infty,\quad
0<\ul{\gamma}<\ol{\gamma}<\infty.
\eeas
\begin{en-text}
	As illustrative examples of the model that has the true parameters $(\lambda^*,\delta^*,\gamma^*)$ 
	on the boundary of the parameter space, 
	we will consider two cases of the true distribution: 
	\begin{enumerate}[(I)] 
		\im\label{2204270541} $\text{IG}(\delta^*,\gamma^*)=\text{GIG}(-1/2,\delta^*,\gamma^*)$ with 
		$\lambda^*=-1/2$, $\delta^*>0$, $\gamma^*>0$. 
		\im\label{2204270542} $ \text{L\'evy}(\delta^*)=\text{GIG}(-1/2,\delta^*,\gamma^*)$ with 
		$\lambda^*=-1/2$, $\delta^*>0$, $\gamma^*=0$. 
	\end{enumerate}
\end{en-text}
Suppose that the distribution generating the data 
is $\Gamma(\lambda^*,(\gamma^*)^2/2)=\text{GIG}(\lambda^*,0,\gamma^*)$, that is, 
the true value $(\lambda^*,\delta^*,\gamma^*)$ of $(\lambda,\delta,\gamma)$ 
is located on the boundary of the parametric model as 
\beas 
\ul{\lambda}<\lambda^*<\ol{\lambda},\quad 
\delta^*=0,\quad
\ul{\gamma}<\gamma^*<\ol{\gamma}.
\eeas
Jorgencen \cite{jorgensen2012statistical} also treated this case although the parametrization is slightly different. 

To consider the problem, it is possible to re-parametrize the model 
into a natural exponential family and to use Theorem \ref{thmB} described below, but it requires to transform the limit distribution after getting it to return the original parameters.
We keep the original parameters here to illustrate the general approach, and derive the  asymptotic distribution of the maximum likelihood estimator $\big(\wh{\lambda}_n,\wh{\delta}_n,\wh{\gamma}_n\big)$ for $(\lambda,\delta,\gamma)$.  

Let
\beas
\calk(\lambda,a,b) 
&=& 
\frac{b^\lambda}{2}\int_0^\infty t^{\lambda-1}
\exp\bigg[-\half\bigg(bt+\frac{a}{t}\bigg)\bigg]dt
\eeas
Then 
\beas 
\calk(\lambda,\delta^2,\gamma^2)
&=& 
(\gamma\delta)^\lambda K_\lambda(\gamma\delta)
\eeas

Simply denoted by $p(x;\lambda,\delta,\gamma)$, the density $p_{\text{GIG}}(x;\lambda,\delta,\gamma)$ 
is expressed as 
\beas 
p(x;\lambda,\delta,\gamma)
&=& 
\frac{\gamma^{2\lambda}}{2\calk(\lambda,\delta^2,\gamma^2)}x^{\lambda-1}
\exp\bigg[-\half\bigg(\frac{\delta^2}{x}+\gamma^2x\bigg)\bigg]\qquad(x>0)
\eeas
This model is a curved exponential family: 
\beas 
p(x;\lambda,\delta,\gamma)
&=& 
\exp\bigg[(\lambda-1)\log x-\frac{\delta^2}{2x}-\frac{\gamma^2}{2}x-\Psi(\lambda,\delta^2,\gamma^2)\bigg]
\qquad(x>0)
\eeas
with the potential 
\beas 
\Psi(\lambda,a,b)
&=&
-\log\frac{b^{\lambda}}{2\calk(\lambda,a,b)}
\eeas

{The $r$ times tensor product of a vector $v$ is denoted by $v^{\otimes r}$. For a tensor $T=(T_{i_1, ..., i_k})_{i_1, ..., i_k}$ and vectors $v_1=(v_1^{i_1})_{i_1}, ..., (v_k^{i_k})_{i_k}$, we write
\beas
T[v_1 ,..., v_k]\yeq T[v_1 \otimes \cdots \otimes v_k]\yeq \sum_{i_1, ..., i_k}T_{i_1, ..., i_k}v_1^{i_1}\cdots v_{k}^{i_k}.
\eeas}
We have 
\beas &&
\Psi(\lambda,\delta^2,\gamma^2)-\Psi(\lambda^*,0,(\gamma^*)^2)
\nn\\&=& 
D\big[\Delta(\lambda,\delta^2,\gamma^2)\big]
+\int_0^1(1-s)H(s,\lambda,\delta^2,\gamma^2)\big[(\Delta(\lambda,\delta^2,\gamma^2)^{\otimes2}\big]ds.
\eeas
Here 
\beas&&
\Delta(\lambda,\delta^2,\gamma^2)
\yeq
\begin{pmatrix}\lambda-\lambda^*\\ \delta^2 \\ \gamma^2-(\gamma^*)^2\end{pmatrix},\qquad
D
\yeq
\big(\partial_{(\lambda,a,b)}\Psi\big)\big(\lambda^*,0,(\gamma^*)^2\big)
\yeq
\begin{pmatrix}E[\log\xi_0]\\ 
	2^{-1}E[\xi_0^{-1}] \\ 
	2^{-1}E[\xi_0]\end{pmatrix}
\eeas
and 
\beas
H(s,\lambda,\delta^2,\gamma^2)
&=&
\big(\partial_{(\lambda,a,b)}^2\Psi\big)
(s\lambda+(1-s)\lambda^*,s\delta^2,s\gamma^2+(1-s)(\gamma^*)^2)
\nn\\
&=&
\begin{pmatrix}\text{Var}\big[\log\xi_s\big]& 
	2^{-1}\text{Cov}\big[\log\xi_s,\xi_s^{-1}\big]&
	2^{-1}\text{Cov}\big[\log\xi_s,\xi_s\big]
	\\
	2^{-1}\text{Cov}\big[\log\xi_s,\xi_s^{-1}\big]& 
	4^{-1}\text{Var}\big[\xi_s^{-1}\big]&
	4^{-1}\text{Cov}\big[\xi_s^{-1},\xi_s\big]
	\\
	2^{-1}\text{Cov}\big[\log\xi_s,\xi_s\big]&
	4^{-1}\text{Cov}\big[\xi_s^{-1},\xi_s\big]&
	4^{-1}\text{Var}\big[\xi_s\big]
\end{pmatrix},
\eeas
with 
\beas 
\xi_s\yeq
\xi_s(\lambda,\delta,\gamma)=\xi(s\lambda+(1-s)\lambda^*,\sqrt{s}\delta,\sqrt{s\gamma^2+(1-s)(\gamma^*)^2}),
\eeas 
where $\xi(\lambda,\delta,\gamma)$ denotes a random variable such that 
$\xi(\lambda,\delta,\gamma)\sim\text{GIG}(\lambda,\delta,\gamma)$.

Let $a_n=\text{diag}\big[n^{-1/2},n^{-1/4},n^{-1/2}\big]$, and define $U_n$ as (\ref{U_T}). 
For $\bbH_n(\theta)=\sum_{j=1}^n\log p(X_j;\theta)$ and $u=(u_1,u_2,u_3)\in U_n$, we obtain
\beas 
\log \bbZ_n(u)
&=& 
\bbH_n(\theta^*+a_nu)-\bbH_n(\theta^*)
\nn\\&=&
u_1n^{-1/2}\sum_{j=1}^n\wt{\log X_j}
-u_2^2n^{-1/2}\sum_{j=1}^n2^{-1}\wt{X_j^{-1}}
-u_3\gamma^*n^{-1/2}\sum_{j=1}^n\wt{X_j}
\nn\\&&
-n\int_0^1(1-s)H(s,\lambda^*+n^{-1/2}u_1,n^{-1/2}u_2^2,(\gamma^*+n^{-1/2}u_3)^2))
\nn\\&&\hspace{30pt}\times
\bigg[\bigg(n^{-1/2}u_1,n^{-1/2}u_2^2,(\gamma^*+n^{-1/2}u_3)^2-(\gamma)^*)^2\bigg)^{\otimes2}\bigg]ds
\nn
\, -u_3^22^{-1}n^{-1}\sum_{j=1}^n\wt{X}_j,
\eeas
where we are writing $\wt{F(X_j)}=F(X_j)-E[F(X_j)]$ for a function $F(X_j)$ of a random variable $X_j$ satisfying $X_j\sim GIG(\lambda^*,0,\gamma^*)=\Gamma(\lambda^*,(\gamma^*)^2/2)$. 
Then it is possible to write it as 
\beas 
\log \bbZ_n(u)
&=& 
u_1n^{-1/2}\sum_{j=1}^n\wt{\log X_j}
-u_2^2n^{-1/2}\sum_{j=1}^n2^{-1}\wt{X_j^{-1}}
-u_3\gamma^*n^{-1/2}\sum_{j=1}^n\wt{X_j}
\nn\\&&
-\half C\big[(u_1,u_2^2,u_3)^{\otimes}\big]
+r_n(u)
\eeas
with the positive-definite covariance matrix 
\beas 
C
&=& 
\begin{pmatrix}\text{Var}\big[\log\xi_0\big]& 
	2^{-1}\text{Cov}\big[\log\xi_0,\xi_0^{-1}\big]&
	\gamma^*\text{Cov}\big[\log\xi_0,\xi_0\big]
	\\
	2^{-1}\text{Cov}\big[\log\xi_0,\xi_0^{-1}\big]& 
	4^{-1}\text{Var}\big[\xi_0^{-1}\big]&
	2^{-1}\gamma^*\text{Cov}\big[\xi_0^{-1},\xi_0\big]
	\\
	\gamma^*\text{Cov}\big[\log\xi_0,\xi_0\big]&
	2^{-1}\gamma^*\text{Cov}\big[\xi_0^{-1},\xi_0\big]&
	(\gamma^*)^2\text{Var}\big[\xi_0\big]
\end{pmatrix},
\eeas
and 
the term $r_n(u)$ satisfying 
\beas 
\sup_{u\in U_n}\frac{|r_n(u)|}{1+|u_1|^2+|u_2|^4+|u_3|^2}
&=&
O_p(n^{-1/2})
\eeas

Then, Condition $[{\bf A1}]$ is verified by the estimate
\beas 
\lim_{R\to\infty}
\limsup_{n\to\infty}
P\bigg[
M_n
\geq 
\big(2^{-1}\lambda_{\rm min}[C]+O_p(n^{-1/2})\big)R
\bigg]
&=&
0
\eeas
for $M_n=\big|n^{-1/2}\sum_{j=1}^n\wt{\log X_j}\big|+
\big|n^{-1/2}\sum_{j=1}^n2^{-1}\wt{X_j^{-1}}\big|+
\big|n^{-1/2}\gamma^*\sum_{j=1}^n\wt{X_j}\big|$.

Condition ${\bf [A2]}$ is satisfied with 
\beas 
\bbV_n(u)
&=& 
\exp\bigg(
u_1n^{-1/2}\sum_{j=1}^n\wt{\log X_j}
-u_2^2n^{-1/2}\sum_{j=1}^n2^{-1}\wt{X_j^{-1}}
-u_3\gamma^*n^{-1/2}\sum_{j=1}^n\wt{X_j}
\nn\\&&\hspace{30pt}
-\half C\big[(u_1,u_2^2,u_3)^{\otimes2}\big]
\bigg], \\
\bbZ(u) 
&=& 
\exp\bigg(\Delta\cdot(u_1,u_2^2,u_3)-\half C\big[(u_1,u_2^2,u_3)^{\otimes2}\big]\bigg)
\eeas
with a three-dimensional random vector $\Delta\sim N_3(0,C)$.

Now from Example \ref{Exexplicit1}, [{\bf A3}] holds, and we obtain   $U=\bbR\times[0,\infty)\times\bbR$.  Condition [{\bf A4}] obviously holds.
Therefore, Theorem \ref{thm1} concludes that the MLE
$\big(\wh{\lambda}_n,\wh{\delta}_n,\wh{\gamma}_n\big)$ 
admits 
\beas 
\big(n^{1/2}(\wh{\lambda}_n-\lambda^*),n^{1/4}\wh{\delta}_n,n^{1/2}(\wh{\gamma}_n-\gamma^*)\big)
&\overset{d}{\to}&
\wh{u}
\eeas
as $n\to\infty$ when the true distribution of the data is $\text{GIG}(\lambda^*,0,\gamma^*)=\Gamma(\lambda^*,(\gamma^*)^2/2)$, where 
\beas 
\wh{u} 
&=& 
\text{argmax}_{u\in U}\bigg(\Delta\cdot(u_1,u_2^2,u_3)-\half C\big[(u_1,u_2^2,u_3)^{\otimes2}\big]\bigg).
\eeas

\section{Quasi-maximum likelihood estimation}

 
 \subsection{Asymptotic behavior of the QMLE}\label{Quasisection1.1}
 Consider the same situation as in Section 2.1. 
 In this section, we will apply Theorem \ref{thm1} to   a quasi-likelihood function {in a regular case except that the true value $\theta^*$ may lie on the boundary.} Let us suppose that $\Theta$ and $\calt$ are compact.   Let $\caln$ be a bounded open  set in $\bbR^{\sfp}$ satisfying
\bd  
 \im[(i)] for some $\delta>0$, $\Theta\cap\{|\theta-\theta^*|<\delta\} \subset \overline\caln$,
 \im[(ii)] for any $\theta \in \caln$ and any $0 < t \leq 1$,  $t\theta +(1-t)\theta^* \in \caln$.
\ed We suppose that $\bbH_T$ can be extended to a continuous random field defined on $\Omega\times \ol\caln\times\calt$ satisfying that for every $\omega\in\Omega$ and $\tau\in\calt$, $\bbH_T(\omega, \cdot, \tau)$ is of class $C^2(\overline\caln)$\footnote{For an open  set $G$, $C^2\big(\ol G\big)$ denotes the set consisting of all functions which are of class $C^2$ in $G$ and whose derivatives can be continuously extended  on $\ol G$. }. 
 From Condition (ii), for every $\omega\in\Omega$ and $\tau\in\calt$, we can consider  the Taylor series of $\bbH_T(\omega, \cdot, \tau)$  around $\theta^*$ on $\ol\caln$.

Let $\calg$ be a sub-$\sigma$-field of $\calf$.
 Let $\Delta(\theta^*)$ be an $\bbR^{\sfp}$-valued random variable, and let $\Gamma(\theta^*)$ be  a  $\calg$-measurable $\bbR^{\sfp}\otimes\bbR^\sfp$-valued random variable. Let $a_T$ be a deterministic sequence in $GL(\sfp)$.  Define a positive sequence $b_T$ as $b_T=\lambda_{\rm min}\big[(a_T^{\prime }a_T)^{-1}\big]$. 
 Define a continuous random field $\bbY_T : \Omega\times\Xi\rightarrow\bbR$ as 
 \beas\bbY_T(\theta, \tau)=\frac{1}{b_T}\big(\bbH_T(\theta, \tau)-\bbH_T(\theta^*, \tau)\big)
 \qquad\big((\theta, \tau)\in\Xi\big).\eeas
 Let $\bbY: \Omega\times\Theta\rightarrow \bbR$ be a  continuous random field.
For [{\bf A1}] and [{\bf A2}], we consider the following conditions.

 \bd
 
 \im[${\bf [B1]}$]
 $\ds \sup_{(\theta, \tau)\in\Xi}|\bbY_T(\theta, \tau)-\bbY(\theta)|\overset{P}{\rightarrow}0.$ Also, with probability $1$, for any $\theta \in \Theta$ with $\theta\neq \theta^*$,
 \beas\bbY(\theta) < 0.
 \eeas

  \im[${\bf [B2]}$] For some $\tau_0 \in \calt$,
 \begin{eqnarray}
 	&&\sup_{\tau \in \calt}\big|a_T^{\prime}\partial_\theta\bbH_T(\theta^*, \tau)-a_T^{\prime}\partial_\theta\bbH_T(\theta^*, \tau_0)\big|\overset{P}{\to}0,\label{B2(i)}\\
 	&&a_T^{\prime}\partial_\theta\bbH_T(\theta^*, \tau_0) \overset{d_s(\calg)}{\rightarrow}\Delta(\theta^*).\label{B2(ii)}
 \end{eqnarray}
 Also, for any positive sequence $\delta_T$ with $\delta_T\rightarrow0$,
\beas\sup_{\substack{(\theta, \tau)\in \ol\caln\times\calt\\|\theta-\theta^*|<\delta_T}}\big\|a_T^{\prime}\partial^2_\theta\bbH_T(\theta, \tau)a_T+\Gamma(\theta^*) \big\|\overset{P}{\rightarrow}0.
\eeas 
 
  \im[${\bf [B3]}$]  
  $\Gamma(\theta^*)$ is almost surely positive definite.

 \ed
 Define $U_T$ and $U$ as (\ref{U_T}) and (\ref{U}), respectively.
 We also define $\hat{u}_T$ by $\hat{u}_T= (a_T)^{-1}(\hat{\theta}_T - \theta^*)$.
 \begin{theorem}\label{thmB}
 	 Assume {\rm $[{\bf B1}]$-$[{\bf B3}]$ and [{\bf A3}]}. Also, assume {\rm [{\bf A4}]} for the continuous random field $\bbZ$ defined as
 		\begin{eqnarray}\label{Z_TB}\bbZ(u)=\exp\bigg\{\Delta(\theta^*)[u]-\frac{1}{2}\Gamma(\theta^*)[u^{\otimes2}]\bigg\} \qquad(u\in\bbR^\sfp).
 		\end{eqnarray}
 		Then for the $U$-valued random variable $\hat u$ defined in $[{\bf A4}]$, \begin{eqnarray*}
 			\hat u_T\overset{d_s(\calg)}{\rightarrow}\hat u.
 		\end{eqnarray*}
 	Moreover, assume {\rm [{\bf A5}]} for the continuous random field $\bbV_T$ defined as
 		\begin{eqnarray}\label{V_TB}\bbV_T(u)=\exp\bigg\{\Delta_T(\theta^*, \hat\tau_T)[u]-\frac{1}{2}\Gamma(\theta^*)[u^{\otimes2}]\bigg\} \qquad(u\in\bbR^\sfp),
 		\end{eqnarray}
 		where  $\Delta_T(\theta^*, \cdot)$ represents $a_T^{\prime}\partial_\theta\bbH_T(\theta^*, \cdot)$. Then for the $U$-valued random variable $\hat v_T$ defined in $[{\bf A5}]$, \begin{eqnarray*}
 			\hat u_T-\hat v_T=o_P(1).
 			\end{eqnarray*}
 \end{theorem}
\begin{remark}\label{remarkB}
		{\rm  Condition [{\bf B1}] implies the following condition.
			\bd
			\im[${\bf [B1]^{\flat}}$] For any $\delta>0$, \beas 
			\varlimsup_{T \to \infty}P\bigg[\sup_{\substack{(\theta, \tau)\in \Theta\times\calt\\|\theta-\theta^*|\geq \delta}}\big\{\bbH_T(\theta, \tau)-\bbH_T(\theta^*, \tau)\big\} \geq 0\bigg] \yeq 0.
			\eeas
			\ed
			Then Theorem \ref{thmB} holds even if we substitute [{\bf B1}] with $[{\bf B1}]^{\flat}$. (See the following proof.)
			
			\begin{en-text}
				[(ii)]  Condition ${\bf [B1]^{\flat}}$ holds under the following condition.
			
			\bd   \im[${\bf [B1]^{\natural}}$] $\Theta \subset \caln$, and for some $\tau_0 \in \calt$, (\ref{B2(i)}) and (\ref{B2(ii)}) hold. Also, 
			\beas\sup_{(\theta, \tau)\in \Theta\times\calt}\big|a_T^{\prime}\partial^2_\theta\bbH_T(\theta, \tau)a_T+G(\theta) \big|\overset{P}{\rightarrow}0
			\eeas 
			for some $C(\Theta : \bbR^\sfp \otimes \bbR^\sfp)$-valued random variable $G$ which satisfies that with probability $1$,  $\int_0^1(1-t)G\big(t\theta+(1-t)\theta^*\big)dt$ is positive definite uniformly in $\theta \in \Theta$. (Then take $\Gamma(\theta^*)$ as $\Gamma(\theta^*)=G(\theta^*)$.)
			\ed
			Besides, ${\bf [B1]^{\natural}}$ implies [{\bf B2}] and [{\bf B3}]. Therefore, we can substitute [{\bf B1}]-[{\bf B3}] with this condition.
			\end{en-text}
		}
		\end{remark}
}
 
 \begin{proof}[Proof of Theorem \ref{thmB}]
 	From Theorem \ref{thm1}, it suffices to show that [{\bf B1}]-[{\bf B3}] imply [{\bf A1}] and [{\bf A2}]. Define  $r_T$ as 
 	\begin{eqnarray*}
 		r_T(u, \tau)&=&-2\int_0^1(1-k)\big\{a_T^{\prime}\partial^2_\theta\bbH_T(\theta^*+ka_Tu, \tau)a_T+\Gamma(\theta^*)\big\}dk \\ 		
 	\end{eqnarray*} for any $(u, \tau) \in U_T\times\calt$.
 	Then,  from Taylor's series, 
 	\begin{eqnarray*}
 		\bbZ_T(u, \tau)=\exp\bigg\{\Delta_T(\theta^*, \tau)[u]-\frac{1}{2}\big(\Gamma(\theta^*)+r_T(u, \tau)\big)[u^{\otimes2}] \bigg\}.
 	\end{eqnarray*}
 Note that from [{\bf B2}], for any positive sequence $\delta_T$ with $\delta_T\to0$,
 \beas
 \sup_{(u, \tau)\in U_T\times \calt, |a_Tu|<\delta_T}\big\|r_T(u, \tau)\big\|=o_P(1).
 \eeas
 	Therefore, from [{\bf B2}], [{\bf A2}] obviously holds if we take $\bbZ$ and $\bbV_T$ as (\ref{Z_TB}) and (\ref{V_TB}), respectively.
 	We show that $[{\bf A1}]$ holds.  Let  $\epsilon$ and $R$ be positive numbers. Also, let $\delta_T$ be a positive sequence with $\delta_T\to0$. Then
 	\bea&&\varlimsup_{T\rightarrow \infty}
 	P\bigg[\sup_{U_T\times \calt,  |u| \geq R} \mathbb{Z}_T(u, \tau)  \geq 1 \bigg]\nonumber\\
 	&\leq&\varlimsup_{T\rightarrow \infty}
 	P\bigg[\sup_{U_T\times\calt,  |a_T u|\geq \delta_T} \mathbb{Z}_T(u, \tau)  \geq 1 \bigg]
 	+
 	\varlimsup_{T\rightarrow \infty}
 	P\bigg[\sup_{\substack{U_T\times\calt,  |u| \geq R\\|a_T u|<\delta_T}} \mathbb{Z}_T(u, \tau)  \geq 1 \bigg].\label{Brights} \eea
 	From [{\bf B1}], we can evaluate the first term in (\ref{Brights}) as
 	\bea
 	&&\varlimsup_{T\rightarrow \infty}
 	P\bigg[\sup_{U_T\times\calt,  |a_T u|\geq \delta_T} \mathbb{Z}_T(u, \tau)  \geq 1 \bigg]\nonumber\\
 	&\leq&
 	\varlimsup_{T\rightarrow \infty}
 	P\bigg[\sup_{\Theta\times\calt,  |\theta-\theta^*|\geq \delta_T} \bigg\{\bbH_T(\theta, \tau)-\bbH_T(\theta^*, \tau)\bigg\}  \geq 0 \bigg]\label{20221015}\\
 	&\leq&
 	\varlimsup_{T\rightarrow \infty}
 	P\bigg[\sup_{\Theta\times\calt,  |\theta-\theta^*|\geq \delta_T} \bigg\{\bbY(\theta)+\bbY_T(\theta, \tau)-\bbY(\theta)\bigg\}  \geq 0 \bigg]\nonumber\\
 	&\leq&
 	\varlimsup_{T\rightarrow \infty}
 	P\bigg[\sup_{\Theta\times\calt,  |\theta-\theta^*|\geq \delta_T}\bbY(\theta)+o_P(1)  \geq 0 \bigg].\nonumber
 	\eea
 	(When using $[{\bf B1}]^{\flat}$,  stop at (\ref{20221015})).  Thus,  by  choosing  $\delta_T$ properly and using [{\bf B1}] (or $[{\bf B1}]^{\flat}$), we have
 	\beas
 \varlimsup_{T\rightarrow \infty}
 	P\bigg[\sup_{U_T\times\calt,  |a_T u|\geq \delta_T} \mathbb{Z}_T(u, \tau)  \geq 1 \bigg]\yeq 0.
 	\eeas
 	Also, from [{\bf B2}] and [{\bf B3}], we can evaluate the second term in (\ref{Brights}) as
 	\begin{eqnarray*}
 		&&\varlimsup_{T\rightarrow \infty}
 		P\bigg[\sup_{\substack{U_T\times\calt,  |u| \geq R\\|a_T u|<\delta_T}} \mathbb{Z}_T(u)  \geq 1 \bigg] \\
 		&=& 
 		\varlimsup_{T\rightarrow \infty}P\bigg[\sup_{\substack{U_T\times\calt,  |u| \geq R\\|a_T u|<\delta_T}} \bigg\{\Delta_T(\theta^*, \tau)[u]-\frac{1}{2}\big(\Gamma(\theta^*)+r_T(u, \tau)\big)[u^{\otimes2}]\bigg\}\geq 0 \bigg]\\
 		&\leq&
 		\varlimsup_{T\rightarrow \infty}P\bigg[\sup_{\substack{U_T\times\calt,  |u| \geq R\\|a_T u|<\delta_T}} \bigg\{O_P(1)|u|-\frac{1}{2}\lambda_{\min}[\Gamma(\theta^*)]|u|^2  +o_P(1)|u|^2\bigg\}\geq0 \bigg]\\
 		&\leq& 
 		\varlimsup_{T\rightarrow \infty}P\bigg[\sup_{\substack{U_T\times\calt,  |u| \geq R\\|a_T u|<\delta_T}} \bigg\{O_P(1)|u|-\frac{1}{2}\epsilon|u|^2  +o_P(1)|u|^2\bigg\}\geq 0 \bigg]
 		+P\bigg[\lambda_{\min}[\Gamma(\theta^*)]<\epsilon \bigg]\\
 		&\overset{R\rightarrow\infty}{\rightarrow}&
 	P\bigg[\lambda_{\min}[\Gamma(\theta^*)]<\epsilon \bigg]\,\overset{\epsilon\rightarrow0}{\rightarrow}\, 0.
 	\end{eqnarray*}
 	Thus, $[{\bf A1}]$ holds.
 \end{proof}


 	 \subsection{It\^o process}\label{Ito}
 			Consider a $\sfd$-dimensional It\^o process 
 		\beas(Y_t)_{t\in[0,T]}=\big((Y_t^1, ..., Y_t^{\sfd})^\prime\big)_{t \in [0, T]}
 		\eeas
 		 having a decomposition 
 		\bea\label{SDE} 
 		Y_t &=& Y_0+\int_0^tb_sds+\int_0^t\sigma(X_s, a)dw_s\qquad(t\in[0,T])
 		\eea
 		on a stochastic basis $(\Omega,\calf,\F,P)$, $\F=(\calf_t)_{t\in[0,T]}$, where 
 		$\{b_t\}_{t\in[0,T]}$ is an unobservable $\sfd$-dimensional progressively measurable process, $w=(w_t)_{t\in[0,T]}$ is an $\sfr$-dimensional standard Wiener process,  and $(X_t)_{t \in [0, T]}$ is a $\sf{c}$-dimensional $\F$-adapted continuous process. Here we parametrize an $\bbR^\sfd \otimes \bbR^\sfr$-valued function $\sigma$ by a matrix $a $ as
 		\beas
 		\sigma(x, a)&=&f(x)\, a \qquad(x \in \bbR^\sfd, a \in \bbR^{\sf{m}}\otimes\bbR^{\sfr} ),
 		\eeas
 		where $f : \bbR^{\sf{c}} \to \bbR^{\sfd}\times\bbR^{\sf{m}}$ is a $C^2$ map.
 		For the sake of identifiability, we parametrize $a$ by   $A=aa^\prime$. Denote by $a^*$ the true value of $a$. Then the true value $A^*$ of $A$ is obviously determined as  $A^*=a^*(a^*)^\prime$.
 		 We take a compact set $\cala \subset \cals_+^{\sfm}$ as the parameter space of $A$, where $\cals_+^\sfm$ denotes 
 		 the set of all $\sfm$-dimensional positive semi-definite matrices. 
 		
 		Consider a case  where $A^*$ may be degenerate. (Though we will assume in Condition [{\bf I4}] below that $f(X_t)A^*f^{\prime}(X_t)$ is non-degenerate for any $0\leq t \leq T$.)
 		Therefore, $A^*$ may be on the boundary of $\cala$.  
 		For simplicity, suppose that for some  $\delta>0$,
 		\bea\label{calnIto}
 		\big\{A \in \cala ; \|A-A^*\|<\delta \big\}\yeq \big\{A \in \cals_+^{\sf{m}} ;  \|A-A^*\|<\delta \big\}.
 		\eea
 		  For when the true value is in the interior of the parameter space, see e.g. Genon-Catalot and Jacod \cite{genon1993estimation} and Uchida and Yoshida \cite{uchida2013quasi}.
 		
 		We consider the following conditions.
 		\bd
 		\im[${\bf [I1]}$]  For every $p>1$, 
 		$\ds \sup_{0\leq t \leq T}E\big[|b_t|^p\big]< \infty$.
 		
 		\im[${\bf [I2]}$]  {The continuous process $X=(X_t)_{t \in [0, T]}$ admits a representation
 		\beas
 		X_t &=& X_0+\int_0^tc_sds+\int_0^td_sdw_s+\int_0^t\tilde d_sd\tilde w_s\qquad(t\in[0,T]),
 		\eeas
 		where $\tilde w$ is an $\sfr_1$-dimensional standard Wiener process  independent of $w$, and  $c$, $d$ and $\tilde d$ are progressively measurable processes taking values in $\bbR^{\sf{c}}$,  $\bbR^{\sf{c}} \otimes \bbR^\sfr$ and $\bbR^{\sf{c}} \otimes \bbR^{\sfr_1}$, respectively, satisfying that  for every $p>1$, \beas
 		\sup_{0\leq t \leq T}E\big[|X_0|^p+|c_t|^p + \|d_t\|^p+\|\tilde d_t\|^p\big] < \infty.
 		\eeas  }
 		
 		\im[${\bf [I3]}$] For some $C>0$, \beas
 		\big\|\partial_x^if(x) \big\| \leq C\big(1+|x|^C\big) \qquad(x \in \bbR^{\sf{c}}, i=0,1,2).
 		\eeas
 		
 		\im[${\bf [I4]}$]  
 		{ With probability $1$, the image of $X$ is contained in $\calx$, where $\calx\subset\bbR^{\sf{c}}$ is a closed subset, and $f(x)Af^{\prime}(x)$ is elliptic uniformly in $(x, A) \in \calx \times \cala$.
 		  (Note that $f^{\prime}(x)$ denotes the transposition of $f(x)$.)}
 		
 	 \im[${\bf [I5]}$] With probability $1$, for any $B\in \cals^{\sfm}$ with $B\neq O$,  $\sup_{0\leq t \leq T}\|f(X_t)B f^{\prime}(X_t)\|>0$, where $\cals^\sfm$ denotes the set of all $\sfm$-dimensional symmetric matrices
 		\ed
 		 Define $\sfp$ as $\sfp=\frac{\sfm(\sfm+1)}{2}$. For simplicity, we parametrize $A$ by $\theta$ as  $\theta=\psi(A)$, where $\psi $ is defined as (\ref{psi}). Then the parameter space $\Theta$ of  $\theta$ is determined as $\Theta=\psi(\cala)$.  Also, the true value $\theta^*$  of $\theta \in \Theta$ becomes $\psi(A^*)$.
 		 
 		 We also suppose that ${\rm Int(\Theta)}$\footnote{${\rm Int}(B)$ denotes the interior of $B$ for a subset $B$.} is a domain of $\bbR^\sfp$ satisfying $\overline{{\rm Int}(\Theta)}=\Theta$ and 
 		 	Sobolev's inequalities for embedding $W^{1, p}\big({\rm Int}(\Theta)\big) \hookrightarrow C(\Theta)$ for $p>\sfp$.\footnote{A sufficient condition for it is the weak cone condition defined by  Adams \cite{adams1977cone} as follows: a domain $\Omega \subset \bbR^{\sf{p}}$ is said to satisfy the weak cone condition if there exists a number $\delta>0$ such that  \beas \mu_{\sf{p}}\big(\{y \in R(x); \, |y-x|<1\}\big) \geq \delta \qquad(x \in \Omega),\eeas
 		 		where $\mu_{\sf{p}}$ denotes the Lebesgue measure on $\bbR^{\sf{p}}$, and $R(x)$ consists of all points $y \in \Omega$ such that the line segment joining $x$ to $y$ lies entirely in $\Omega$.  If $\Omega$ is bounded and convex, then the weak cone condition holds.
 	 		For example, when taking $\cala$ as $\cala=\big\{A \in \cals_+^\sfm ; \|A\|  \leq  R\big\}$ for  $R>0$ with $\|A^*\| <R$,  both  (\ref{calnIto}) and the weak cone condition on ${\rm Int}(\Theta)$ hold.} 

 		Observing the data $(X_\tj)_{j=0,1,...,n}$ with $\tj=jT/n$, we want to estimate  the true value $\theta^*$. 
 		For this purpose, we maximize the estimation function
 		\beas
 		\mathbb{\Psi}_n(\theta)=\log\prod_{i=1}^n\frac{1}{|S(X_{t_{i-1}}, \theta)|^{\frac{1}{2}}}\exp\bigg(-\frac{1}{2h}S(X_{t_{i-1}}, \theta)^{-1}\big[(\Delta_iY)^{\otimes2}\big]\bigg) \qquad(\theta \in \Theta), 
 		\eeas
 	with $\Delta_i Y = Y_{t_i}-Y_{t_{i-1}}$, where $S$   is defined as
 		\beas
 		S(x, \theta) &=& f(x) Af(x)^{\prime}
 	\qquad\big(x \in \bbR^{\sf{c}}, \theta \in \Theta,  A=\psi^{-1}(\theta)\big).
 		\eeas
 		Here $|S(x, \theta)|$ denotes $\det\big(S(x, \theta)\big)$.
 		Let $\hat\theta_n$ be a $\Theta$-valued random variable  maximizing $\mathbb\Psi_n(\theta)$ on $\Theta$. We derive the limit distribution of  the estimator $\hat\theta_n$ of $\theta^*$.
 		
 		Define $\bbY(\theta)$ as
 		\beas
 		\bbY(\theta)&=&-\frac{1}{2T}\int_0^T\bigg\{{\rm Tr}\big(S^{-1}(X_t, \theta)S(X_t, \theta^*)-I_\sfd\big)+\log\frac{|S(X_t, \theta)|}{|S(X_t, \theta^*)|}\bigg\}dt\\
 		&=&-\frac{1}{2T}\int_0^T\int_0^1k\bigg\| \big(kS(X_t, \theta)+(1-k)S(X_t, \theta^*)\big)^{-\frac{1}{2}}\big(S(X_t, \theta)-S(X_t, \theta^*)\big)\\
 		&&\big(kS(X_t, \theta)+(1-k)S(X_t, \theta^*)\big)^{-\frac{1}{2}}\bigg\|^2dkdt
 		\qquad(\theta \in \Theta).
 		\eeas
 		For any $D=(d_{ij})_{1\leq i,j \leq \sfr} \in \bbR^\sfr \otimes \bbR^\sfr$, the ${\sfr^2}$-dimensional vector $(d_{11}, ..., d_{1\sfr}, d_{21}, ..., d_{\sfr\sfr})^{\prime}$ is denoted by $v(D)$.  Also, we denote $\{v(D)\}^{\prime}$ by $v^\prime(D)$.
 		Then define an $\bbR^{\sfp} \otimes \bbR^{\sfr^2}$-valued $\F$-adapted continuous process $\rho=(\rho_t)_{t \in [0, T]}$ as
 		\beas
 		u \rho_t\yeq \frac{1}{\sqrt{2}}v^{\prime}\big(\sigma^{\prime}(X_t, a^*)\,\partial_\theta S^{-1}(X_t, \theta^*)\,\sigma(X_t, a^*)[u]\big)
 		\qquad(u \in \bbR^\sfp).
 		\eeas
 		Denote by $W=(W_t)_{t\geq 0}$  an $\sfr^2$-dimensional Wiener process that is independent of $\calf$ and defined on an extension of $(\Omega, \calf, P)$.
 		Define a random field $\bbZ$ as 
 		\begin{eqnarray*}
 			\bbZ(u)=\Delta(\theta^*)[u]-\frac{1}{2}\Gamma(\theta^*)[u^{\otimes2}] \qquad (u \in \bbR^\sfp),
 		\end{eqnarray*}
 		where \beas
 		\Delta(\theta^*)&=&\frac{1}{\sqrt{T}}\int_0^T \rho_t dW_t,\\
 		 \Gamma(\theta^*)[u, v]&=&\frac{1}{T}\int_0^T \rho_t^{\otimes2}[u, v] dt\yeq \frac{1}{2T}\int_0^T{\rm Tr}\big(S^{-1}\partial_\theta S\, S^{-1} \partial_\theta S(X_t, \theta^*) [u, v]\big)dt
 		 	\eeas
 		 	  for any  $u, v \in \bbR^\sfp$. Also, define a subset $U$ of $\bbR^\sfp$ as  (\ref{ExpositivesemidefiU}), that is,
 		\beas
 		U \yeq \psi\bigg(\big\{ w \in \cals^{\sfm} ; K^\prime w K \in \cals_+^{\sfm-\sfr^*}\big\}\bigg),
 		\eeas
 		where $K $ is a $\sfm \times (\sfm-\sfr^*)$  matrix whose column vectors form a basis for ${\rm Ker}(A^*)$. \big(If $\sfr^*=\sfm$, then  consider $U$ just as $\psi(\cals^\sfm)$.\big)
 		
 		{\colorr Note that under [{\bf I5}], with probability one, for any $u \in \bbR^\sfp\setminus\{0\}$,
 		\bea\label{Gammanon}
 	0<\sup_{t \in [0, T]}\big\|f(X_t)\psi^{-1}(u)f^\prime(X_t)\big\|^2=\sup_{t \in [0, T]}\big\|S(X_t, u)\big\|^2=\sup_{t \in [0, T]}\big\|\partial_\theta S(X_t, \theta^*)[u]\big\|^2.
 \eea Then, with probability $1$, $\Gamma(\theta^*)$ is non-degenerate. Indeed,   $\Gamma(\theta^*)[u^{\otimes2}]=0$ implies $\partial_\theta S(X_t, \theta^*) [u]=0$ for any $t \in [0, T]$ since $\ds \Gamma(\theta^*)[u^{\otimes2}]=T^{-1}\int_0^T \big\|S^{-1/2}\partial_\theta S\, S^{-1/2} (X_t, \theta^*) [u]\big\|^2dt$.  This contradicts (\ref{Gammanon}) when $u\neq 0$.
 
 Similarly, with probability $1$, for any $\theta\in \Theta$ with $\theta\neq \theta^*$,  $\bbY(\theta)<0$.
 		From the convexity of $U$ and the convexity of $\bbZ$, Condition [{\bf A4}] holds.} 
 		Define $\hat{u}$ as the $U$-valued random variable defined in [{\bf A4}] which becomes the unique maximizer of $\bbZ$ on $U$ almost surely. Then we obtain the asymptotic distribution of $\hat \theta_n$. The proof of the following theorem is written in Section \ref{section6}. 
 		\begin{theorem}\label{ex2thm}  Assume $(\ref{calnIto})$ and $[{\bf I1}]$-$[{\bf I5}]$. Then for the maximizer $\hat \theta_n$ of $\mathbb{\Psi}_n$,
 		\begin{eqnarray}\label{ex2hatucon}
 		\sqrt{n}(\hat\theta_n - \theta^*) \overset{d_s(\calf)}{\rightarrow} \hat{u}.
 		\end{eqnarray}
 
 	\end{theorem}
 		
 		{Also, consider the same situation above, and  parametrize $\{b_t\}_{t \in [0, T]}$ in (\ref{SDE}) by $\gamma \in {\rm \ol G}$ as $b_t = g(X_t, \gamma)$, 
 	 where ${\rm G}$ is a domain in $\bbR^\sfq$ admitting Sobolev's inequalities for embedding $W^{1, p}({\rm G}) \hookrightarrow C(\ol G)$ for $p>\sfq$, and $g : \bbR^{\sf{c}} \times \bbR^\sfq  \to \bbR^\sfd$ is a continuous function.
 		Thus, the stochastic differential equation (\ref{SDE}) is parametrized by $(A, \gamma) \in \cala \times {\rm \ol G}$. In addition to  [{\bf I1}]-[{\bf I5}], we consider the following condition.
 	r  \bd
 		\im[${\bf [I6]}$]  The function $g$ has continuous derivatives satisfying that  for some $C>0$, \beas
 		\sup_{\gamma \in {\rm G}}\big\|\partial_\gamma^i g(x, \gamma) \big\| \leq C\big(1+|x|^C\big)  \qquad(x \in \bbR^{\sf{c}}, i=0, 1).
 		\eeas
 		\ed} 
 	 As before, we parametrize $A$ by $\theta=\psi(A)$. When estimating  $\theta^*\in \Theta$, the drift parameter $\gamma$ can be considered as nuisance. Therefore, we denote $\gamma$ by $\tau$. Then the parameter space $\Xi =\Theta \times \calt$  is determined as $\Theta  \times \calt =\psi(\cala) \times {\rm \ol G}$.  
 		Instead of $\mathbb{\Psi}_n$, we maximize the estimation function
 		\beas
 		\widetilde{\mathbb{\Psi}}_n(\theta, \tau)=\log\prod_{i=1}^n\frac{1}{|S(X_{t_{i-1}}, \theta)|^{\frac{1}{2}}}\exp\bigg(-\frac{1}{2h}S(X_{t_{i-1}}, \theta)^{-1}\big[(\Delta_iY-hg(X_{t_{i-1}}, \tau))^{\otimes2}\big]\bigg) \qquad\big((\theta, \tau) \in \Xi\big).
 		\eeas
 		For each $n$, take an arbitrary $\calt$-valued random variable $\hat\tau_n$, and  take $\hat\theta_n$ again as  a $\Theta$-valued random variable  maximizing $\widetilde{\mathbb\Psi}_n(\cdot, \hat\tau_n)$ on $\Theta$.  Also, define a random field $\bbV_n$ as
 		\begin{eqnarray*}\bbV_n(u)=\exp\bigg\{\frac{1}{\sqrt n}\partial_\theta \widetilde{\mathbb{\Psi}}_n(\theta^*, \hat\tau_n)[u]-\frac{1}{2}\Gamma(\theta^*)[u^{\otimes2}]\bigg\} \qquad(u\in\bbR^\sfp).
 		\end{eqnarray*}
 		Since $\Gamma(\theta^*)$ is non-degenerate under [{\bf I5}], we can  take  a $U$-valued random variable  $\hat v_n$ which is the unique maximizer of $\bbV_n$ on $U$ almost surely.
 		\begin{theorem}\label{SDEtaudisappear}  Assume $(\ref{calnIto})$ and $[{\bf I1}]$-$[{\bf I6}]$. Then for the maximizer $\hat\theta_n$ of $\widetilde{\mathbb{\Psi}}_n(\cdot, \hat\tau_n)$,
 			\begin{eqnarray}\label{ex2hatu}
 			\hat{u}_n:=\sqrt{n}(\hat\theta_n - \theta^*) \overset{d_s(\calf)}{\rightarrow} \hat{u}.
 			\end{eqnarray}
 			Also,
 			\begin{eqnarray}\label{ex2hatv_n}
 				\hat u_n-\hat v_n=o_P(1).
 			\end{eqnarray}
 		\end{theorem}
 	The proof  is written in Section \ref{section6}.
 	\begin{rem}\label{remIto}
 		{\rm \bd 
 			\im[(i)]{\colorr Even if we do not assume (\ref{calnIto}) and the form of $\cala$ changes, we could  obtain the same results by deriving the corresponding $U$.}
 			\im[(ii)]Consider the case where  $T \to \infty$ and the ergodicity of $X$ holds.  Let $(\hat \theta_n, \hat \tau_n)$ be the joint maximizer of $\widetilde{\mathbb{\Psi}}$.
 		If there exists the true value $\gamma^* \in {\rm \ol G}$ and if  $nh^2 \to 0$, then under suitable assumptions, we also obtain the weak  convergence  of $\hat w_n := \sqrt{nh}(\hat \tau_n-\gamma^*)$. Given this background,  we use
 		$\widetilde {\mathbb{\Psi}}_n$ as the estimation function rather than  $\mathbb{\Psi}_n$ here. To obtain the joint convergence of  $(\hat u_n, \hat w_n)$, it is sufficient to drive the convergence of $(\hat v_n,  \hat w_n)$.
 	\ed}
 		\end{rem}


\section{Penalized quasi-maximum likelihood estimation}
\subsection{Settings}
In this section, we will apply Theorem \ref{thm1} to   a penalized quasi-likelihood function in a regular case except that the true value $\theta^*$ may lie on the boundary.
 Consider the same situation as in Section 2.1. 
 	Suppose that  $\Theta$ and $\calt$ are compact in $\bbR^\sfp$ and $\bbR^\sfq$, respectively.
 Consider the following penalty term. 
 \beas
 \sum_{j\in \calj}\xi_{j, T} p_j(\theta_j)+s_T(\tau) \qquad\big((\theta_1, ..., \theta_\sfp) \in \Theta, \tau \in \calt\big),
 \eeas 
 where  $s_T : \Omega\times \bbR^\sfq \to [0, \infty)$ is a continuous random field depending on $T \in \bbT$,  $\calj$ is a subset of $\{1, ..., \sfp\}$,  $\xi_{j, T} : \Omega \to [0, \infty)$ $(j \in \calj)$ are random variables depending on $T\in \bbT$,
 and $p_j : \bbR \to [0, \infty)$ $(j \in \calj)$ are deterministic continuous functions satisfying the following conditions.
 \bd
 \im[(i)]  $p_{j}\big|_{\bbR\setminus\{0\}}$ is of class $C^1$
 
 \im[(ii)] For any  $x \in \bbR$, $p_{j}(x)=0$ if and only if $x=0$.

 \im[(iii)] There exist some  positive constants $\delta_0$ and {$q_j$}   such that
 \beas
 p_{j}(x) =|x|^{q_j} \qquad\big(x \in [-\delta_0, \delta_0]\big).
 \eeas
 
 \ed
 We consider penalized quasi-maximum likelihood estimation  and the random field $\bbH_T$ is given by 
 \begin{eqnarray*}
 	\bbH_T(\theta, \tau)=\calh_T(\theta, \tau)-\sum_{j\in \calj}\xi_{j, T} p_j(\theta_j)-s_T(\tau) \qquad\big((\theta, \tau)\in \Xi\big),
 \end{eqnarray*}
 where for each $T\in\bbT$, $\calh_T: \Omega\times \Xi\rightarrow \bbR$ is the continuous random field corresponding to the given quasi-log likelihood function.  Recall that for each $T \in \bbT$, $\hat \tau_T$ is a given $\calt$-random variable and $\hat \theta_T$ is a $\Theta$-valued random variable that asymptotically maximizes $\bbH_T(\cdot, \hat\tau_T)$ on $\Theta$. 
 
 For each $j=1, ..., \sfp$, we denote by $\theta^*_j$ the $j$-th component of $\theta^*$. We decompose $(\theta^*_j)_{j \in \calj}$ as 
 \begin{eqnarray*}
 	\theta^*_{i}\neq 0    ~~and~~\theta^*_{k}=0\qquad(i\in \calj_1, k\in\calj_0),
 \end{eqnarray*}
 where $\calj_1$ and $\calj_0$ are two partitions of $\calj$. 

 \subsection{Asymptotic behavior of the PQMLE}
 Let $\frak a_{j, T}$ $(j=1, ..., \sfp)$ be  positive numbers depending on  $T \in \bbT$ satisfying that $\frak a_{j, T} \to 0$. Define a $\sfp \times \sfp$ matrix $\frak a_T$ as  $\frak a_T={\rm diag}(\frak a_{1, T}, ..., \frak a_{\sfp, T})$.  
 Define a positive sequence $\frak b_T$ as 
 \beas \frak b_T=\lambda_{\rm min}\big[(\frak a_T^{\prime }\frak a_T)^{-1}\big]=\min_{j \in \calj}\frak a_{j, T}^{-2}.
 \eeas 
 Define a continuous random field $\caly_T : \Omega\times\Xi\rightarrow\bbR$ as 
 \beas\caly_T(\theta, \tau)=\frac{1}{\frak b_T}\big(\calh_T(\theta, \tau)-\calh_T(\theta^*, \tau)\big)
 \qquad\big((\theta, \tau)\in\Xi\big).\eeas
 Let $\caly: \Omega\times\Theta\rightarrow \bbR$ be a  continuous random field, and $\caln$ a bounded open set in $\bbR^{\sfp}$ satisfying Conditions {\bf (i)} and {\bf (ii)} given in Section \ref{Quasisection1.1}.
 We suppose that $\calh_T$ can be extended to a continuous random field defined on $\Omega\times \ol \caln\times\calt$  satisfying that for every $\omega\in\Omega$ and $\tau\in\calt$, $\calh_T(\omega, \cdot, \tau)$ is of class $C^2(\ol \caln)$. Let ${\Delta}({\theta^*})$ be an $\bbR^{\sfp}$-valued random variable, and let $ \Gamma(\theta^*)$ be a $\calg$-measurable  $ \bbR^{\sfp}\otimes\bbR^{\sfp}$-valued random variable, where $\calg \subset \calf$ is a sub-$\sigma$-field. Let $\rho_k \geq 1$ $(k \in \calj_0)$ be positive numbers.
 Also, let $c_i : \Omega \to [0, \infty)$ $(i \in \calj_1)$ and $d_k : \Omega \to {[0, \infty)}$ $(k \in \calj_0)$  non-negative random variables.

For [{\bf A1}] and [{\bf A2}], we consider the following conditions.

 \bd
 \im[${\bf [C1]}$]
 $\ds \sup_{(\theta, \tau)\in\Xi}|\caly_T(\theta, \tau)-\caly(\theta)|\overset{P}{\rightarrow}0.$ Also, with probability $1$, for any $\theta \in \Theta$ with $\theta\neq \theta^*$,
 \beas\caly(\theta) < 0.
 \eeas 

 \im[${\bf [C2]}$] For some $\tau_0 \in \calt$,
 \bea
 \sup_{\tau \in \calt}\big|\frak a_T\partial_\theta\calh_T(\theta^*, \tau)-\frak a_T\partial_\theta\calh_T(\theta^*, \tau_0)\big|&\overset{P}{\to}& 0,\label{C2(i)}\\
 	\big(\frak a_T\partial_\theta\calh_T(\theta^*, \tau_0), ~  (\frak a_{i, T}\xi_{i, T})_{i \in \calj_1}, ~(\frak a_{k, T}^{q_k \rho_k}\xi_{k, T})_{k \in \calj_0} \big) &\overset{d_s(\calg)}{\rightarrow}&\bigg(\Delta({\theta^*}), (c_i)_{i \in \calj_1}, (d_k)_{k \in \calj_0} \bigg),\label{C2(ii)}
 	\eea  in $\bbR^\sfp \times \bbR^{|\calj_1|} \times \bbR^{|\calj_0|}$.
 	Also, for any positive sequence $\delta_T$ with $\delta_T\rightarrow0$,
 	\beas\sup_{\substack{(\theta, \tau)\in \ol\caln\times\calt\\|\theta-\theta^*|<\delta_T}}\big\|\frak a_T\partial^2_\theta\calh_T(\theta, \tau)\frak a_T+ \Gamma(\theta^*) \big\|\overset{P}{\rightarrow}0.
 	\eeas

\im[${\bf [C3]}$]  $\Gamma(\theta^*)$ is almost surely positive definite.

 \im[${\bf [C4]}$]
 For each $i \in \calj_1$, 
 \beas
 \xi_{i, T} \frak b_T^{-1}\overset{P}{\to} 0.
 \eeas

\im[${\bf [C5]}$]
For any $k \in \calj_0$, if $\rho_k >1$, then 
$\ds d_k > 0~~a.s.$

 \ed
  Take $a_T \in GL(\sfp)$ as a deterministic diagonal matrix  defined by \bea
 (a_T)_{jj}=\left\{\begin{array}{ll} 	\frak a_{j, T} &\big(j \in \{1, ..., \sfp\}\setminus\calj_0\big) \\ \frak a_{j, T}^{\rho_j}&(j\in \calj_0) 
 \end{array}\right.. \label{a_T}
 \eea  
Also, define a diagonal $\sfp \times \sfp $ matrix $b$ as
\begin{eqnarray*}
	(b)_{jj}=\lim_{T \to \infty}\frac{(a_T)_{jj}}{(\frak a_T)_{jj}} =\left\{\begin{array}{ll}1 &\big(j \in \{1, ..., \sfp\}\setminus\calj_0\big)\\  1_{\{\rho_j=1\}}& (j\in\calj_0)\end{array}\right..
\end{eqnarray*} 
Recall that $U_T$ and $U$ are defined by (\ref{U_T}) and (\ref{U}), respectively and that  $\hat{u}_T$ is defined  by $\hat{u}_T= (a_T)^{-1}(\hat{\theta}_T - \theta^*)$. 
 
 \begin{theorem}\label{thmpenalized}
 	Assume {\rm [{\bf C1}]-[{\bf C5}]} and $[{\bf A3}]$. Also, assume $[{\bf A4}]$  for the continuous random field $\bbZ$ defined as for any $u \in \bbR^\sfp$,
 	\begin{eqnarray}
 	\bbZ(u)&=&\exp\bigg\{\Delta(\theta^*)[bu]-\frac{1}{2} \Gamma(\theta^*)\big[{(bu)}^{\otimes2}\big]-\sum_{i \in \calj_1}c_i\frac{d}{dx}p_i (\theta^*_i)u_i-\sum_{k \in \calj_0}d_k|u_k|^{q_k}\bigg\}\label{thmpenalizedcalz}.
 	\end{eqnarray}
 	Then  for the $U$-valued random variable $\hat u$ defined in $[{\bf A4}]$, 
 	\beas
 	\hat u_T \overset{d_s(\calg)}{\to} \hat u.
 	\eeas
 	Moreover, assume $[{\bf A5}]$ for the continuous random field $\bbV_T$ defined as
 	\begin{eqnarray}
 	\bbV_T(u)&=&\exp\bigg\{ \Delta_T(\theta^*, \hat\tau_T)[bu]-\frac{1}{2} \Gamma(\theta^*)\big[{(bu)}^{\otimes2}\big]-\sum_{i \in \calj_1}\frak a_{i, T}\xi_{i, T}\frac{d}{dx}p_i (\theta^*_i)u_i\nonumber\\
 	&&-\sum_{k \in \calj_0}\frak a_{k, T}^{{q_k\rho_k}}\xi_{k, T} |u_k|^{q_k}\bigg\}
 	\qquad(u \in \bbR^\sfp),\label{thmpenalizedcalv_T}
 	\end{eqnarray}
 	where $\Delta_T(\theta^*, \cdot)$ represents $\frak a_T\partial_\theta\calh_T(\theta^*, \cdot)$.
 	Then for the $U$-valued random variable $\hat v_T$ defined in $[{\bf A5}]$, 
 	\beas
 	\hat u_T-\hat v_T =o_P(1).
 	\eeas
 \end{theorem}
 
\begin{remark}  {{\rm Similarly as  Remark \ref{remarkB}, under  [{\bf C1}] and [{\bf C4}], the following condition holds.
 			\bd
 			\im[${\bf [C1]^{\flat}}$] For any $\delta>0$, \beas 
 			\varlimsup_{T \to \infty}P\bigg[\sup_{\substack{(\theta, \tau)\in \Theta\times\calt\\|\theta-\theta^*|\geq \delta}}\big\{\bbH_T(\theta, \tau)-\bbH_T(\theta^*, \tau)\big\} \geq 0\bigg] \yeq 0.
 			\eeas
 			\ed
 			Then Theorem \ref{thmpenalized} holds even if we substitute [{\bf C1}] and [{\bf C4}] with $[{\bf C1}]^{\flat}$. (See the following proof.)}  }
 	
 	\end{remark}

\begin{proof}[Proof of Theorem \ref{thmpenalized}]
	From Theorem \ref{thm1}, it suffices to show that [{\bf C1}]-[{\bf C5}] imply [{\bf A1}] and [{\bf A2}]. Define  $r_T$ as 
	\begin{eqnarray*}
		r_T(u, \tau)&=&-2\int_0^1(1-k)\big\{\frak a_T \partial^2_\theta\calh_T(\theta^*+ka_Tu, \tau)\frak a_T+\Gamma(\theta^*)\big\}dk \\ 		
	\end{eqnarray*} for any $(u, \tau) \in U_T\times\calt$.
	Then,  from Taylor's series, for any $(u, \tau) \in U_T\times\calt$,
	\begin{eqnarray}
		\bbZ_T(u, \tau)&=&\exp\bigg\{\Delta_T(\theta^*, \tau)[\frak a_T^{-1} a_Tu]-\frac{1}{2}\big(\Gamma(\theta^*)+r_T(u, \tau)\big)[(\frak a_T^{-1} a_T u)^{\otimes2}]\nonumber\\ &&-\sum_{i \in \calj_1}\xi_{i, T}\big( p_i(\theta^*_i+\frak a_{i,T}u_i)-p_i(\theta^*_i) \big)-\sum_{i \in \calj_0}\xi_{k, T} p_k(\frak a_{k,T}^{\rho_k}u_k)\bigg\}. \label{thmpenalizedTaylor}
	\end{eqnarray}
	Note that from [{\bf C2}], for any positive sequence $\delta_T$ with $\delta_T\to0$,
	\beas
	&&\sup_{(u, \tau)\in U_T\times \calt, |a_Tu|<\delta_T}\big\|r_T(u, \tau)\big\|=o_P(1),\\
		&&\sup_{u\in U_T, |a_Tu|<\delta_T}\bigg|\sum_{i \in \calj_1}\xi_{i, T}\big( p_i(\theta^*_i+\frak a_{i,T}u_i)-p_i(\theta^*_i)\big) -\sum_{i \in \calj_1}\frak a_{i, T}\xi_{i, T}\frac{d}{dx}p_i (\theta^*_i)u_i\bigg|=o_P(1),\\
&&\sup_{u\in U_T, |a_Tu|<\delta_T}\bigg|\sum_{k \in \calj_0}\xi_{k, T}p_k(\frak a_{k, T}^{\rho_k} u_k)-\sum_{k \in \calj_0}\frak a_{k, T}^{{q_k\rho_k}}\xi_{k, T} |u_k|^{q_k}\bigg|\to 0.
	\eeas
	Therefore, from [{\bf C2}], [{\bf A2}] obviously holds if we take $\bbZ$ and $\bbV_T$ as (\ref{thmpenalizedcalz}) and (\ref{thmpenalizedcalv_T}), respectively.
	
	We show that $[{\bf A1}]$ also  holds.  Let  $\epsilon$ and $R$ be positive numbers. Besides, let $\delta_T$ be a positive sequence with $\delta_T\to0$. Then
	\bea&&\varlimsup_{T\rightarrow \infty}
	P\bigg[\sup_{U_T\times \calt,  |u| \geq R} \mathbb{Z}_T(u, \tau)  \geq 1 \bigg]\nonumber\\
	&\leq&\varlimsup_{T\rightarrow \infty}
	P\bigg[\sup_{U_T\times\calt,  |a_T u|\geq \delta_T} \mathbb{Z}_T(u, \tau)  \geq 1 \bigg]
	+
	\varlimsup_{T\rightarrow \infty}
	P\bigg[\sup_{\substack{U_T\times\calt,  |u| \geq R\\|a_T u|<\delta_T}} \mathbb{Z}_T(u, \tau)  \geq 1 \bigg].\label{Bright} \eea
	From [{\bf C1}] and [{\bf C4}], we can evaluate the first term in (\ref{Bright}) as
	\bea
	&&\varlimsup_{T\rightarrow \infty}
	P\bigg[\sup_{U_T\times\calt,  |a_T u|\geq \delta_T} \mathbb{Z}_T(u, \tau)  \geq 1 \bigg]\nonumber\\
	&=&
	\varlimsup_{T\rightarrow \infty}
	P\bigg[\sup_{\Theta\times\calt,  |\theta-\theta^*|\geq \delta_T} \bigg\{\bbH_T(\theta, \tau)-\bbH_T(\theta^*, \tau)\bigg\}  \geq 0 \bigg]\label{C1flat}\\
	&\leq&
	\varlimsup_{T\rightarrow \infty}
	P\bigg[\sup_{\Theta\times\calt,  |\theta-\theta^*|\geq \delta_T} \bigg\{\caly(\theta)+\caly_T(\theta, \tau)-\caly(\theta)+\frak b_T^{-1}\sum_{i \in \calj_1}\xi_ip_i(\theta^*_i)\bigg\}  \geq 0 \bigg]\nonumber\\
	&\leq&
	\varlimsup_{T\rightarrow \infty}
	P\bigg[\sup_{\Theta\times\calt,  |\theta-\theta^*|\geq \delta_T}\caly(\theta)+o_P(1)  \geq 0 \bigg],\nonumber
	\eea where $o_P(1)$ satisfies $\sup_{\Theta \times \calt} \big|o_P(1)\big| \to^P 0$.
	(When substituting [{\bf C1}] and [{\bf C4}] with $[{\bf C1}]^{\flat}$, stop at (\ref{C1flat}).) Thus,  by  choosing  $\delta_T$ properly and using [{\bf C1}] (or $[{\bf C1}]^{\flat}$), we have
	\beas
	\varlimsup_{T\rightarrow \infty}
	P\bigg[\sup_{U_T\times\calt,  |a_T u|\geq \delta_T} \mathbb{Z}_T(u, \tau)  \geq 1 \bigg]\yeq 0.
	\eeas
	
	We evaluate the second term in (\ref{Bright}). Let $\epsilon>0$. From (\ref{thmpenalizedTaylor}) and [{\bf C2}], 
  for any $(u, \tau) \in U_T\times \calt$ with $|a_Tu|< \delta_T$ and for any $\omega\in\Omega$ belonging to a set
	 \beas
	 \bigg\{\frak a_{k,T}^{q_k\rho_k}\xi_{k, T} \geq \epsilon~( k\in\calj_0,  \rho_k>1)\bigg\}\cap \bigg\{\lambda_{\min}\big[\Gamma(\theta^*)\big]\geq \epsilon\bigg\},\eeas 
\begin{eqnarray}
	\bbZ_T(u, \tau)\nonumber&= &\exp\bigg\{O_P(1)\big|\frak a_T^{-1} a_Tu\big|-\frac{1}{2}\big(\Gamma(\theta^*)+o_P(1)\big)\big[(\frak a_T^{-1} a_T u)^{\otimes 2}\big]\nonumber\\&&-\sum_{i \in \calj_1}O_P(1)u_i-\sum_{k \in \calj_0}\frak a_{k,T}^{q_k\rho_k}\xi_{k, T} |u_k|^{q_k}\bigg\}\nonumber\\
	&= &\exp\bigg\{O_P(1)\big|\frak a_T^{-1} a_Tu\big|-\frac{1}{2}\big(\Gamma(\theta^*)+o_P(1)\big)\big[(\frak a_T^{-1} a_T u)^{\otimes 2}\big]-\sum_{k \in \calj_0}\frak a_{k,T}^{q_k\rho_k}\xi_{k, T} |u_k|^{q_k}\bigg\}\label{thmpenalizedasymptotic}\\
	&&
	\bigg(\because \,O_P(1)\big|\frak a_T^{-1} a_Tu\big|-\sum_{i \in \calj_1}O_P(1)u_i=O_P(1)\big|\frak a_T^{-1} a_Tu\big|\bigg)\nonumber\\
	&\leq& \exp\bigg\{O_P(1)\big|\frak a_T^{-1} a_Tu\big|-\bigg(\frac{\epsilon}{2}+o_P(1)\bigg)\big|\frak a_T^{-1} a_T u\big|^2\nonumber-\epsilon \sum_{k \in \calj_0}\frak  |u_k|^{q_k}1_{\{\rho_k>1\}}\bigg\}\nonumber\\
	&\leq&\left\{\begin{array}{lll}\exp\big\{O_P(1)\big|\frak a_T^{-1} a_Tu\big|-\big(\frac{\epsilon}{2}+o_P(1)\big)\big|\frak a_T^{-1} a_T u\big|^2\big\}\\
		\\
		\exp\big\{O_P(1)-\epsilon \sum_{k \in \calj_0} |u_k|^{q_k}1_{\{\rho_k>1\}} \big\}\end{array}\right.,\label{thmpenalizedA1}
\end{eqnarray}
where we use  the evaluation $\sup_{u\in\bbR^\sfp}\big\{O_P(1)\big|\frak a_T^{-1} a_Tu\big|-\big(\frac{\epsilon}{2}+o_P(1)\big)\big|\frak a_T^{-1} a_T u\big|^2\big\}=O_P(1)$.
	Since $|u|\leq |\frak a_T^{-1}a_T u|+\sum_{k \in \calj_0} |u_k|1_{\{\rho_k>1\}}$, (\ref{thmpenalizedA1}) converges to $0$ as $|u| \to \infty$. Therefore, we have
	\begin{eqnarray*}
		&&\varlimsup_{T\rightarrow \infty}
		P\bigg[\sup_{\substack{U_T\times\calt,  |u| \geq R\\|a_T u|<\delta_T}} \mathbb{Z}_T(u, \tau)  \geq 1 \bigg] \\
		&\leq & 
		\varlimsup_{T\rightarrow \infty}
		P\bigg[\sup_{\substack{U_T\times\calt,  |u| \geq R\\|a_T u|<\delta_T}} \mathbb{Z}_T(u, \tau)  \geq 1,  	\frak a_{k,T}^{q_k\rho_k}\xi_{k, T} \geq \epsilon ~(k\in\calj_0,  \rho_k>1),  ~\lambda_{\min}\big[\Gamma(\theta^*)\big]\geq \epsilon\bigg]\\ &&+\varlimsup_{T\rightarrow \infty}
		P\bigg[	\frak a_{k,T}^{q_k\rho_k}\xi_{k, T}\leq  \epsilon ~(k\in\calj_0,  \rho_k>1)\bigg] +P\bigg[\lambda_{\min}\big[\Gamma(\theta^*)\big]\leq \epsilon \bigg]\\
		&\overset{R\to\infty}{\to}&
		\varlimsup_{T\rightarrow \infty}
		P\bigg[	\frak a_{k,T}^{q_k\rho_k}\xi_{k, T}\leq  \epsilon~( k\in\calj_0,  \rho_k>1)\bigg] +P\bigg[\lambda_{\min}\big[\Gamma(\theta^*)\big]\leq \epsilon \bigg]\\
		&\leq&P\bigg[	d_k\leq  \epsilon ~( k\in\calj_0,  \rho_k>1)\bigg] +P\bigg[\lambda_{\min}\big[\Gamma(\theta^*)\big]\leq \epsilon \bigg]\\
		&\overset{\epsilon\rightarrow0}{\rightarrow}&0 \qquad\big(\because \text {[{\bf C3}] and [{\bf C5}]}\big).
	\end{eqnarray*}
	Thus, $[{\bf A1}]$ holds.
\end{proof}

\subsection{Selection consistency}

Denote by $\hat u_{j, T}$ the $j$-th component of $\hat u_T$  $(j=1, ..., \sfp)$.  When we can show that  $\hat u_{k, T} \overset{P}{\to} 0$ $(k\in \calj_0)$ using Theorem \ref{thmpenalized}, we can also show selection consistency under the following conditions [{\bf S1}] and [{\bf S2}]. Note that  for $u=(u_1, ..., u_\sfp) \in \bbR^\sfp$ and $v=(v_1, ..., v_\sfp) \in \bbR^\sfp$, we denote $(u_j)_{j \in \{1, ..., \sfp\}\setminus\calj_0}$,   $(u_j)_{j \in \calj_0}$,  $(v_j)_{j \in \{1, ..., \sfp\}\setminus\calj_0}$ and $(v_j)_{j \in \calj_0}$ by $\overline u$, $\underline u$,  $\overline v$ and $\underline v$, respectively. Define $\calj_{0, +} \subset \calj_0$ as
\beas
\calj_{0, +} \yeq \{k \in \calj_0 ; d_k >0 ~~a.s.\}.
\eeas
\bd
\im[{$\bf [S1]$}] $\calj_{0, +}$ is not empty, and for any $R>0$ and any positive sequence $\delta_T$ with $\delta_T\rightarrow 0$ as $T\rightarrow\infty$, 
\beas
\sup_{u\in S(R, \delta_T)}\inf_{v \in I(R)}\frac{|\overline u-\overline v|+|(u_k)_{k \in \calj_0, \rho_k=1}|}{\sum_{k \in \calj_{0, +}}|u_k|^{q_k}}\rightarrow 0 \qquad(T\rightarrow \infty),
\eeas
where \beas S(R, \delta_T)&=&\{u \in U_T ; |\overline u|\leq R, ~0<|\underline u|\leq\delta_T\}~~~and\\
I(R)&=&\{v\in U_T; |\overline v|\leq R, ~|\underline v|=0\}.
\eeas

\im[{$\bf[S2]$}] For any $k \in \calj_{0, +}$, $q_k  \leq 1$.
\ed

\begin{theorem}\label{Selection1}
	Assume that {\rm [{\bf C1}]-[{\bf C5}]}, {\rm [{\bf S1}]} and {\rm [{\bf S2}]} hold and that 
	\begin{eqnarray*}\hat u_{k, T} \overset{P}{\rightarrow} 0\qquad(k\in \calj_0).
	\end{eqnarray*}
	Then 
	\begin{eqnarray*}\lim_{T \to \infty}P\big[(\hat \theta_{k, T})_{k \in \calj_0}=0\big] = 1.
	\end{eqnarray*}
	
\end{theorem}
\begin{remark}
	{\rm Theorem \ref{Selection1} holds even if we  substitute [{\bf C1}] and [{\bf C4}] with $[{\bf C1}]^{\flat}$.}
\end{remark}
\begin{proof}[Proof of Theorem \ref{Selection1}]
	Take an arbitrary positive sequence $\epsilon_T$ with $\epsilon_T\rightarrow 0$ as $T\rightarrow\infty$.
	It suffices to show that \begin{eqnarray*}P\big[\big|(\hat u_{k, T})_{k \in \calj_0}\big|\geq \epsilon_T \big] \rightarrow 0.
	\end{eqnarray*}
	Since $\big|(\hat u_{k, T})_{k \in \calj_0}\big|=o_P(1)$, we can take a positive sequence $\delta_T$ with $\delta_T\rightarrow 0$ as $T\rightarrow\infty$ such that
	\beas
	\varlimsup_{T\rightarrow\infty}P\big[\big|(\hat u_{k, T})_{k \in \calj_0}\big|\geq \delta_T \big]=0.
	\eeas
		Then 
	\beas
	&&\varlimsup_{T\rightarrow\infty}P\big[\big|(\hat u_{k, T})_{k \in \calj_0}\big|\geq \epsilon_T \big]\\&\leq&
	\varlimsup_{T\rightarrow\infty}P\bigg[\sup_{\substack{u\in U_T, \tau\in\calt,\\  \delta_T \geq |\underline u|\geq\epsilon_T}}\bbZ_T(u, \tau)-\sup_{\substack{v\in U_T, \mu\in\calt, \\ \underline v=0}}\bbZ_T(v, \mu)\geq0\bigg] \\
	&=&\varlimsup_{R\rightarrow\infty}\varlimsup_{T\rightarrow\infty}
	P\bigg[\sup_{\substack{u\in S(R, \delta_T), \\\tau\in\calt, |\underline u|\geq\epsilon_T}}\bbZ_T(u, \tau)-\sup_{\substack{v\in I(R), \\ \mu\in\calt}}\bbZ_T(v, \mu)\geq0\bigg] \\
	&=&\varlimsup_{R\rightarrow\infty}\varlimsup_{T\rightarrow\infty}
	P\bigg[\sup_{\substack{u\in S(R, \delta_T), \\\tau\in\calt, |\underline u|\geq\epsilon_T}}\inf_{\substack{v\in I(R), \\ \mu\in\calt}}\big\{\bbZ_T(u, \tau)-\bbZ_T(v, \mu)\big\}\geq0\bigg] \\
	&=&\varlimsup_{R\rightarrow\infty}\varlimsup_{T\rightarrow\infty}
	P\bigg[\sup_{\substack{u\in S(R, \delta_T), \\\tau\in\calt, |\underline u|\geq\epsilon_T}}\inf_{\substack{v\in I(R), \\ \mu\in\calt}}\bigg\{O_P(1)\big(|\overline u-\overline v|+|(u_k)_{k \in \calj_0, \rho_k=1}|\big)+o_P(1)|(u_k)_{k \in \calj_0, \rho_k>1}|\\
	&&-\sum_{k\in \calj_0}\xi_{k, T}\frak a_{k, T}^{q_k\rho_k}|u_k|^{q_k}\bigg\}\geq0\bigg] \qquad\big(\because (\ref{thmpenalizedasymptotic})\big).
	\eeas
	 Since $\xi_{k, T}\frak a_{k, T}^{q_k\rho_k} \overset{d}{\to} d_k>0$ $(k \in \calj_{0, +})$ from [{\bf C2}], for any $k \in \calj_{0, +}$,
	\beas
	\varlimsup_{\eta \to +0}\varlimsup_{T \to \infty}P\big[\xi_{k, T}\frak a_{k, T}^{q_k\rho_k} \leq \eta\big]=0.
	\eeas
	Therefore,
	\beas
	&&\varlimsup_{T\rightarrow\infty}P\big[\big|(\hat u_{k, T})_{k \in \calj_0}\big|\geq \epsilon_T \big]\\&\leq&
	\varlimsup_{\eta \to +0}\varlimsup_{R\rightarrow\infty}\varlimsup_{T\rightarrow\infty}
	P\bigg[\sup_{\substack{u\in S(R, \delta_T), \\\tau\in\calt, |\underline u|\geq\epsilon_T}}\inf_{\substack{v\in I(R), \\ \mu\in\calt}}\big\{O_P(1)\big(|\overline u-\overline v|+|(u_k)_{k \in \calj_0, \rho_k=1}|\big)\\&&+o_P(1)|(u_k)_{k \in \calj_0, \rho_k>1}|-\eta\sum_{k\in \calj_{0, +}}|u_k|^{q_k}\big\}\geq0\bigg]\\
	&\leq&
	\varlimsup_{\eta \to +0}\varlimsup_{R\rightarrow\infty}\varlimsup_{T\rightarrow\infty}
	P\bigg[\sup_{\substack{u\in S(R, \delta_T), \\\tau\in\calt, |\underline u|\geq\epsilon_T}}\inf_{\substack{v\in I(R), \\ \mu\in\calt}}\big\{O_P(1)\frac{|\overline u-\overline v|+|(u_k)_{k \in \calj_0, \rho_k=1}|}{\sum_{k\in \calj_{0, +}}|u_k|^{q_k}}\\&&+o_P(1)\frac{|(u_k)_{k \in \calj_0, \rho_k>1}|}{\sum_{k\in \calj_{0, +}}|u_k|^{q_k}}-\eta\big\}\geq0\bigg].
	\eeas
	 Condition [{\bf S1}] implies that  for any $R>0$,
	\begin{eqnarray*}
		\sup_{\substack{u\in S(R, \delta_T), \\ |\underline u|\geq\epsilon_T}}\inf_{v\in I(R)}\frac{|\overline u-\overline v|+|(u_k)_{k \in \calj_0, \rho_k=1}|}{\sum_{k\in \calj_{0, +}}|u_k|^{q_k}}\rightarrow 0 \qquad(T\rightarrow\infty).
	\end{eqnarray*}
Also, since $\{k \in \calj_0 ; \rho_k >1\} \subset \calj_{0, +} $ from [{\bf C5}],  Condition [{\bf S2}] implies that  for any $R>0$,
\beas
\sup_{|u| \leq R }\frac{|(u_k)_{k \in \calj_0, \rho_k>1}|}{\sum_{k\in \calj_{0, +}}|u_k|^{q_k}} = O_P(1).
\eeas
	Thus, we have $\ds \varlimsup_{T\to\infty}P\big[\big|(\hat u_{k, T})_{k \in \calj_0}\big|\geq \epsilon_T \big] = 0$. Thus, Theorem  \ref{Selection1} holds.

\end{proof}
 
 The following example is a simple case where the selection consistency condition [{\bf S1}] holds. A more complex case is treated in Section \ref{linearmixedmodel}.

\begin{example}\label{exeasysparse}
	{\rm Take $\Theta$ as 
	\beas
	\Theta = \prod_{i=1}^\sfp I_i,
	\eeas
	where $I_i $ $(i=1, ..., \sfp)$ are subsets of $\bbR^\sfp$.  Assume that $\calj_0$ is not empty and  that for any $k \in \calj_0$, $\rho_k >1$. Then [{\bf S1}] obviously holds under [{\bf C5}]. In fact, $\calj_{0, +}$ is not empty from [{\bf C5}], and 
	for any $R>0$ and any positive sequence $\delta_T$ with $\delta_T\rightarrow 0$ as $T\rightarrow\infty$, 
	\beas
	\sup_{u\in S(R, \delta_T)}\inf_{v \in I(R)}\frac{|\overline u-\overline v|+|(u_k)_{k \in \calj_0, \rho_k=1}|}{\sum_{k \in \calj_{0, +}}|u_k|^{q_k}}&=& \sup_{u\in S(R, \delta_T)}\inf_{v \in I(R)}\frac{|\overline u-\overline v|}{\sum_{k \in \calj_{0, +}}|u_k|^{q_k}}\\
	&\leq & \sup_{u\in S(R, \delta_T)}\frac{|\overline u-\overline u|}{\sum_{k \in \calj_{0, +}}|u_k|^{q_k}}\\
	&=& 0.
	\eeas
	Thus, [{\bf S1}] holds. More generally, if  $\Theta$ can be decomposed as 
	\beas
	\Theta \yeq \bigg\{\theta \in \bbR^\sfp ; \big( (\theta_i)_{i \in \{1, ..., \sfp\} \setminus \calj_0}, (\theta_k)_{k \in \calj_0} \big)  \in A \times B \bigg\}
	\eeas
	for some $A \subset \bbR^{\sfp -|\calj_0|}$ and some $B \subset \bbR^{|\calj_0|}$,
	then the same argument goes.
 	}

	\end{example}

\subsection{Linear mixed model}\label{linearmixedmodel}
{We give an example where the Bridge estimator with $q<1$ does not show selection consistency.  
	Consider the  linear mixed model
	\beas
	&&Y \yeq   X \beta +  Z b + \epsilon, \\
	  &&b \sim N_{\sf{e}}(0, D), ~ \epsilon \sim N_{\sf{d}}(0, \sigma^2I_\sfd),
	\eeas
	where $X$ and $Z$ are   observable  random variables taking values in $\bbR^{\sf{d}}\otimes \bbR^{\sf{c}}$ and  $\bbR^{\sfd} \otimes \bbR^{\sf{e}}$, respectively. The nonrandom vector $\beta$ is an unknown parameter, 
	while $b$ is  an unobservable random effect that follows an $\sf{e}$-dimensional Gaussian distribution with mean zero and covariance matrix $D$. Besides, $(X, Z), b$ and $\epsilon$  are mutually independent. 
	
	Bondell, Krishna and Ghosh \cite{bondell2010joint} and Ibrahim et al.\,\cite{ibrahim2011fixed} consider penalized estimation using  Cholesky parametrizations for $D$, while  we use $D$ as is without re-parametrization. Not re-parametrizing $D$ is important not only since  it is more intuitive, but also since   the estimator  can be invariant under reordering of the parameter components.
	Let us  estimate the true value of the unknown parameters $\theta =\big (\beta,   \sigma^2, \psi(D)\big)$, where $\psi$ is defined in (\ref{psi}) when $\sfm=\sf{e}$.
	 Denote by $\Theta = \calb \times \Sigma \times \psi(\cald)  $ a parameter space of $\theta$, where ${\calb}$, $\Sigma$ and $\cald$ are compact subsets of $\bbR^{\sf{c}}$, $(0, \infty)$  and  $\cals_+^{\sf{e}}$, respectively. Denote  by $\theta^* =\big(\beta^*,  (\sigma^*)^2, \psi(D^*)\big)$ the true value of $\theta\in \Theta$. Suppose that $\beta^* \in {\rm Int}({\calb})$ 
	  and  $(\sigma^*)^2 \in {\rm Int}(\Sigma)$ and  that for some $\delta>0$,
	 \bea\label{Lineardelta}
	 \big\{D \in  {\cald}  ; \|D-D^*\|< \delta \big \} \yeq \big\{D \in  {\cals_+^{\sf{e}}}  ; \|D-D^*\|< \delta \big \}.
	 \eea

	 Let $\{X_i, Z_i, b_i, \epsilon_i\}_{i=1}^n$ be an $n$ independent copies of $\{X, Z, b, \epsilon\}$. We have $n$ couples of data $(Y_i, X_i, Z_i)$ $(i=1, ..., n)$. Then consider sparse estimation and maximize the estimation function
	 \beas
	 \mathbb{\Psi}_n(\theta) &=& -\frac{1}{2}\sum_{i=1}^n  
	 (Z_i D Z_i^\prime + \sigma^2I_\sfd)^{-1}\big[(Y_i -  X_i \beta)^{\otimes2}\big] -\frac{1}{2}\sum_{i=1}^n  
	 \log\det(Z_i D Z_i^\prime + \sigma^2I_\sfd) \\
	 && -n^{\frac{r}{2}}\sum_{1 \leq i \leq \sf{e}} \lambda_{i} |D_{ii}|^q \qquad\bigg(\theta=\big(\beta, \sigma^2, \psi(D)\big) \in \Theta\bigg),
	 \eeas
	 where $D=(D_{ij})_{1\leq i, j \leq \sf{e}}$, and $\lambda_{i}> 0$,  $0<q \leq 1$ and $0\leq   r \leq 1$ are tuning parameters. Note that we only penalize the diagonal elements of $D$.
	 Let $\hat \theta_n=\big(\hat \beta_n, \hat \sigma_n^2, \psi(\hat D_n)\big)$ be a maximizer of $\Psi_n$ on $\Theta$.  
	 
	For simplicity, suppose that $\sf{c} =\sf{d} =1$ and $\sf{e} =2$.  Then $\theta =(\beta, \sigma^2, D_{11}, D_{12}, D_{22})$ becomes a $5$-dimensional vector. 
	Consider the case where
	 \beas
	  D_{11}^*>0, ~~D_{12}^*=D_{22}^*=0.
	 \eeas
Assume $|X|, |Z| \in L^p(dP)$ for any $p>0$, and define a continuous function $L$ on $\Theta$ as for any $\theta=\big(\beta, \sigma^2, D_{11}, D_{12}, D_{22}\big) \in \Theta$,
\beas
L(\theta) &=&  E\bigg[ -\frac{1}{2} 
(ZD Z^\prime + \sigma^2)^{-1}(Y -  X \beta)^{2}-\frac{1}{2} 
\log(ZD Z^\prime + \sigma^2)\bigg].
\eeas
 Also, define a continuous random field $\bbZ$ on $\bbR^5$ as for any $u=(u_{\beta}, u_{\sigma^2}, w_1, w_2, w_3) \in \bbR^5$,
\bea\label{exfianlZ}
\bbZ(u)&=&\exp\bigg\{\Delta(\theta^*)[Bu]-\frac{1}{2} \Gamma(\theta^*)\big[{(Bu)}^{\otimes2}\big]
-1_{\{r=1\}}\lambda_{1}\frac{q}{|D^*_{11}|^{1-q}}w_1-\lambda_{2}1_{ \{q \leq r\}}|w_3|^{q}\bigg\},
\eea where $\Delta(\theta^*) \sim N_{5}\big(0, \Gamma(\theta^*)\big)$, $\Gamma(\theta^*)=\big(\Gamma(\theta^*)_{ij}\big)_{1\leq i,j\leq 5} =-\partial_\theta^2 L(\theta^*)$ and 
\beas
B &=& \begin{cases} I_5 & (q \geq r)\\
 {\rm diag}(1, 1, 1, 1, 0) & (q < r)
	\end{cases}.
\eeas
Take $a_n \in GL(5)$ as a diagonal matrix defined as
\beas
(a_n)_{jj} \yeq \begin{cases} n^{-\frac{1}{2}} & \big(j= 1, ..., 4 \big)\\
	n^{ -\frac{\rho}{2}} & \big(j = 5\big)
\end{cases},
\eeas
where $\rho = \frac{r}{q}\vee 1$. Define $U_n$ and $U$ as  (\ref{U_T}) and (\ref{U}), respectively.
Then from Example \ref{exU3}, [{\bf A3}] holds, and $U$ is determined as 
\bea\label{exfinalU}
U \yeq  \bbR^2 \times W \yeq \bbR^2 \times \begin{cases}  \big\{(w_1, w_2, w_3) \in \bbR^3 ; w_3 \geq 0 \big\} & (q>\frac{r}{2})\\
	\big\{(w_1, w_2, w_3) \in \bbR^3 ; D_{11}^*w_3 - w_2^2 \geq 0\big\} & (q=\frac{r}{2})\\
	\big\{(w_1, w_2, w_3) \in \bbR^3 ; w_3 \geq 0, w_2= 0\big\} & (q<\frac{r}{2})
\end{cases},
\eea
where $W \in \bbR^3$ is defined as  (\ref{ex3U}).
Suppose that $\Gamma(\theta^*)$ is non-degenerate and that for any $\theta\in\Theta$ with $\theta \neq \theta^*$, $L(\theta)<L(\theta^*)$. 

\begin{lemma}\label{finallem} Condition $[{\bf A4}]$ holds for $\bbZ$ defined in $(\ref{exfianlZ})$. Moreover, for the maximizer $\hat u$    defined in $[{\bf A4}]$, 
	\bea
	P[\hat u_4 =\hat u_5  =0] =1~~ \Leftrightarrow ~~ q<\frac{r}{2},\label{finalcondi}
	\eea
where $\hat u_i$ denotes the $i$-th component of $\hat u$ for each $i=1, ..., 5$.
	\end{lemma}
\begin{proof}
	Assume $q=r2^{-1}$. Then from (\ref{exfianlZ}),
	for any $u=(u_{\beta}, u_{\sigma^2}, w_1, w_2, w_3) \in \bbR^5$,
	\bea\label{formofZ}
	\bbZ(u)&=&\exp\bigg\{H(u_{\beta}, u_{\sigma^2}, w_1, w_2)-\lambda_{2}|w_3|^{q}\bigg\},
	\eea where $H$ is some continuous random field on $\bbR^4$ not depending on $w_3$.
	Also, from (\ref{exfinalU}), $U = \bbR^3 \times \big\{(w_2, w_3) ; D_{11}^*w_3 \geq w_2^2\big\}$.
	 Since $w_3$ can decrease  when  $D_{11}^*w_3 > w_2^2$, the maximizers of $\bbZ$ on $U$ is contained in $\partial U= \bbR^3 \times \{D_{11}^*w_3 = w_2^2\}$, and for any $u=(u_{\beta}, u_{\sigma^2}, w_1, w_2, w_3) \in \partial U$,
	\beas
	\bbZ(u)&=&\exp\bigg\{\Delta(\theta^*)[Bu]-\frac{1}{2} \Gamma(\theta^*)\big[{(Bu)}^{\otimes2}\big]
	-1_{\{r=1\}}\lambda_{1}\frac{q}{|D^*_{11}|^{1-q}}w_1-\lambda_{2}(D_{11}^*)^{-q}|w_2|^{r}\bigg\}\\
	&=&\exp\bigg\{C+Aw_2-\gamma w_2^2-\lambda_{2}(D_{11}^*)^{-q}|w_2|^{r}\bigg\},
	\eeas
where $C$ and $A$ are  random variables depending on $(u_{\beta}, u_{\sigma^2}, w_1)$,  and $\gamma>0$ is a deterministic  number. Then we can narrow down the candidates for the $w_2$-component of the maximizers of $\bbZ$  to  two points. One of those two points equals $0$. 
 The  maximizer  of $\bbZ$ on $\partial U$ whose $w_2$-component equals $0$ is uniquely determined, and we denote it by $u^\dagger=(u_{\beta}^\dagger, u_{\sigma^2}^\dagger, w_1^\dagger, 0, 0)$.  Also, denote   $\log\bbZ\big(u_{\beta}, u_{\sigma^2}, w_1, w_2, (D_{11}^*)^{-q}w_2^2\big)$ by $L\big(u_{\beta}, u_{\sigma^2}, w_1, w_2; \Delta(\theta^*)\big)$. 
 
If the maximizer of $\bbZ$ of $\partial U$ is not uniquely determined,  then there exists  some $(u_{\beta}, u_{\sigma^2}, w_1, w_2) \in \bbR^4\setminus\{(u_{\beta}^\dagger, u_{\sigma^2}^\dagger, w_1^\dagger, 0)\}$  such that
 \bea
 F_1\big(u_{\beta}, u_{\sigma^2}, w_1, w_2, \Delta(\theta^*)\big)&:=&L\big(u_{\beta}, u_{\sigma^2}, w_1, w_2; \Delta(\theta^*)\big)- L\big(u_{\beta}^\dagger, u_{\sigma^2}^\dagger, w_1^\dagger, 0; \Delta(\theta^*)\big)\yeq 0 \label{F_1}\\
 F_2\big(u_{\beta}, u_{\sigma^2}, w_1, w_2, \Delta(\theta^*)\big)&:=&\partial_{(u_{\beta}, u_{\sigma^2}, w_1, w_2 )}L\big(u_{\beta}, u_{\sigma^2}, w_1, w_2 ; \Delta(\theta^*)\big) \yeq 0. \label{F_2} 
 \eea
Denote by $\Delta_i$ the $i$-th component of $\Delta:=\Delta(\theta^*)$. Now we consider $\Delta_1, ..., \Delta_4$ as variables  in the functions $ F_1, F_2$ as well as $(u_{\beta}, u_{\sigma^2}, w_1, w_2)$.
Then for any   $(u_{\beta}, u_{\sigma^2}, w_1, w_2, \Delta)$ with $w_2 \neq 0$ satisfying (\ref{F_1}) and (\ref{F_2}), 
\beas
\partial^\prime_{(u_{\beta}, u_{\sigma^2}, w_1, w_2, \Delta_4)}\begin{pmatrix} F_1\\ F_2\end{pmatrix} \yeq   \begin{pmatrix} 0& \cdots&& 0&w_2\\ & & & & 0 \\  & \text{\large$-\big(\Gamma(\theta^*)_{ij}\big)_{1\leq i, j \leq 4}$}& & & 0 \\ & & & & 0\\ &  &  &  &1\end{pmatrix}+ r(1-r)\frac{ \lambda_{2}(D_{11}^*)^{-q}}{|w_2|^{2-r}}\begin{pmatrix} 0\\0\\0\\0\\1\end{pmatrix}\begin{pmatrix} 0\\0\\0\\1\\0\end{pmatrix}^{\prime},
\eeas noting that $L(u_{\beta}^\dagger, u_{\sigma^2}^\dagger, w_1^\dagger, 0)$ does not depend on $\Delta_4$.
This matrix is non-degenerate except when $|w_2|$ equals $0$ or some  value $a$  depending on $\Gamma(\theta^*)$.  Therefore, from the implicit function theorem, for any   $(u_{\beta}, u_{\sigma^2}, w_1, w_2, \Delta)$ with $|w_2| \neq 0, a$ satisfying (\ref{F_1}) and (\ref{F_2}), $\Delta_4$ is  locally equal to some function of  $\Delta_1, \Delta_2, \Delta_3$. If $(u_{\beta}, u_{\sigma^2}, w_1, w_2, \Delta)$  satisfies (\ref{F_1}), (\ref{F_2}) and $|w_2|=a$, then removing the 4-th component of $F_2$ and  considering the Jacobian
\beas
\partial^\prime_{(u_{\beta}, u_{\sigma^2}, w_1,  \Delta_4)}\begin{pmatrix} F_1\\ \partial_{(u_{\beta}, u_{\sigma^2}, w_1 )}L\end{pmatrix} \yeq   \begin{pmatrix} 0& \cdots&& 0&w_2\\ & & & & 0 \\  & \text{\large$-\big(\Gamma(\theta^*)_{ij}\big)_{1\leq i, j \leq 3}$}& & & 0 \\ & & & & 0\end{pmatrix}, 
\eeas 
 we see $\Delta_4$ is  locally equal to some function of  $\Delta_1, \Delta_2, \Delta_3$. Therefore,  since the distribution of $\Delta(\theta^*)$ is $5$-dimensional Gaussian, the probability that (\ref{F_1}) and (\ref{F_2}) hold equals $0$.  Thus, [{\bf A4}] holds when $q=r2^{-1}$.  In this case, considering when $\Delta_4(\theta^*)$ takes some large value, we have
$P[\hat u_4 =0] < 1$.

Similarly, when $q=r$, we can show [{\bf A4}].  
\begin{en-text}
In fact, {\colorg from (\ref{exfinalU}), $U = \bbR^4\times[0, \infty)$, and 
	\beas
	\bbZ(u)&=&\exp\bigg\{\Delta(\theta^*)[u]-\frac{1}{2} \Gamma(\theta^*)\big[{u}^{\otimes2}\big]
	-1_{\{r=1\}}\lambda_{1}\frac{q}{|D^*_{11}|^{1-q}}w_1-\lambda_{2}|w_3|^{r}\bigg\}\\
	&=&\exp\bigg\{C+Aw_3-\gamma w_3^2-\lambda_{2}|w_3|^{r}\bigg\},
	\eeas
	where $C$ and $A$ are  random variables depending on $(u_{\beta}, u_{\sigma^2}, w_1, w_2)$,  and $\gamma>0$ is a deterministic  number. Then we can narrow down the candidates for the $w_3$-component of the maximizers of $\bbZ$  to  two points. One of those two points equals $0$. 
	The  maximizer  of $\bbZ$ on $U$ whose $w_3$-component equals $0$ is uniquely determined, and we denote it by $u^\dagger=(u_{\beta}^\dagger, u_{\sigma^2}^\dagger, w_1^\dagger, w_2^\dagger, 0)$.  Also, denote   $\log\bbZ\big(u_{\beta}, u_{\sigma^2}, w_1, w_2, w_3)$ by $L\big(u_{\beta}, u_{\sigma^2}, w_1, w_2, w_3; \Delta(\theta^*)\big)$. 
	
	If the maximizer of $\bbZ$ of $U$ is not uniquely determined,  then there exists  some $(u_{\beta}, u_{\sigma^2}, w_1, w_2, w_3) \in \bbR^4 \times[0, \infty)\setminus\{u^\dagger\}$  such that
	\bea
	F_1\big(u_{\beta}, u_{\sigma^2}, w_1, w_2, w_3, \Delta(\theta^*)\big)&:=&L\big(u_{\beta}, u_{\sigma^2}, w_1, w_2, w_3; \Delta(\theta^*)\big)- L\big(u_{\beta}^\dagger, u_{\sigma^2}^\dagger, w_1^\dagger, w_2^\dagger, 0; \Delta(\theta^*)\big)\yeq 0 \label{F_1prime}\\
	F_2\big(u_{\beta}, u_{\sigma^2}, w_1, w_2, w_3, \Delta(\theta^*)\big)&:=&\partial_{(u_{\beta}, u_{\sigma^2}, w_1, w_2, w_3 )}L\big(u_{\beta}, u_{\sigma^2}, w_1, w_2, w_3 ; \Delta(\theta^*)\big) \yeq 0. \label{F_2prime} 
	\eea
	Denote by $\Delta_i$ the $i$-th component of $\Delta:=\Delta(\theta^*)$. Now we consider $\Delta_1, ..., \Delta_5$ as variables  in the functions $ F_1, F_2$ as well as $(u_{\beta}, u_{\sigma^2}, w_1, w_2, w_3)$.
	Then for any   $(u_{\beta}, u_{\sigma^2}, w_1, w_2, w_3, \Delta)$ with $w_3 \neq 0$ satisfying (\ref{F_1prime}) and (\ref{F_2prime}), 
	\beas
	\partial^\prime_{(u_{\beta}, u_{\sigma^2}, w_1, w_2, w_3, \Delta_5)}\begin{pmatrix} F_1\\ F_2\end{pmatrix} \yeq   \begin{pmatrix} 0& \cdots&& 0&w_3\\ & & & & 0 \\  & \text{\large$-\Gamma(\theta^*)$}& & & 0 \\ & & & & 0\\ &&&&0\\&  &  &  &1\end{pmatrix}+ r(1-r)\frac{ \lambda_{2}}{|w_3|^{2-r}}\begin{pmatrix} 0\\0\\0\\0\\0\\1\end{pmatrix}\begin{pmatrix} 0\\0\\0\\0\\1\\0\end{pmatrix}^{\prime},
	\eeas noting that $L(u_{\beta}^\dagger, u_{\sigma^2}^\dagger, w_1^\dagger, w_2^\dagger, 0)$ does not depend on $\Delta_4$.
	This matrix is non-degenerate except when $|w_3|$ equals some  value $a$  depending on $\Gamma(\theta^*)$.  Therefore, from the implicit function theorem, for any   $(u_{\beta}, u_{\sigma^2}, w_1, w_2, w_3, \Delta)$ with $|w_3| \neq a$ satisfying (\ref{F_1prime}) and (\ref{F_2prime}), $\Delta_5$ is  locally equal to some function of  $\Delta_1, \Delta_2, \Delta_3, \Delta_4$. If $(u_{\beta}, u_{\sigma^2}, w_1, w_2, w_3, \Delta)$  satisfies (\ref{F_1}), (\ref{F_2}) and $|w_3|=a$, then removing the 5-th component of $F_2$ and  considering the Jacovian
	\beas
	\partial^\prime_{(u_{\beta}, u_{\sigma^2}, w_1,  \Delta_4)}\begin{pmatrix} F_1\\ \partial_{(u_{\beta}, u_{\sigma^2}, w_1, w_2)}L\end{pmatrix} \yeq   \begin{pmatrix} 0& \cdots&& 0&w_2\\ & & & & 0 \\  & \text{\large$-\big(\Gamma(\theta^*)_{ij}\big)_{1\leq i, j \leq 4}$}& & & 0 \\ & & & & 0\\&&&&0\end{pmatrix}, 
	\eeas 
	we see $\Delta_5$ is  locally equal to some function of  $\Delta_1, \Delta_2, \Delta_3, \Delta_4$. Therefore,  since the distribution of $\Delta(\theta^*)$ is $5$-dimensional Gaussian, the probability that (\ref{F_1prime}) and (\ref{F_2prime}) hold equals $0$.  Thus, [{\bf A4}] holds when $q=r$. 
}
\end{en-text}
Also, when $q>r$, $r2^{-1}<q<r$ or $q<r2^{-1}$, we easily obtain [{\bf A4}] from the convexity of $U$ and the convexity of  the function $\Delta(\theta^*)[Bu]-2^{-1} \Gamma(\theta^*)\big[{(Bu)}^{\otimes2}\big]$.
Thus, [{\bf A4}] holds for any $q$ and $r$. If $q>r2^{-1}$,  then  considering when $\Delta_4(\theta^*)$ takes some large value, we have $P[\hat u_4 =0] < 1$.   If $q< r2^{-1}$, then we obviously obtain $P[\hat u_4=\hat u_5 =0]=1$. Thus, we obtain [{\bf A4}] and (\ref{finalcondi}).
	\end{proof}

Denotes by  $\hat D_{ij, n}$ the $(i, j)$ entry of $\hat D_n$ for each $i, j= 1,2$. Then  the following theorem gives the weak convergence of the estimator and  the condition for its selection consistency. 

\begin{theorem}\label{sparseex}
	\bea \hat u_n :=a_n^{-1}(\hat \theta_n -\theta^*) \overset{d}{\to} \hat u.\label{sparsehatu}\eea
	Moreover, selection consistency as
	\bea
	P[\hat D_{22, n}=0]=P[\hat D_{12, n}=\hat D_{22, n}=0] \to 0 \qquad(n\to \infty) \label{exsparse}
	\eea
	holds if and only if
	\beas
	q < \frac{r}{2}.
	\eeas
	{In particular, the estimator has the oracle property if and only if $q < \frac{r}{2}$ and $r<1$.}
	\end{theorem}
{
\begin{proof}
	 To apply Theorem \ref{thmpenalized}, we ensure [{\bf C1}]-[{\bf C5}], [{\bf A3}] and [{\bf A4}] . 
 Let $\calg$ be $\{\phi, \calf\}$.
Define $\calj$, $\calj_0$ and $\calj_1$  as
 \beas
 \calj \yeq \{3,  5\}, ~~\calj_1 \yeq \{3\}, ~~\calj_0 \yeq  \{5\}.
 \eeas
	Take  $p_j$, $q_j$, $\xi_{j, n}$  $(j \in \calj)$, 
	\beas
	p_3(x)=p_5(x)= |x|^q \qquad(x \in \bbR),~~~q_3=q_5=q,~~~~
	(\xi_{3, n}, \xi_{5, n} ) \yeq  (n^{\frac{r}{2}} \lambda_{1},  n^{\frac{r}{2}} \lambda_{2}).
	\eeas
	Define $\calh_n$, $\caly$ as  for any $ \theta=(\beta, \psi(D), \sigma^2) \in \Theta$,
	\beas
	\calh_n(\theta) &=& -\frac{1}{2}\sum_{i=1}^n  
	(Z_iD Z_i^\prime + \sigma^2)^{-1}(Y_i -  X_i \beta)^{2} -\frac{1}{2}\sum_{i=1}^n  
	\log(Z_iD Z_i^\prime + \sigma^2),\\
	{ \caly(\theta) }&=& L(\theta)-L(\theta^*).
	\eeas
	Take a bounded open set $\caln\subset{\bbR^5}$ satisfying Condition (i) and (ii) in Section \ref{Quasisection1.1} as
	\beas   \caln = \bigg\{(\beta, \sigma^2) \in \bbR \times (0, \infty) ; |\beta-\beta^*|<1, |\sigma^2 -(\sigma^*)^2|<\frac{(\sigma^*)^2}{2} \bigg\} \times \psi\bigg(\bigg\{D \in \cals_+^2; \det(D)>0, \|D-D^*\| < \delta\bigg\}\bigg) \eeas for $\delta>0$ in (\ref{Lineardelta}).
Take 	$\frak a_n$,  $\rho_k$ $(k \in \calj_0)$, $c_i$ $(i \in \calj_1)$ and $d_k$ $(k \in \calj_0)$ as
	\beas
	\frak a_n \yeq n^{-\frac{1}{2}}I_5, ~~\rho_5\yeq \rho \yeq  \frac{r}{q} \vee 1, ~~c_3 \yeq \lambda_{1}1_{\{r=1\}}, ~~d_5  \yeq  \lambda_{2}1_{ \{q \leq r\}}. 
	\eeas
	Then [{\bf C1}]-[{\bf C5}] obviously hold. Condition [{\bf A3}], [{\bf A4}] hold as mentioned above. Therefore, we obtain (\ref{sparsehatu}).
	
 From Lemma \ref{finallem}, (\ref{finalcondi}) holds. Therefore, it is sufficient to show (\ref{exsparse}) when $q <r2^{-1}$.
For that, we show [{\bf S1}]. Take any $R>0$ and any positive sequence $\delta_n$ with $\delta_n \to 0$. Then  for sufficiently large $n$, 
\beas
U_n\cap B_{R} &=&  \big\{u \in \bbR^4 \times [0, \infty) ; ~(D_{11}^*+n^{-\frac{1}{2}}u_3)n^{-\frac{r}{2q} }u_5 \geq n^{-1}u_4^2 \big\}\cap B_R\\
S(R, \delta_n) &=&\big\{u \in U_n ;\, 0<|u_5|<\delta_n, |\ol u|\leq R\big\}\\
&=&\big\{u \in \bbR^4 \times [0, \infty) ; ~(D_{11}^*+n^{-\frac{1}{2}}u_3)n^{-\frac{r}{2q} }u_5 \geq n^{-1}u_4^2, ~ 0<|u_5|<\delta_n, ~|\ol u|\leq R\big\}\\
&\subset&\big\{u \in \bbR^5 ; ~2D_{11}^*n^{-(\frac{r}{2q} -1)}u_5 \geq u_4^2, ~ 0<u_5<\delta_n, ~|\ol u|\leq R\big\},\\
I(R) &=&\big\{v \in U_n ;\, v_5=0, |\ol v|\leq R\big\}\yeq\big\{v \in \bbR^5 ;\, v_4=v_5=0, |\ol v|\leq R\big\},
\eeas
where $u=(u_{1}, u_{2}, u_3, u_4, u_5)$, $v=(v_{1}, v_{2}, v_3, v_4, v_5)$, and $\ol u, \ol v$ denote  $(u_1, u_2, u_3, u_4), (v_1, v_2, v_3, v_4)$, respectively.
Then noting that $\calj_{0, +} =\{k \in \calj_0 ; d_k >0 \,\,a.s.\}=\{5\}$,
\beas
\sup_{u\in S(R, \delta_n)}\inf_{v \in I(R)}\frac{|\overline u-\overline v|+|(u_k)_{k \in \calj_0, \rho_k=1}|}{\sum_{k \in \calj_{0, +}}|u_k|^{q}}&=&\sup_{u\in S(R, \delta_n)}\inf_{v \in I(R)}\frac{|\overline u-\overline v|}{|u_5 |^{q}}\\
&\leq&\sup_{u\in S(R, \delta_n)}\frac{\big|\overline u-(u_1, u_2, u_3, 0)\big|}{|u_5|^{q}}
\\&=&\sup_{u\in S(R, \delta_n)}\frac{|u_4|}{|u_5|^{q}}\\
&\leq &\sup_{0<|u_5|<\delta_n}\frac{\big|2D_{11}^*n^{-(\frac{r}{2q}-1)}u_5\big|^{\frac{1}{2}}}{|u_5|^q}\\
&\rightarrow& 0  \qquad(n\rightarrow \infty).
\eeas
Thus, [{\bf S1}] holds. Besides, [{\bf S2}] obviously holds. Thus, from Theorem \ref{Selection1}, (\ref{exsparse}) holds.
	\end{proof}
}

}

\section{Proof of Theorem \ref{thm1}}

In this section, we prove Theorem \ref{thm1}. Before that, we prepare some lemmas.

\begin{lemma}\label{lem1}
	\bd
	\im[(i)] $U$ is closed in $\mathbb{R}^\sfp$.
	\im[(ii)] For any open sets $O$ of $\mathbb{R}^\sfp$ and for any $R>0$ and $\delta>0$, there exists some number $N=N(O, R, \delta)$ such that for any $T\in\bbT$ with ${T\geq N}$,
	\begin{eqnarray*}
		U(R)\cap O \subset \big(U_T(2R)\cap O\big)^\delta. 
	\end{eqnarray*}
	
	\im[(iii)] Under {\rm [{\bf A3}]}, for any closed sets $F$ of $\mathbb{R}^\sfp$ and for any $R >0$ and $\delta>0$, there exists some number $N = N(F, R, \delta)$ such that for any $T \in\bbT$ with  $T \geq N$,
	\begin{eqnarray*}
		U_T(R)\cap F \subset \big(U(R)\cap F\big)^{\delta}.
	\end{eqnarray*}
	\ed
\end{lemma}

\begin{proof}
	We define $B_{\delta}(x)$ by $B_{\delta}(x) = \{y\in \mathbb{R}^\sfp; |y - x|< \delta\}$ for $\delta >0 $ and $x \in \mathbb{R}^\sfp$.\y
	\noindent (i) 
	Let  $\{x_n\}$ be an arbitrary sequence in $U$ which converges to some $x \in \mathbb{R}^\sfp$. We will show $x \in U$.  Let $\delta > 0$, then there exists a number ${n_0}$ such that $x \in B_{\delta/2}(x_{n_0}) $. Since $x_{n_0} \in U$, there exists a number $N = N(\delta)$ such that for all $T \geq N$, $x_{n_0} \in {U_T}^{\delta/2}$. Then for all $T \geq N$, $x \in {U_T}^{\delta}$, which means $x \in U$.
	\y\noindent (ii) 
	Take any  open set $O$ and any $R, \delta>0$.
	Let $x \in U(R)\cap O$. Choose $\delta_0$ with $0<\delta_0\leq \delta\wedge R$ satisfying $B_{\delta_0}(x)  \subset O$. 
	Since $x \in U$, there exists some number $N = N(\delta_0)$ such that for any $T\geq N$, $x \in {U_T}^{\delta_0}$. In particular, for any $T\geq N$, there exists ${x}_T \in U_T$ satisfying $|x-x_T|<\delta_0$. But then for any $T\geq N$, $x_T \in U_T(2R)\cap O$ and $x \in B_{\delta_0}(x_T) \subset \big(U_T(2R) \cap O\big)^{\delta}$.
	\y\noindent (iii) 
	{
		We will prove the assertion by contradiction. 
		Under [{\bf A3}], suppose that }
	\begin{eqnarray*}
		\exists F : a ~closed ~subset ~of~ \mathbb{R}^\sfp ~~\exists R >0 ~~\exists \delta_0>0 ~~\exists\{T_N\}_{N=1}^{\infty}\subset\bbT &&\\
		s.t.~~
		{\lim_{N\to\infty}T_N=\infty} ~~{ \text{and}}
		~~ U_{T_N}(R)\cap F \nsubseteq \big(U(R)\cap F\big)^{\delta_0}.&&
	\end{eqnarray*}	
	Choose ${x_{T_N}}$ satisfying $x_{T_N} \in \big(U_{T_N}(R)\cap F\big) \setminus \big(U(R)\cap F\big)^{\delta_0}$.
	Since ${x}_{T_N} \in F \cap \overline B_{R}$, there exists a subsequence ${x}_{T_{N_k}}$ satisfying ${x}_{T_{N_k}} \rightarrow x$ as $k \rightarrow \infty$ where $x \in F\cap\overline B_R$. 
	Then
	\begin{eqnarray*}&&
		x\in 		\bigcap_{\delta >0} \bigcup_{n=1}^{\infty} \bigcap_{k\geq n} 	{U_{T_{N_k}}}^{\delta} 
		\>\subset\> 
		\bigcap_{\delta >0} \bigcap_{n=1}^{\infty} \bigcup_{k\geq n} 	{U_{T_{N_k}}}^{\delta} 
		\>\subset\> 
		\bigcap_{N=1}^{\infty}\bigcap_{\delta >0} \bigcup_{T\geq N} 	{U_T}^{\delta}
		\\&&\hspace{20pt}
		\yeq
		\bigcap_{N=1}^{\infty}\bigcap_{\delta >0} \big(\bigcup_{T\geq N} 	{U_T}\big)^{\delta}
		\yeq
		\bigcap_{N=1}^{\infty}\overline{\big(\bigcup_{T\geq N} 	{U_T}\big)}
		\yeq 
		U.
	\end{eqnarray*}
	Therefore, $x \in U(R)\cap F$. But then for sufficiently large $k$, $x_{T_{N_k}} \in B_{\delta_0}(x) \subset \big(U(R)\cap F\big)^{\delta_0}$, which contradicts how to take $x_{T_{N}}$.
\end{proof}
\halflineskip

	\begin{lemma}\label{lem3}
		Let $\bbV$ be a $C(\bbR^\sfp)$-valued random variable satisfying that $\bbV\overset{d}{=}\bbZ$. Assume {\rm [{\bf A4}]}. 
		 Let  $\hat{v}$  be a
		 $U$-valued random variable satisfying that with probability $1$,  
		\begin{eqnarray*}\label{111}{\mathbb{V}(\hat{v})} = \sup_{U }{\mathbb{V}(u)}.
		\end{eqnarray*}
		 Then, with probability 1, for all $u \in U$ with $u \neq \hat{v}$,
		\begin{eqnarray}\label{1111}  {\mathbb{V}(u)} < {\mathbb{V}(\hat{v})},
		\end{eqnarray}
and 
\bea
 (\bbV, \hat v)\overset{d}{=}(\bbZ, \hat u). \label{11111}
\eea
		\end{lemma}
	
	\begin{proof}
	 Since $\bbV \overset{d}{=} \bbZ$, (\ref{1111}) is obviously holds from [{\bf A4}].
		We show (\ref{11111}). Take any open sets $G\in C(\bbR^\sfp)$ and $O\in\bbR$. It is sufficient to show that 
		{\colorr \beas
		P\big[(\bbV, \hat v)\in G\times O\big]=P\big[(\bbZ, \hat u)\in G\times O\big].
		\eeas
		For this,
		\beas
		P\big[(\bbV, \hat v)\in G\times O\big]&=&\lim_{R \to \infty}P\big[(\bbV, \hat v)\in G\times O, \hat v \in \ol B_R\big]\\
		&=&\lim_{R \to \infty} P\big[\bbV\in G, ~\sup_{u\in O\cap U(R)}\bbV(u)-\sup_{u\in O^c\cap U(R)}\bbV(u)>0, \,\, \hat v \in \ol B_R\big]\\
		&=&\lim_{R \to \infty} P\big[\bbV\in G, ~\sup_{u\in O\cap U(R)}\bbV(u)-\sup_{u\in O^c\cap  U(R)}\bbV(u)>0\big]\\
		&=&\lim_{R \to \infty} P\big[\bbZ\in G, ~\sup_{u\in O\cap U(R)}\bbZ(u)-\sup_{u\in O^c\cap  U(R)}\bbZ(u)>0\big]\\
		&=&\lim_{R \to \infty} P\big[\bbZ\in G, ~\sup_{u\in O\cap U(R)}\bbZ(u)-\sup_{u\in O^c\cap U(R)}\bbZ(u)>0, \,\, \hat u \in \ol B_R\big]\\
		&=&\lim_{R \to \infty}P\big[(\bbZ, \hat u)\in G\times O, \hat u \in \ol B_R\big]
		\\&=&P\big[(\bbZ, \hat u)\in G\times O\big].
		\eeas}
		
		\end{proof}

\begin{proof}[Proof of  Theorem \ref{thm1}.] First, we assume $[{\bf A1}]^{\flat\flat}$  and  [{\bf A2}]-[{\bf A4}], and   prove (\ref{hatu}) i.e.
	\beas
	\hat{u}_T\overset{d_s(\calg)}{\rightarrow}\hat u.
	\eeas Let $Y$ be any  $\calg$-measurable non-negative random variable which is bounded and satisfies  that $E[Y]>0$. Define a probability measure $P_Y$ as
	\beas
	P_Y[A]=\frac{E[1_AY]}{E[Y]} \qquad(A \in \calf).
	\eeas 
	Let $O$ be an  arbitrary open set of $\mathbb{R}^\sfp$. It is sufficient to show that \begin{eqnarray*}
		\underset{T \rightarrow \infty}{\varliminf}P_Y[\hat{u}_T \in O] \geq P_Y[\hat{u} \in O].
	\end{eqnarray*}
	It is obvious when $U\cap O = \phi$, so we consider only the case where $U \cap O \neq \phi$. Then there exists $R_0 > 0$ satisfying $U(R_0/2) \cap O \neq \phi$.  From  Lemma \ref{lem1} (ii), for sufficient large $T \in \mathbb{T}$, $U_T(R_0) \cap O \neq \phi$.  In the following, $T$  will be such large numbers, and let $R \geq R_0$ and $\epsilon >0$.
	We have	
	\begin{eqnarray*}
		&\underset{T \rightarrow \infty}{\varliminf}&P_Y[\hat{u}_T \in O] 
		\\&\geq& 
		\underset{T \rightarrow \infty}{\varliminf}
		P_Y[\hat{u}_T \in O, \hat{u}_T \in \overline{B_R}] 
		\\&\geq&  
		\underset{T \rightarrow \infty}{\varliminf}
		P_Y\bigg[\sup_{u \in U_T(R) \cap O} \mathbb{Z}_T (u, \hat\tau_T)
		- \sup_{u \in U_T(R) \cap O^c } \mathbb{Z}_T(u, \hat\tau_T) >0, \hat{u}_T \in \overline{B_R}\bigg] 
		\\&\geq& 
		\underset{T \rightarrow \infty}{\varliminf}
		P_Y\bigg[\sup_{ u \in U_T(R) \cap O} \mathbb{Z}_T (u, \hat\tau_T)
		- \sup_{ u \in U_T(R) \cap O^c} \mathbb{Z}_T(u, \hat\tau_T) >0\bigg] 
		-   \underset{T \rightarrow \infty}{\varlimsup}P_Y\big[ \hat{u}_T \notin \overline{B_R}\big]  
		\\&\geq& 
		\underset{T \rightarrow \infty}{\varliminf}
		P_Y\bigg[\sup_{U_T(R) \cap O} \mathbb{V}_T 
		- \sup_{U_T(R) \cap O^c} \mathbb{V}_T -2\sup_{U_T(R) \times \mathcal{T}}|\mathbb{V}_T - \mathbb{Z}_T| >0\bigg] 
		-   \underset{T \rightarrow \infty}{\varlimsup}P_Y\big[ \hat{u}_T \notin \overline{B_R} \big] .
	\end{eqnarray*}
	By $[{\bf A1}]^{\flat\flat}$, the second term  on the right-hand side
	converges to 0 as $R \rightarrow 0$. 
	By [{\bf A2}], the first term can be evaluated as follows:
	\begin{eqnarray*} 
		&\underset{T \rightarrow \infty}{\varliminf}&
		P_Y\bigg[\sup_{U_T(R) \cap O} \mathbb{V}_T 
		- \sup_{U_T(R) \cap O^c} \mathbb{V}_T -2\sup_{U_T(R) \times \mathcal{T}}|\mathbb{V}_T - \mathbb{Z}_T| >0\bigg] 
		\\ &\geq& 
		\underset{T \rightarrow \infty}{\varliminf}
		P_Y\bigg[\sup_{U_T(R) \cap O} \mathbb{V}_T 
		- \sup_{U_T(R) \cap O^c} \mathbb{V}_T >2\sup_{U_T(R) \times \mathcal{T}}|\mathbb{V}_T - \mathbb{Z}_T|,  \sup_{U_T(R) \times \mathcal{T}}|\mathbb{V}_T - \mathbb{Z}_T|< \epsilon/2\bigg] 
		\\ &\geq& 
		\underset{T \rightarrow \infty}{\varliminf}
		P_Y\bigg[\sup_{U_T(R) \cap O} \mathbb{V}_T 
		- \sup_{U_T(R) \cap O^c} \mathbb{V}_T >\epsilon\bigg] 
		- \underset{T \rightarrow \infty}{\varlimsup}
		P_Y\bigg[ \sup_{U_T(R) \times \mathcal{T}}|\mathbb{V}_T 
		- \mathbb{Z}_T|\geq \epsilon/2\bigg] 
		\\ &=&
		\underset{T \rightarrow \infty}{\varliminf}
		P_Y\bigg[\sup_{U_T(R) \cap O} \mathbb{V}_T 
		- \sup_{U_T(R) \cap O^c} \mathbb{V}_T >\epsilon\bigg] .
	\end{eqnarray*}
	Therefore, it is sufficient to show that
	\begin{eqnarray*} 
		\underset{\epsilon \rightarrow 0}{\varliminf}\underset{R \rightarrow \infty}{\varliminf}	\underset{T \rightarrow \infty}{\varliminf}
		P_Y\bigg[\sup_{U_T(R) \cap O} \mathbb{V}_T 
		- \sup_{U_T(R) \cap O^c} \mathbb{V}_T >\epsilon\bigg] 
		\geq 
		P_Y\big[\hat{u} \in O\big]. 
	\end{eqnarray*}
Let $\delta$ be a positive number.
Considering Lemma \ref{lem1} (ii) and (iii), we have
\beas
&&\underset{\epsilon \rightarrow 0}{\varliminf}\underset{R \rightarrow \infty}{\varliminf}	\underset{T \rightarrow \infty}{\varliminf}
P_Y\bigg[\sup_{U_T(R) \cap O} \mathbb{V}_T 
- \sup_{U_T(R) \cap O^c} \mathbb{V}_T >\epsilon\bigg] \\
&\geq& \underset{\epsilon \rightarrow 0}{\varliminf}\underset{R \rightarrow \infty}{\varliminf}	\underset{\delta \rightarrow 0}{\varliminf}	\underset{T \rightarrow \infty}{\varliminf}
P_Y\bigg[\sup_{U_T(R) \cap O} \mathbb{V}_T 
- \sup_{(U(R) \cap O^c)^\delta} \mathbb{V}_T >\epsilon\bigg]  \qquad\big(\because {\rm Lemma ~\ref{lem1}~(iii)}\big)\\
&\geq& \underset{\epsilon \rightarrow 0}{\varliminf}\underset{R \rightarrow \infty}{\varliminf}	\underset{\delta \rightarrow 0}{\varliminf}	\underset{T \rightarrow \infty}{\varliminf}
P_Y\bigg[\sup_{(U_T(R) \cap O)^\delta} \mathbb{V}_T 
- \sup_{U(R) \cap O^c} \mathbb{V}_T >2\epsilon\bigg]  \qquad\big(\because \text{(\ref{2110240341}) of [{\bf A2}]}\big)\\
&\geq& \underset{\epsilon \rightarrow 0}{\varliminf}\underset{R \rightarrow \infty}{\varliminf} \underset{T \rightarrow \infty}{\varliminf}
P_Y\bigg[\sup_{U(R/2) \cap O} \mathbb{V}_T 
- \sup_{U(R) \cap O^c} \mathbb{V}_T >2\epsilon\bigg]  \qquad\big(\because {\rm Lemma ~\ref{lem1}~(ii)}\big)\\
&\geq& \underset{\epsilon \rightarrow 0}{\varliminf}\underset{R \rightarrow \infty}{\varliminf} 
P_Y\bigg[\sup_{U(R/2) \cap O} \bbZ 
- \sup_{U(R) \cap O^c} \bbZ >2\epsilon\bigg]  \qquad\big(\because (\text{\ref{2110240341}) of [{\bf A2}]}\big)
\eeas
Also, 
\beas
\underset{\epsilon \rightarrow 0}{\varliminf}\underset{R \rightarrow \infty}{\varliminf} 
P_Y\bigg[\sup_{U(R/2) \cap O} \bbZ 
- \sup_{U(R) \cap O^c} \bbZ >2\epsilon\bigg]  &\geq& \underset{\epsilon \rightarrow 0}{\varliminf}\underset{R \rightarrow \infty}{\varliminf} 
P_Y\bigg[\sup_{U(R/2) \cap O} \bbZ 
- \sup_{U \cap O^c} \bbZ >2\epsilon\bigg]  \\
&=& \underset{\epsilon \rightarrow 0}{\varliminf}
P_Y\bigg[\sup_{U \cap O} \bbZ 
- \sup_{U \cap O^c} \bbZ >2\epsilon\bigg] \\
&=& 
P_Y\bigg[\sup_{U \cap O} \bbZ 
- \sup_{U \cap O^c} \bbZ >0\bigg] \\
&\geq&P_Y\bigg[\hat u \in O\bigg]\qquad \big(\because [{\bf A4}]\big).
\eeas
Thus, under  $[{\bf A1}]^{\flat\flat}$ and [{\bf A2}]-[{\bf A4}], (\ref{hatu}) holds. Therefore,  under [{\bf A1}]-[{\bf A5}], (\ref{hatu}) holds. 

Next, we also assume [{\bf A5}], and prove (\ref{hatu-hatv}) i.e.
\beas
\hat u_T-\hat v_T=o_P(1).
\eeas

{\noindent{\bf (i)} We prove that for any $R_1>0$, 
	\begin{eqnarray}\label{i}
	\big(\bbV_T, \hat v_T\big) \overset{d}{\rightarrow} \big(\bbZ, \hat u\big) \qquad in~C(\overline{B_{R_1}})\times \bbR^\sfp.
	\end{eqnarray}
	It is sufficient (and also necessary) to  prove
	\begin{eqnarray}\label{i1}
	\big(\bbV_T, \hat v_T\big) \overset{d}{\rightarrow} \big(\bbZ, \hat u\big) \qquad in~C(\bbR^\sfp)\times \bbR^\sfp.
	\end{eqnarray}
	For any  diverging sequence $\{S_k\}_{k=1}^\infty$ in $\bbT$, due to the tightness of $\big\{(\bbV_{S_k}, \hat v_{S_k})\big\}_{k=1}^\infty$, we can find a subsequence $\{T_n\}_{n=1}^\infty$ of $\{S_k\}_{k=1}^\infty$ such that
	\begin{eqnarray*}
		\big(\bbV_{T_n}, \hat v_{T_n}\big) \overset{d}{\rightarrow} \big(\bbV_{\infty}, \hat v_{\infty}\big) \qquad in~C(\bbR^\sfp)\times \bbR^\sfp.
	\end{eqnarray*}
Then since $\bbV_{T}(\hat v_T)=\sup_{U}\bbV_T(u)~a.s.$, we have
\beas
P\bigg[\forall u \in U ; \bbV_\infty(\hat v_\infty) \geq \bbV_\infty(u)\bigg] &=&P\bigg[\forall R_1>0  ; \bbV_\infty(\hat v_\infty) \geq \sup_{u \in U(R_1)}\bbV_\infty(u)\bigg] \\
&=& \lim_{R_1\to\infty}P\bigg[ \bbV_\infty(\hat v_\infty) - \sup_{u \in U(R_1)}\bbV_\infty(u)\geq 0\bigg]\\
&\geq& \lim_{R_1\to\infty}\varlimsup_{n \to \infty}P\bigg[ \bbV_{T_n}(\hat v_{T_n}) - \sup_{u \in U(R_1)}\bbV_{T_n}(u)\geq 0\bigg]\yeq 1.
\eeas
	Thus, $\bbV_\infty(\hat v_\infty)=\sup_{u \in U}\bbV_\infty(u)~a.s.$. Also,  [{\bf A2}] implies that $\bbV_\infty\overset{d}{=}\bbZ$. 
	Then from Lemma \ref{lem3}, 
	\begin{eqnarray*}
		\big(\bbV_{\infty}, \hat v_{\infty}\big) \overset{d}{=} \big(\bbZ, \hat u\big).
	\end{eqnarray*}
	Thus, (\ref{i1}) holds, which implies  (\ref{i}).
}

\vspace{3mm}
\noindent{\bf (ii)} Using (\ref{i}), we prove (\ref{hatu-hatv}).
Take any $\eta>0$, and define a random set $O_{T}$ depending on $T\in\bbT$ as
\beas
O_{T}(\omega)=\{u\in \bbR^\sfp ; |u-\hat v_T(\omega)|< \eta\}\qquad(\omega \in \Omega).
\eeas
Also, for each $v\in \bbR^\sfp$, define $\phi_v \in C(\bbR^\sfp)$ as 
\beas
\phi_v(u)=2-\bigg\{\bigg(\frac{2|u-v|}{\eta} \vee {1}\bigg)\wedge2 \bigg\} \qquad (u \in \bbR^\sfp).
\eeas
Then for any $R>0$,
\beas
&&\bigg\{\omega ;\,\hat u_T\in O_{T}, \hat u_T \in \overline{B_R}\bigg\}\\
&\supset&
\bigg\{\omega ; \sup_{(U_T(R) \cap O_T) \times \mathcal{T}} \mathbb{Z}_T 
- \sup_{(U_T(R) \cap O_T^c) \times \mathcal{T}} \mathbb{Z}_T >0, \hat{u}_T \in \overline{B_R}\bigg\}\\
&=&
\bigg\{\omega ; \sup_{(u, \tau)\in U_T(R)  \times \mathcal{T}} \mathbb{Z}_T(u, \tau)1_{O_T} (u)
- \sup_{(u, \tau)\in U_T(R) \times \mathcal{T}} \mathbb{Z}_T(u, \tau)\big(1-1_{O_T}(u)\big) >0, \hat{u}_T \in \overline{B_R}\bigg\}\\
&\supset&
\bigg\{\omega ; \sup_{(u, \tau)\in U_T(R)  \times \mathcal{T}} \mathbb{Z}_T(u, \tau)\phi_{\hat v_T} (u)
- \sup_{(u, \tau)\in U_T(R) \times \mathcal{T}} \mathbb{Z}_T(u, \tau)\big(1-\phi_{\hat v_T}(u)\big) >0, \hat{u}_T \in \overline{B_R}\bigg\}.
\eeas
Therefore,{
\beas
&&\underset{T\rightarrow\infty}{\varliminf}P\bigg[|\hat u_T-\hat v_T|< \eta\bigg]
\yeq \underset{T\rightarrow\infty}{\varliminf}P\bigg[\hat u_T\in O_{T}\bigg]\\
 &\geq&  \underset{R\rightarrow\infty}{\varliminf}\underset{T\rightarrow\infty}{\varliminf}P\bigg[\hat u_T\in O_{T}, \hat u_T \in \overline{B_R}\bigg]\\
&\geq&
\underset{R\rightarrow\infty}{\varliminf}\underset{T\rightarrow\infty}{\varliminf}P\bigg[\sup_{ U_T(R)  \times \mathcal{T}} \mathbb{Z}_T(u, \tau)\phi_{\hat v_T} (u)
- \sup_{U_T(R) \times \mathcal{T}} \mathbb{Z}_T(u, \tau)\big(1-\phi_{\hat v_T}(u)\big) >0,\, \hat{u}_T \in \overline{B_R}\bigg] 
\\&=& 
\underset{R\rightarrow\infty}{\varliminf}\underset{T\rightarrow\infty}{\varliminf}
P\bigg[\sup_{ U_T(R)  \times \mathcal{T}} \mathbb{Z}_T(u, \tau)\phi_{\hat v_T} (u)
- \sup_{ U_T(R) \times \mathcal{T}} \mathbb{Z}_T(u, \tau)\big(1-\phi_{\hat v_T}(u)\big) >0\bigg] ~~\big(\because [{\bf A1}]^{\flat\flat}\big).
\eeas}
Also,
\beas 
&&\underset{R\rightarrow\infty}{\varliminf}\underset{T\rightarrow\infty}{\varliminf}
P\bigg[\sup_{(u, \tau)\in U_T(R)  \times \mathcal{T}} \mathbb{Z}_T(u, \tau)\phi_{\hat v_T} (u)
- \sup_{(u, \tau)\in U_T(R) \times \mathcal{T}} \mathbb{Z}_T(u, \tau)\big(1-\phi_{\hat v_T}(u)\big) >0\bigg]\\&\geq& 
\underset{\epsilon\rightarrow0}{\varliminf}\underset{R\rightarrow\infty}{\varliminf}\underset{T\rightarrow\infty}{\varliminf}
P\bigg[\sup_{u\in U_T(R) } \mathbb{V}_T(u)\phi_{\hat v_T} (u)
- \sup_{u\in U_T(R) } \mathbb{V}_T(u)\big(1-\phi_{\hat v_T}(u)\big) >\epsilon\bigg] 
\qquad \big(\because [{\bf A2}]\big)\\
&\geq& \underset{\epsilon \rightarrow 0}{\varliminf}\underset{R \rightarrow \infty}{\varliminf}	\underset{\delta \rightarrow 0}{\varliminf}	\underset{T \rightarrow \infty}{\varliminf}
P\bigg[\sup_{u\in U_T(R) } \mathbb{V}_T(u)\phi_{\hat v_T} (u)
- \sup_{u\in U(R)^\delta  } \mathbb{V}_T(u)\big(1-\phi_{\hat v_T}(u)\big) >\epsilon\bigg] \\
 &&\big(\because {\rm Lemma ~\ref{lem1}~(iii)}\big)\\
&\geq& \underset{\epsilon \rightarrow 0}{\varliminf}\underset{R \rightarrow \infty}{\varliminf}	\underset{\delta \rightarrow 0}{\varliminf}	\underset{T \rightarrow \infty}{\varliminf}
P\bigg[\sup_{u\in (U_T(R))^\delta } \mathbb{V}_T(u)\phi_{\hat v_T} (u)
- \sup_{u\in U(R) } \mathbb{V}_T(u)\big(1-\phi_{\hat v_T}(u)\big) >2\epsilon\bigg] \\ 
&&\big(\because \text{ tightness of  $\{\hat v_T\}_{T\geq T_0}$ and (\ref{2110240341}) } \big)\\
&\geq& \underset{\epsilon \rightarrow 0}{\varliminf}\underset{R \rightarrow \infty}{\varliminf} \underset{T \rightarrow \infty}{\varliminf}
P\bigg[\sup_{u\in U(R/2)} \mathbb{V}_T(u)\phi_{\hat v_T} (u)
- \sup_{u\in U(R) } \mathbb{V}_T(u)\big(1-\phi_{\hat v_T}(u)\big) >2\epsilon\bigg] \\ &&\big(\because {\rm Lemma ~\ref{lem1}~(ii)}\big)
\eeas
Finally,
\beas
&& \underset{\epsilon \rightarrow 0}{\varliminf}\underset{R \rightarrow \infty}{\varliminf} \underset{T \rightarrow \infty}{\varliminf}
P\bigg[\sup_{u\in U(R/2)} \mathbb{V}_T(u)\phi_{\hat v_T} (u)
- \sup_{u\in U(R) } \mathbb{V}_T(u)\big(1-\phi_{\hat v_T}(u)\big) >2\epsilon\bigg] \\
&\geq& \underset{\epsilon \rightarrow 0}{\varliminf}\underset{R \rightarrow \infty}{\varliminf} 
P\bigg[\sup_{u\in U(R/2)} \mathbb{Z}(u)\phi_{\hat u} (u)
- \sup_{u\in U(R) } \mathbb{Z}(u)\big(1-\phi_{\hat u}(u)\big) >2\epsilon\bigg]\qquad\big(\because (\ref{i})\big) \\
&\geq& \underset{\epsilon \rightarrow 0}{\varliminf}\underset{R \rightarrow \infty}{\varliminf} 
P\bigg[\sup_{u\in U(R/2)} \mathbb{Z}(u)\phi_{\hat u} (u)
- \sup_{u\in U } \mathbb{Z}(u)\big(1-\phi_{\hat u}(u)\big) >2\epsilon\bigg] \\
&=& \underset{\epsilon \rightarrow 0}{\varliminf}P\bigg[\sup_{u\in U} \mathbb{Z}(u)\phi_{\hat u} (u)
- \sup_{u\in U } \mathbb{Z}(u)\big(1-\phi_{\hat u}(u)\big) >2\epsilon\bigg] \\
&=& P\bigg[\sup_{u\in U} \mathbb{Z}(u)\phi_{\hat u} (u)
- \sup_{u\in U } \mathbb{Z}(u)\big(1-\phi_{\hat u}(u)\big) >0\bigg] \yeq 1 \qquad\big(\because [{\bf A4}]\big).
\eeas
Thus  $\underset{T\rightarrow\infty}{\varliminf}P\bigg[|\hat u_T-\hat v_T|< \eta\bigg]=1$ for any $\eta>0$, and we obtain (\ref{hatu-hatv}).

\end{proof}

\section{Proof of Theorems \ref{ex2thm} and \ref{SDEtaudisappear}}\label{section6}
 
 \begin{proof}[Proof of Theorem \ref{ex2thm}] { The proof is similar   to  Uchida and Yoshida \cite{uchida2013quasi}.} Take $\sfp\times \sfp$ matrices $a_n$ as $a_n ={\rm diag}(n^{-\frac{1}{2}}, ..., n^{-\frac{1}{2}})$. Define  $\bbH_n$ as $\bbH_n(\theta)=\mathbb\Psi_n(\theta)$. 
 	We ensure [{\bf B1}]-[{\bf B3}], [{\bf A3}] and [{\bf A4}] for $\bbT=\bbN$ to apply Theorem \ref{thmB}. 
 	For any $1 \leq i \leq n$ and any $\theta \in \Theta$, denote $S(X_{t_{i-1}}, \theta)$ and $\sigma(X_{t_{i-1}}, a^*)$ by $S_{i-1}(\theta)$ and $\sigma_{i-1}$, respectively. Then
 	{ \bea
 	\Delta_i Y &=& \int_{t_{i-1}}^{t_i}(b_s-\beta_s )ds+\int_{t_{i-1}}^{t_i}(\beta_s-\beta_{t_{i-1}} )ds+h\beta_{t_{i-1}}+h\frac{\int_{t_{i-1}}^{t_i}\big(\sigma(X_s, a^*)-\sigma_{i-1} \big)dw_s}{h}\nonumber\\
 	&&+\sqrt h\frac{\sigma_{i-1}\Delta_iw}{\sqrt h},\label{DeltaY}
 	\eea
 	where $\Delta_iw =w_{t_{i}}-w_{t_{i-1}}$, and $\beta=\beta(n)$ is an $\F$-adapted continuous process \footnote{Unlike (\ref{defbeta}), we may define $\beta$ as $\beta_t= \frac{1}{\epsilon_n}\int_{(t-\epsilon_n)\vee 0}^t b_s ds$ $(t\geq 0)$ outside a $P$-null set where $b(\omega) \notin L^1([0, T])$. In this case, completion of $\F$ may be necessary for  the adaptedness of $\beta$. }
 	defined as \bea\label{defbeta}
 	\beta_t \yeq \frac{1}{\epsilon_n}\int_{(t-\epsilon_n)\vee 0}^t b_s 1_{\{|b_s| \leq \epsilon_n^{-1}\}}ds \qquad(t \geq 0)
 	\eea for some positive sequence $\epsilon_n$ with $\epsilon_n \to 0$ and $h\epsilon_n^{-1} \to 0$.
 From [{\bf I1}] and $h\epsilon_n^{-1} \to 0$, for any $p>1$,
\bea\label{evalubeta}
E\bigg[\sum_{i=1}^n\bigg\{\int_{t_{i-1}}^{t_i}|b_s-\beta_s |^pds+\int_{t_{i-1}}^{t_i}|\beta_s-\beta_{t_{i-1}} |^pds\bigg\}\bigg]  \to 0 \qquad(n \to \infty).
\eea
 Therefore,   
 	\bea\label{SDEsame}
 	-\frac{1}{2 \sqrt{n}h}\sum_{i=1}^n\partial_\theta S_{i-1}^{-1}( \theta^*)\big[(\Delta_iY)^{\otimes2}\big]=-\frac{1}{2}\sum_{i=1}^n\partial_\theta S_{i-1}^{-1}( \theta^*)\bigg[\bigg(\frac{\sigma_{i-1}\Delta_iw}{\sqrt h}\bigg)^{\otimes2}\bigg]+J_n +R_n,
 	\eea
 	where 
 	\beas
 	J_n &=& -\frac{\sqrt h}{\sqrt{n}}\sum_{i=1}^n\partial_\theta S_{i-1}^{-1}( \theta^*)\bigg[\bigg(\frac{\sigma_{i-1}\Delta_iw}{\sqrt h}\bigg)\otimes
 	\bigg(\beta_{t_{i-1}}+\frac{\int_{t_{i-1}}^{t_i}\big(\sigma(X_s, a^*)-\sigma_{i-1} \big)dw_s}{h}\bigg)\bigg]\\
 	&&-\frac{h}{2\sqrt{n}}\sum_{i=1}^n\partial_\theta S_{i-1}^{-1}( \theta^*)\bigg[
 	\bigg(\beta_{t_{i-1}}+\frac{\int_{t_{i-1}}^{t_i}\big(\sigma(X_s, a^*)-\sigma_{i-1} \big)dw_s}{h}\bigg)^{\otimes2}\bigg],
 	\eeas
 	and $R_n$ is the residual term involved with the first and second terms of (\ref{DeltaY}) satisfying $R_n = o_P(1)$ from (\ref{evalubeta}) and [{\bf I1}]-[{\bf I4}]. Using It\^o  formula and [{\bf I2}]-[{\bf I4}], we also obatin $J_n = o_P(1)$.}
 	Then for any $u \in \bbR^\sfp$,
 	\bea\label{ex2partial1}
 	&&\frac{1}{\sqrt n}\partial_\theta\bbH_n(\theta^*)[u]\nonumber\\&=& -\frac{1}{2\sqrt{n} }\sum_{i=1}^n\bigg\{\partial_\theta S_{i-1}^{-1}( \theta^*)\big[u, h^{-1}(\sigma_{i-1}\Delta_iw)^{\otimes2}\big]+\partial_\theta\log|S_{i-1}|(\theta^*)[u]\bigg\}+o_P(1)|u|\nonumber\\
 	&=&-\frac{1}{2\sqrt{n} }\sum_{i=1}^n{\rm Tr}\bigg(\sigma_{i-1}^\prime\partial_\theta S_{i-1}^{-1}(\theta^*)\sigma_{i-1}\big\{h^{-1}(\Delta_iw)^{\otimes2}-I_\sfr\big\}[u]\bigg)+o_P(1)|u|.
 	\eea
 	Similarly, for any $\theta \in \Theta$ and any $u \in \bbR^\sfp$,
 	\bea\label{ex2partial2}
 	&&\frac{1}{n}\partial_\theta^2\bbH_n(\theta)[u^{\otimes2}]\nonumber\\&=&  -\frac{1}{2n }\sum_{i=1}^n\bigg\{\partial^2_\theta S_{i-1}^{-1}( \theta)\big[u^{\otimes2}, h^{-1}{(\sigma_{i-1}\Delta_iw)^{\otimes2}}\big]+\partial^2_\theta\log|S_{i-1}|(\theta)[u^{\otimes2}]\bigg\}+o_P(1)|u|^2.
 	\eea
 	where $o_P(1)$ satisfies $\sup_{\theta \in \Theta}|o_P(1)| \to^P 0$  {
 	by using  Sobolev's inequalities, noting that  $\partial_\theta^2S^{-1}_{i-1}(\theta)$ simply consists of rational functions of $\theta$. 
 }  
 	Also, for any $\theta \in \Theta$,
 	\bea\label{ex2Y}
 	\bbY_n(\theta)
 	&=&\frac{1}{n}\big\{\bbH_n(\theta)-\bbH_n(\theta^*)\big\}\nonumber\\
 	&=&-\frac{1}{2n }\sum_{i=1}^n\bigg\{\big(S_{i-1}^{-1}( \theta)-S_{i-1}^{-1}( \theta^*)\big)\big[h^{-1}(\sigma_{i-1}\Delta_iw)^{\otimes2}\big]+\log\frac{|S_{i-1}(\theta)|}{|S_{i-1}(\theta^*)|}\bigg\}+o_P(1).
 	\eea
 	Then from the following Lemmas \ref{ex2lemDelta}-\ref{ex2lemY}, [{\bf B1}] and [{\bf B2}] hold, {where $\caln = \psi\big(\{A \in \cals_+^\sfm; \det(A)>0, \|A-A^*\| < \delta\}\big)$ for $\delta>0$ in (\ref{calnIto}).}  {\colorr Note that $\caln$ satisfies Condition (i) and (ii) in Section \ref{Quasisection1.1}, and $\caln \subset \Theta$.} From [{\bf I5}],  [{\bf B3}] also holds. From Example \ref{Expositivesemidefinite}, we obtain [{\bf A3}].
 	Also,  [{\bf A4}] obviously holds. Thus, (\ref{ex2hatucon}) holds.
 \end{proof}
 \begin{lemma}\label{ex2lemDelta}
 	\beas
 	\frac{1}{\sqrt n}\partial_\theta\bbH_n(\theta^*) \overset{d_s(\calf)}{\to}  \Delta(\theta^*).
 	\eeas
 \end{lemma}
 \begin{proof}
 	From (\ref{ex2partial1}),
 	$\frac{1}{\sqrt n}\partial_\theta\bbH_n(\theta^*)-M^n=o_P(1)$, where $M^n$ is an $\bbR^\sfp$-valued random variable as for any $u \in \bbR^\sfp$,
 	\beas
 	M^n[u]
 	&=&-\frac{1}{2\sqrt{n} }\sum_{i=1}^n{\rm Tr}\bigg(\sigma_{i-1}^\prime\partial_\theta S_{i-1}^{-1}(\theta^*)\sigma_{i-1}\big\{h^{-1}(\Delta_iw)^{\otimes2}-I_\sfr\big\}[u]\bigg)\\
 	&=&-\frac{1}{2\sqrt{n} }\sum_{i=1}^nv^{\prime}\big(\sigma_{i-1}^\prime\partial_\theta S_{i-1}^{-1}(\theta^*)\sigma_{i-1}[u]\big)\,v\big(h^{-1}(\Delta_iw)^{\otimes2}-I_\sfr\big).
 	\eeas 			
 	{Define $\bbR^\sfp$-valued random variables $\chi_{i}=\chi_i(n)$ as for any $u \in \bbR^\sfp$,
 		\beas
 		\chi_i[u]\yeq -\frac{1}{2\sqrt{n} }v^{\prime}\big(\sigma_{i-1}^\prime\partial_\theta S_{i-1}^{-1}(\theta^*)\sigma_{i-1}[u]\big)\,v\big(h^{-1}(\Delta_iw)^{\otimes2}-I_\sfr\big).
 		\eeas
 		 Then for any $t \in [0, T]$,
 	\beas
 	&&\sum_{i=1}^{\lfloor nt/T\rfloor} E\big[(\chi_i)^{\otimes2}|\calf_{t_{i-1}}\big] \overset{P}{\to}\frac{1}{T}\int_0^t \rho_s^{\otimes2}ds,\qquad\sum_{i=1}^{\lfloor nt/T\rfloor} E\big[\chi_i \Delta_i w|\calf_{t_{i-1}}\big] \overset{P}{\to}0,\\
 	&&\sum_{i=1}^{n} E\big[|\chi_i|^4|\calf_{t_{i-1}}\big] \overset{P}{\to}0,\qquad\sum_{i=1}^{\lfloor nt/T\rfloor} E\big[\chi_i \Delta_i N|\calf_{t_{i-1}}\big] \overset{P}{\to}0\qquad\big(N \in \calm_b(w^{\perp})\big),
 	\eeas where $\Delta_iN=N_{t_{i}}-N_{t_{i-1}}$, and $\calm_b(w^{\perp})$ denotes the class of all bounded $\F$-martingales which is orthogonal to $w$.
 Thus, we can apply Theorem 3-2 of Jacod \cite{jacod1997continuous}.}
 \end{proof}
 \begin{lemma}\label{ex2lemGamma}
 	For any positive sequence $\delta_n$ with $\delta_n \to 0$,
 	\beas
 	\sup_{\theta \in \Theta, |\theta-\theta^*|<\delta_n}\bigg|\frac{1}{n}\partial_\theta^2\bbH_n(\theta)+\Gamma(\theta^*)\bigg| \overset{P}{\to} 0.
 	\eeas
 \end{lemma}
 \begin{proof}
 	From (\ref{ex2partial2}), 
 	\beas
 	&&\frac{1}{n}\partial_\theta^2\bbH_n(\theta)[u^{\otimes2}]\nonumber\\&=&  -\frac{1}{2n }\sum_{i=1}^n\bigg\{{\rm Tr}\bigg(\partial^2_\theta S_{i-1}^{-1}( \theta)S_{i-1}(\theta^*)\big[u^{\otimes2}\big]\bigg)+\partial^2_\theta\log|S_{i-1}|(\theta)[u^{\otimes2}]\bigg\}+o_P(1)|u|^2\\
 	&=&  -G(\theta)[u^{\otimes2}]+o_P(1)|u|^2,
 	\eeas  where $o_P(1)$ satisfies $\sup_{\theta \in \Theta}|o_P(1)| \to^P 0$ from Sobolev's inequalities, and $G(\theta)$ is defined as
 	\beas
 	G(\theta)[u^{\otimes2}]&=&\frac{1}{2 }\int_0^T\bigg\{{\rm Tr}\bigg(\partial^2_\theta S^{-1}(X_t,  \theta)S(X_t, \theta^*)\big[u^{\otimes2}\big]\bigg)+\partial^2_\theta\log|S|(X_t, \theta)[u^{\otimes2}]\bigg\}dt.
 	\eeas
 	Then $G(\cdot)$ is almost surely continuous on $\Theta$, and $G(\theta^*)=\Gamma(\theta^*)$. Thus we obtain the desired result.
 	
 \end{proof}
 
 \begin{lemma}\label{ex2lemY}
 	\beas
 	\sup_{\theta \in \Theta}\bigg|\bbY_n(\theta)-\bbY(\theta)\bigg| \overset{P}{\to} 0.
 	\eeas
 \end{lemma}
 \begin{proof}
 	From (\ref{ex2Y}), for any $\theta \in \Theta$,
 	\beas
 	\bbY_n(\theta)
 	&=&-\frac{1}{2n}\sum_{i=1}^n\bigg\{{\rm Tr}\bigg(\big(S_{i-1}^{-1}( \theta)-S_{i-1}^{-1}( \theta^*)\big)S_{i-1}(\theta^*)\bigg)+\log\frac{|S_{i-1}(\theta)|}{|S_{i-1}(\theta^*)|}\bigg\}+o_P(1)\\
 	&=& \bbY(\theta)+o_P(1).
 	\eeas where $o_P(1)$ satisfies $\sup_{\theta \in \Theta}|o_P(1)| \to^P 0$  from Sobolev's inequalities.
 	
 \end{proof}
 
\begin{proof}[Proof of Theorem \ref{SDEtaudisappear}]
	Redefine $\bbH_n$ as $\bbH_n(\theta, \tau) = \widetilde{\mathbb{\Psi}}_n (\theta, \tau)$.
	As (\ref{SDEsame}), from [{\bf I1}]-[{\bf I4}] and [{\bf I6}],
	\beas-\frac{1}{2\sqrt{n} h}\sum_{i=1}^n\partial_\theta S_{i-1}^{-1}( \theta^*)\big[\big(\Delta_iY-hg(X_{t_{i-1}}, \tau))^{\otimes2}\big]
	&=&-\frac{1}{2 }
	\sum_{i=1}^n\partial_\theta S_{i-1}^{-1}( \theta^*)\bigg[\bigg(\frac{\sigma_{i-1}\Delta_iw}{\sqrt h}\bigg)^{\otimes2}\bigg]+o_P(1),
	\eeas where $o_P(1)$ satisfies $\sup_{\tau \in \calt}|o_P(1)| \to^P 0$ from Sobolev's inequalities.
	Therefore, $\frac{1}{\sqrt n}\partial_\theta\bbH_n(\theta^*, \tau)$ satisfies (\ref{ex2partial1}). Similarly, (\ref{ex2partial2}) and (\ref{ex2Y}) hold for this $\bbH_n(\theta, \tau)$. Then use Lemmas \ref{ex2lemDelta}-\ref{ex2lemY} as Theorem \ref{ex2thm}, and we obtain (\ref{ex2hatu}) from Theorem \ref{thmB}. Also, since [{\bf A5}] holds for $\bbV_n$ and $\hat v_n$, (\ref{ex2hatv_n})  holds from Theorem \ref{thmB}.
\end{proof}

\bibliographystyle{spmpsci} 
\bibliography{bibtex-NonregularCondition}

\end{document}